\documentclass{article}%
\usepackage{amsfonts}
\usepackage{graphicx}
\usepackage{amsmath}
\usepackage{amssymb}%
\providecommand{\U}[1]{\protect\rule{.1in}{.1in}}
\numberwithin{equation}{section}
\providecommand{\U}[1]{\protect\rule{.1in}{.1in}}
\newtheorem{theorem}{Theorem}

\newtheorem{corollary}[theorem]{Corollary}

\newtheorem{definition}[theorem]{Definition}

\newtheorem{lemma}[theorem]{Lemma}

\newtheorem{proposition}[theorem]{Proposition}
\newtheorem{remark}[theorem]{Remark}

\newenvironment{proof}[1][Proof]{\textbf{#1.} }{\ \rule{0.5em}{0.5em}}
\begin{document}
\bigskip

\bigskip

\begin{center}
\bigskip

\bigskip

{\Large Measure valued solutions of the 2D Keller-Segel system.}

\bigskip

\bigskip

S. Luckhaus\footnote{Department of Mathematics and Computer Science,
Universit\"{a}t Leipzig, Leipzig D-04109, Germany.}, Y.
Sugiyama\footnote{Department of Mathematics. Tsuda University, Tokyo 187-8577,
Japan.}, J. J. L. Vel\'{a}zquez\footnote{ICMAT (CSIC-UAM-UC3M-UCM), Facultad
de Matem\'{a}ticas, Universidad Complutense, Madrid 28040, Spain.}

\bigskip

\bigskip
\end{center}

\section{Introduction}

In this paper we study the solutions of the following two-dimensional
Keller-Segel system describing chemotaxis:
\begin{align}
\partial_{t}u-\Delta u+\nabla\left(  u\nabla v\right)   &  =0\;\;\text{in\ \ }%
\Omega\;,\;\left.  \partial_{\nu}u\right|  _{\partial\Omega}=0\label{S1E1}\\
-\Delta v  &  =u-\frac{1}{\left|  \Omega\right|  }\int udx\;\;\text{in\ \ }%
\Omega\;,\;\left.  \partial_{\nu}v\right|  _{\partial\Omega}=0\label{S1E2}\\
u\left(  x,0\right)   &  =u_{0}\left(  x\right)  \;\;\text{in\ \ }%
\Omega\label{S1E3}%
\end{align}
where $\Omega\subset\mathbb{R}^{2}$ is a bounded domain with boundary
$\partial\Omega\in C^{4}$ and $u_{0}$ is a bounded, nonnegative function.

It is well known that the solutions of (\ref{S1E1})-(\ref{S1E3}) blow-up in
finite time, i.e.
\[
\overline{\lim_{t\rightarrow T}}\int_{\Omega}u^{p}dx=\infty\;\;\text{for all
}p>1
\]
for some $T<\infty$ (cf. \cite{JL}).

The Keller-Segel system as well as the properties of the blow-up set has been
extensively studied. An idea that was introduced in \cite{Nagai} to prove
discreteness of the blow-up set is the symmetrization of the nonlinear term in
(\ref{S1E1}) in the equation that describes the evolution of the mass of $u.$
The symmetrization idea has been used in a more general form in
\cite{SenbaSuzuki} to show that the solutions of a system analogous to
(\ref{S1E1})-(\ref{S1E3}) but with (\ref{S1E2}) replaced by
\[
-\Delta v+\gamma v=u\;\;\text{in}\ \ \Omega\;,\;\left.  \partial_{\nu
}v\right\vert _{\partial\Omega}=0
\]
blow-up in a finite set of points. The method used in \cite{SenbaSuzuki}
relies heavily in the symmetry properties of the operator $u\nabla v$. Similar
ideas to the ones in \cite{SenbaSuzuki} can be applied to prove discreteness
of the blow-up set for the solutions of (\ref{S1E1})-(\ref{S1E3}).\bigskip

Continuation beyond blow-up has been considered from several points of view.
The usual approach used to extend the solutions of Keller-Segel systems beyond
the blow-up time consists in regularizing some of the nonlinearities in the
equations by means of a sequence of problems depending on a parameter
$\varepsilon>0.\;$The regularization is chosen in order to obtain a problem
with global solutions in time and also to recover formally the original
Keller-Segel system as the parameter $\varepsilon\rightarrow0$. The papers
\cite{V1}, \cite{V2} study in detail one of these regularizations using
matched asymptotics. In particular, it was obtained in those papers that
formal limits of solutions of the system (\ref{S1E1})-(\ref{S1E3}) with
$\Omega=\mathbb{R}^{2}$ can be described by a set of Dirac measures whose
positions and masses evolve according to a system of ODEs. A different
regularization was considered in the papers \cite{Poupaud}, \cite{DS} where it
was introduced a concept of weak solutions for systems analogous to
(\ref{S1E1})-(\ref{S1E3}) as the limit of regularized problems, different from
the ones considered in \cite{V1}, \cite{V2}. A key idea in \cite{Poupaud},
\cite{DS} is the use of a symmetrization procedure for the nonlinear term
similar to the one in \cite{SenbaSuzuki}. It was also seen in \cite{DS} that
assuming that the measures were Dirac masses concentrated in smoothly moving
curves, the masses and positions of the Dirac masses would evolve according to
a system of ODEs, closely related to the one obtained in \cite{V1}, but
exhibiting some differences due to the different choice of regularization used.

In this paper we obtain generalized solutions for (\ref{S1E1})-(\ref{S1E3}) in
the sense of measures as the limit of two-different regularizations of it. We
will show that the resulting limit measures, that are in some suitable sense
global weak solutions of (\ref{S1E1})-(\ref{S1E3}), depend in the regularization.

\bigskip

The plan of the paper is the following. In Section 2 we introduce two
different regularizations of the Keller-Segel system. Section 3 contains some
properties of the fundamental solution for the Laplace equation in bounded
domains that will be used throughout the whole paper. In Section 4 we describe
a key argument that allows to control the local change of mass in a given
region. Section 5 contains an estimate for the solutions of the second
regularization obtained using an entropy inequality. Local regularity
estimates for the solutions of both regularized problems in the regions where
the mass is small are obtained in Section 6. Section 7 describes how to obtain
limit measure solutions for both limit problems, as well as the fact that such
measures consist in a finite number of atoms plus a regular part. Section 8
describes the limit problems satisfied by such measures. Section 9 proves that
both regularization yield different limit measures for masses above the
critical value. Finally, Section 10 contains a formalism to describe the form
of the nonlinear terms arising in the limit weak formulation using measured
valued Young measures, since some fast oscillations could take place near the singularities.

\section{Two regularizations of (\ref{S1E1})-(\ref{S1E3}).}

\bigskip

We will use in this paper two regularizations of the system (\ref{S1E1}%
)-(\ref{S1E3}). The first one is:
\begin{align}
\partial_{t}u-\Delta u+\nabla\left(  f_{\varepsilon}\left(  u\right)  \nabla
v\right)   &  =0\;\;\text{in\ \ }\Omega\;\;,\;\;\left.  \partial_{\nu
}u\right|  _{\partial\Omega}=0\label{S2E1}\\
-\Delta v  &  =f_{\varepsilon}\left(  u\right)  -\frac{1}{\left|
\Omega\right|  }\int f_{\varepsilon}\left(  u\right)  dx\;\;\text{in\ \ }%
\Omega\;,\;\left.  \partial_{\nu}v\right|  _{\partial\Omega}=0 \label{S2E2}%
\end{align}
where:
\begin{equation}
f_{\varepsilon}\left(  u\right)  =\int_{0}^{u}\min\left\{  1,\left(  \frac
{1}{\varepsilon}-s\right)  _{+}\right\}  ds\;\; \label{S2E3}%
\end{equation}

The second regularization that we will use is:
\begin{align}
\partial_{t}u-\Delta\left(  u+\varepsilon u^{\frac{7}{6}}\right)
+\nabla\left(  u\nabla v\right)   &  =0\;\;\text{in\ \ }\Omega\;\;,\;\;\left.
\partial_{\nu}u\right|  _{\partial\Omega}=0\label{S2E4}\\
-\Delta v  &  =u-\frac{1}{\left|  \Omega\right|  }\int udx\;\;\text{in\ \ }%
\Omega\;,\;\left.  \partial_{\nu}v\right|  _{\partial\Omega}=0 \label{S2E5}%
\end{align}

The choice of the exponent $\frac{7}{6}$ in (\ref{S2E4}) is not essential. The
arguments could be made in a similar manner for any number greater than one.
However, some computations in Section \ref{RegL2} will become slightly simpler
with the particular exponent $\frac{7}{6}.$

We assume in both cases $\varepsilon>0.$ It would be possible to use in
(\ref{S2E1}), (\ref{S2E2}) $f_{\varepsilon}\left(  u\right)  =\frac
{u}{1+\varepsilon u}$ or similar cutoff functions. A key feature of these
regularizations is the symmetry of the nonlinear terms ($f_{\varepsilon
}\left(  u\right)  \nabla v,\;u\nabla v$ respectively) on the function $u.$
This restriction is needed, because the idea\ used in \cite{SenbaSuzuki} to
control the motion of the mass relies heavily on these symmetry properties.

\bigskip

It is trivially seen that the classical solution of the problems (\ref{S2E1}),
(\ref{S2E2}) or (\ref{S2E4}), (\ref{S2E5}) with initial data $u\left(
x,0\right)  =u_{0}\left(  x\right)  $ is globally defined in time for any
$\varepsilon>0.$ In particular, $\left\Vert u\right\Vert _{L^{\infty}\left(
\Omega\right)  }$ is bounded in any interval $0\leq t\leq T<\infty,$ although
the resulting estimate depends on $\varepsilon$ and it can be expected to
blow-up as $\varepsilon\rightarrow0^{+}.$ The choice of boundary conditions
imply:
\begin{equation}
\int_{\Omega}u\left(  x,t\right)  dx=\int_{\Omega}u_{0}\left(  x\right)
dx\ \ . \label{S2E6}%
\end{equation}
The rest of the paper is devoted to characterize the limit of the solutions of
these two problems as $\varepsilon\rightarrow0.$

\bigskip

\section{A local approximation of the Green's function for the Laplace
operator with Neumann boundary conditions.}

We will use the a detailed description of the Green's function associated to
the Laplace equation with Neumann boundary conditions near the
boundary.\bigskip

The following Lemma collects some basic geometrical results. The proof is
elementary and it will be omitted.

\begin{lemma}
\label{Lgeom}Suppose that $\Omega\subset\mathbb{R}^{2}$ is an open set with
$\partial\Omega\in C^{4}.$ Let us denote as $d\left(  y\right)
=\operatorname*{dist}\left(  y,\partial\Omega\right)  $ the distance of $y$ to
$\partial\Omega.$ There exists $\sigma_{0}>0$ depending only on $\partial
\Omega$ such that the function $d\left(  y\right)  $ is uniquely defined and
it has two continuous derivatives in the set $\mathcal{U}=\left\{  y\in
\Omega:\operatorname*{dist}\left(  y,\partial\Omega\right)  \leq2\sigma
_{0}\right\}  $. For any $D\in\left[  0,2\sigma_{0}\right]  $ we define the
curves $\Gamma_{D}=\left\{  y\in\mathbb{R}^{2};\ d\left(  y\right)
=D\right\}  .$ For any $y\in\Gamma_{D}$ the vector $\nu\left(  y\right)
=-\nabla d\left(  y\right)  $ is the normal unit vector to the curve at the
point $y$ and $h\left(  y\right)  =\nabla\cdot\left(  \nu\left(  y\right)
\right)  $ is the curvature of $\Gamma_{D}$ at $y$. Moreover, suppose that we
denote as $t\left(  y\right)  $ the unit tangent vector to $\Gamma_{D}$ at
$y.$ Then:
\begin{equation}
\nabla\nu\left(  y\right)  =h\left(  y\right)  t\left(  y\right)  \otimes
t\left(  y\right)  \ . \label{G1E1f}%
\end{equation}

\end{lemma}

The following general property of the Green's function of the Laplace equation
with Neumann boundary conditions will be useful:

\begin{lemma}
\label{Lsymm}Suppose that $G\left(  y,x\right)  $ is the unique solution of:
\begin{align}
-\Delta_{y}G\left(  y,x\right)   &  =\delta_{x}\left(  y\right)  -\frac
{1}{\left|  \Omega\right|  }\;\;,\;\;y\in\Omega\;\;,\;\;x\in\Omega
\label{N1E1}\\
\frac{\partial G}{\partial\nu_{y}}\left(  y,x\right)   &  =0\;\;,\;\;y\in
\partial\Omega\label{N1E2}\\
\int_{\Omega}G\left(  y,x\right)  dy  &  =0\;\;,\;\;x\in\Omega\label{N1E3}%
\end{align}
where $\nu_{y}$ denotes the outer normal to $\partial\Omega$ at $y\in
\partial\Omega.$ Suppose that $x,x_{0}\in\Omega,\;x\neq x_{0}.$ Then:
\[
G\left(  x,x_{0}\right)  =G\left(  x_{0},x\right)
\]

\end{lemma}

\begin{proof}
Multiplying (\ref{N1E1})-(\ref{N1E3}) by $G\left(  y,x_{0}\right)  ,$ and
integrating by parts we obtain:
\[
\int_{\Omega}\nabla_{y}G\left(  y,x_{0}\right)  \nabla_{y}G\left(  y,x\right)
dy=G\left(  x,x_{0}\right)
\]

Exchanging the role of $x,x_{0}$ the result follows from the symmetry of the
left-hand side.
\end{proof}

\subsection{Uniform regularity estimates near the boundary.}

We now describe the above mentioned Green's function. The main content of the
next lemma is the uniform continuity of the remainder function $K\left(
y,x\right)  .$

\begin{lemma}
\label{L1} Suppose that $G\left(  y,x\right)  $ is as in Lemma \ref{Lsymm}.
Let $\sigma_{0}$ be as in Lemma \ref{Lgeom} and let $Z\in C^{\infty}\left(
\bar{\Omega}\right)  $ be a cutoff function satisfying with $Z\left(
y\right)  =1$ for $d\left(  y\right)  \leq\sigma_{0},$ $Z\left(  y\right)  =0$
for $d\left(  y\right)  \geq2\sigma_{0}.$

Then, there exists $K\in C\left(  \bar{\Omega}\times\bar{\Omega}\right)  $
with $\nabla_{y}K,\nabla_{x}K\in C\left(  \bar{\Omega}\times\bar{\Omega
}\right)  $ such that:
\begin{equation}
G\left(  y,x\right)  =-\frac{1}{2\pi}\left[  \log\left(  \left|  y-x\right|
\right)  +Z\left(  y\right)  \log\left(  \left|  \tau\left(  y\right)
-x\right|  \right)  \right]  +K\left(  y,x\right)  \label{M1E1}%
\end{equation}
where
\[
\tau\left(  y\right)  =y+f\left(  y\right)  \nu\left(  y\right)  ,\;f\left(
y\right)  =2d\left(  y\right)  +h\left(  y\right)  \left(  d\left(  y\right)
\right)  ^{2}.
\]

\end{lemma}

\begin{proof}
Existence and uniqueness of the function $G$ is standard (cf. \cite{TS}).
Since the result is inmediate for $d\left(  y\right)  \geq\sigma_{0}$ we
restrict our analysis to the case $d\left(  y\right)  \leq\sigma_{0}.$ We
define a function
\[
G^{\ast}\left(  y,x\right)  =-\frac{1}{2\pi}\left[  \log\left(  \left\vert
y-x\right\vert \right)  +\log\left(  \left\vert \tau\left(  y\right)
-x\right\vert \right)  \right]  \ \ .
\]

Notice that $G^{\ast}$ satisfies:
\begin{equation}
\frac{\partial G^{\ast}}{\partial\nu_{y}}=0\ \ ,\ \ y\in\partial
\Omega\label{G1E1}%
\end{equation}
as well as:
\begin{equation}
-\Delta_{y}G^{\ast}\left(  y,x\right)  =-\frac{1}{4\pi}\Delta_{y}\left(
\log\left(  \left\vert \tau\left(  y\right)  -x\right\vert ^{2}\right)
\right)  \;\;\text{in\ \ }\Omega\label{G1E1a}%
\end{equation}

In order to compute the right hand side of (\ref{G1E1a}) we write:
\begin{equation}
\Delta_{y}\left(  \log\left(  \left\vert \tau\left(  y\right)  -x\right\vert
^{2}\right)  \right)  =\frac{1}{\left\vert \tau\left(  y\right)  -x\right\vert
^{2}}\Delta_{y}\left(  \left\vert \tau\left(  y\right)  -x\right\vert
^{2}\right)  -\frac{1}{\left\vert \tau\left(  y\right)  -x\right\vert ^{4}%
}\left(  \nabla_{y}\left(  \left\vert \tau\left(  y\right)  -x\right\vert
^{2}\right)  \right)  ^{2}\ \ . \label{G1E1b}%
\end{equation}

Using Lemma \ref{Lgeom} we obtain:
\begin{equation}
\nabla_{y}\left(  \left\vert \tau\left(  y\right)  -x\right\vert ^{2}\right)
=2\left(  \tau\left(  y\right)  -x\right)  +2\left[  \nu\left(  y\right)
\cdot\left(  \tau\left(  y\right)  -x\right)  \right]  \nabla_{y}f+2hf\left[
t\left(  y\right)  \cdot\left(  \tau\left(  y\right)  -x\right)  \right]
t\left(  y\right)  \ \ . \label{G1E1c}%
\end{equation}
We will use also the equivalent formula:
\begin{align}
\nabla_{y}\left(  \left\vert \tau\left(  y\right)  -x\right\vert ^{2}\right)
&  =2\left[  \left(  \tau\left(  y\right)  -x\right)  -2\left[  \nu\left(
y\right)  \cdot\left(  \tau\left(  y\right)  -x\right)  \right]  \nu\left(
y\right)  \right]  +\label{G1E1H}\\
&  +2\left[  \nu\left(  y\right)  \cdot\left(  \tau\left(  y\right)
-x\right)  \right]  \left[  \nabla_{y}f+2\nu\left(  y\right)  \right]
+2hf\left[  t\left(  y\right)  \cdot\left(  \tau\left(  y\right)  -x\right)
\right]  t\left(  y\right)  \ .\nonumber
\end{align}

Applying the operator $\nabla\cdot$ to (\ref{G1E1c}), using that $d\leq
C\left\vert \tau\left(  y\right)  -x\right\vert ,\;f\left(  d\right)  \leq Cd$
where $C$ is a constant depending only on the regularity bounds of $\Omega
$\ and neglecting terms that are smaller than $\left\vert \tau\left(
y\right)  -x\right\vert ^{2}$ we obtain:
\begin{align}
\Delta_{y}\left(  \left\vert \tau\left(  y\right)  -x\right\vert ^{2}\right)
&  =2\nabla_{y}\cdot\left(  \tau\left(  y\right)  \right)  +2\left[
\nu\left(  y\right)  \cdot\left(  \tau\left(  y\right)  -x\right)  \right]
\Delta_{y}f+2\left(  \nabla_{y}\left[  \nu\left(  y\right)  \cdot\left(
\tau\left(  y\right)  -x\right)  \right]  \right)  \cdot\nabla_{y}%
f+\nonumber\\
&  +2\left[  t\left(  y\right)  \cdot\left(  \tau\left(  y\right)  -x\right)
\right]  \left(  \nabla_{y}\left(  hf\right)  \cdot t\left(  y\right)
\right)  +2hf\left(  \nabla_{y}\left[  t\left(  y\right)  \cdot\left(
\tau\left(  y\right)  -x\right)  \right]  \cdot t\left(  y\right)  \right)
+\nonumber\\
&  +O\left(  \left\vert \tau\left(  y\right)  -x\right\vert ^{2}\right)  \ .
\label{G1E1d}%
\end{align}

The first term on the right-hand side of (\ref{G1E1d}) can be approximated,
using the definition of $f,$ as:
\begin{equation}
2\nabla_{y}\cdot\left(  \tau\left(  y\right)  \right)  =2fh-4dh+O\left(
d^{2}\right)  =O\left(  d^{2}\right)  \ . \label{Gsup1}%
\end{equation}

The second term on the right-hand side of (\ref{G1E1d}) satisfies:
\begin{align}
2\left[  \nu\left(  y\right)  \cdot\left(  \tau\left(  y\right)  -x\right)
\right]  \Delta_{y}f  &  =2\left[  \nu\left(  y\right)  \cdot\left(
\tau\left(  y\right)  -x\right)  \right]  \left[  -2\nabla_{y}\cdot\left(
\nu\left(  y\right)  \right)  +2h\left(  \nabla_{y}d\left(  y\right)  \right)
^{2}+O\left(  d\right)  \right] \nonumber\\
&  =O\left(  \left\vert \tau\left(  y\right)  -x\right\vert ^{2}\right)  \ .
\label{Gsup2}%
\end{align}

The third term on the right-hand side of (\ref{G1E1d}) can be written, using
Lemma \ref{Lgeom} as:
\begin{equation}
2\left(  \nabla_{y}\left[  \nu\left(  y\right)  \cdot\left(  \tau\left(
y\right)  -x\right)  \right]  \right)  \cdot\nabla_{y}f=-4-4hd+2\left[
2+2hd\right]  ^{2}+O\left(  d^{2}\right)  =4+12dh+O\left(  d^{2}\right)  \ .
\label{Gsup3}%
\end{equation}

We can estimate the fourth term on the right-hand side of (\ref{G1E1d}), using
the orthogonality of $t\left(  y\right)  $ and $\nu\left(  y\right)  $ as:
\begin{equation}
2\left[  t\left(  y\right)  \cdot\left(  \tau\left(  y\right)  -x\right)
\right]  \left(  \nabla_{y}\left(  hf\right)  \cdot t\left(  y\right)
\right)  =O\left(  \left\vert \tau\left(  y\right)  -x\right\vert ^{2}\right)
\ . \label{Gsup4}%
\end{equation}

Moreover, using (\ref{G1E1f}) and the orthogonality of $t\left(  y\right)  $
and $\nu\left(  y\right)  $ we estimate the fifth term on the right-hand side
of (\ref{G1E1d}) as:
\begin{equation}
2hf\left(  \nabla_{y}\left[  t\left(  y\right)  \cdot\left(  \tau\left(
y\right)  -x\right)  \right]  \cdot t\left(  y\right)  \right)  =2hf+O\left(
\left\vert \tau\left(  y\right)  -x\right\vert ^{2}\right)  =4dh+O\left(
\left\vert \tau\left(  y\right)  -x\right\vert ^{2}\right)  \ . \label{Gsup5}%
\end{equation}

Therefore, combining (\ref{G1E1d})-(\ref{Gsup5}):
\begin{equation}
\Delta_{y}\left(  \left\vert \tau\left(  y\right)  -x\right\vert ^{2}\right)
=4+16dh+O\left(  \left\vert \tau\left(  y\right)  -x\right\vert ^{2}\right)
\ . \label{Q1E1}%
\end{equation}

On the other hand we compute $\left(  \nabla_{y}\left(  \left\vert \tau\left(
y\right)  -x\right\vert ^{2}\right)  \right)  ^{2}.$ To this end we use
(\ref{G1E1H}). Using that $\left[  \nabla_{y}f+2\nu\left(  y\right)  \right]
=O\left(  d\right)  $ as well as the fact that $\left(  I-2\nu\otimes
\nu\right)  $ is an isometry we obtain, after some computations:
\begin{equation}
\left(  \nabla_{y}\left(  \left\vert \tau\left(  y\right)  -x\right\vert
^{2}\right)  \right)  ^{2}=4\left\vert \tau\left(  y\right)  -x\right\vert
^{2}+16dh\left\vert \tau\left(  y\right)  -x\right\vert ^{2}+O\left(
\left\vert \tau\left(  y\right)  -x\right\vert ^{4}\right)  \ . \label{Q1E2}%
\end{equation}

Combining\ (\ref{G1E1b}), (\ref{Q1E1}), (\ref{Q1E2}) we obtain:
\begin{equation}
\left|  \Delta_{y}\left(  \log\left(  \left|  \tau\left(  y\right)  -x\right|
\right)  \right)  \right|  \leq C\;\;\text{in\ \ }\Omega\label{LapEstimate}%
\end{equation}
for some constant $C$ depending only on $\Omega.$

We now derive an estimate for $\nabla_{x}\left(  \Delta_{y}\left(  \log\left(
\left\vert \tau\left(  y\right)  -x\right\vert \right)  \right)  \right)  $
with respect to $x.$ To this end we differentiate (\ref{G1E1b}):
\begin{align}
&  \nabla_{x}\left(  \Delta_{y}\left(  \log\left(  \left\vert \tau\left(
y\right)  -x\right\vert ^{2}\right)  \right)  \right) \nonumber\\
&  =\frac{2\left(  \tau\left(  y\right)  -x\right)  }{\left\vert \tau\left(
y\right)  -x\right\vert ^{4}}\Delta_{y}\left(  \left\vert \tau\left(
y\right)  -x\right\vert ^{2}\right)  -\frac{2}{\left\vert \tau\left(
y\right)  -x\right\vert ^{2}}\Delta_{y}\left(  \tau\left(  y\right)
-x\right)  -\nonumber\\
&  -\frac{4\left(  \tau\left(  y\right)  -x\right)  }{\left\vert \tau\left(
y\right)  -x\right\vert ^{6}}\left(  \nabla_{y}\left(  \left\vert \tau\left(
y\right)  -x\right\vert ^{2}\right)  \right)  ^{2}+\frac{4}{\left\vert
\tau\left(  y\right)  -x\right\vert ^{4}}\nabla_{y}\left(  \left\vert
\tau\left(  y\right)  -x\right\vert ^{2}\right)  \cdot\nabla_{y}\left(
\tau\left(  y\right)  -x\right)  \label{Q1E6bis}%
\end{align}
for $d\left(  y\right)  \leq\sigma_{0}.$ Notice that, using Lemma \ref{Lgeom}
and the definition of $\tau\left(  y\right)  $ we obtain, after some
computations:
\begin{equation}
\Delta_{y}\left(  \tau_{i}\left(  y\right)  -x_{i}\right)  =\nabla_{y}%
\cdot\left(  \nabla_{y}f\nu_{i}\left(  y\right)  +f\nabla_{y}\nu_{i}\left(
y\right)  \right)  =O\left(  d\right)  \label{Q1E4}%
\end{equation}

On the other hand:
\begin{equation}
\nabla_{y}\left(  \tau_{i}\left(  y\right)  -x_{i}\right)  =\nabla_{y}\left(
\tau_{i}\left(  y\right)  \right)  =e_{i}-\left[  2+2dh\right]  \nu\left(
y\right)  \nu_{i}\left(  y\right)  +fht_{i}\left(  y\right)  t\left(
y\right)  +O\left(  d^{2}\right)  \label{Q1E3}%
\end{equation}
where $e_{i}$ is a unit vector in the direction of the $y_{i}$ axis. Combining
(\ref{G1E1c}) and (\ref{Q1E3}) it then follows that:
\begin{equation}
\nabla_{y}\left(  \left\vert \tau\left(  y\right)  -x\right\vert ^{2}\right)
\cdot\nabla_{y}\left(  \tau\left(  y\right)  -x\right)  =\left[  2+8dh\right]
\left(  \tau\left(  y\right)  -x\right)  +O\left(  \left\vert \tau\left(
y\right)  -x\right\vert ^{3}\right)  \label{Q1E5}%
\end{equation}

It then follows from (\ref{Q1E1}), (\ref{Q1E2}), (\ref{Q1E6bis}),
(\ref{Q1E4}), (\ref{Q1E5}):
\begin{align}
&  \nabla_{x}\left(  \Delta_{y}\left(  \log\left(  \left\vert \tau\left(
y\right)  -x\right\vert \right)  \right)  \right) \nonumber\\
&  =\frac{2\left(  \tau\left(  y\right)  -x\right)  }{\left\vert \tau\left(
y\right)  -x\right\vert ^{4}}\left[  4+16dh\right]  -\frac{4\left(
\tau\left(  y\right)  -x\right)  }{\left\vert \tau\left(  y\right)
-x\right\vert ^{6}}\left[  4\left\vert \tau\left(  y\right)  -x\right\vert
^{2}+16dh\left\vert \tau\left(  y\right)  -x\right\vert ^{2}\right]
+\nonumber\\
&  +\frac{4\left(  \tau\left(  y\right)  -x\right)  }{\left\vert \tau\left(
y\right)  -x\right\vert ^{4}}\left[  2+8dh\right]  +O\left(  \frac
{1}{\left\vert \tau\left(  y\right)  -x\right\vert }\right) \label{Q1E5a}\\
&  =O\left(  \frac{1}{\left\vert \tau\left(  y\right)  -x\right\vert }\right)
\nonumber
\end{align}

Combining (\ref{LapEstimate}), (\ref{Q1E5a}) we obtain:
\begin{align}
\left\vert \Delta_{y}\left(  Z\left(  y\right)  \log\left(  \left\vert
\tau\left(  y\right)  -x\right\vert \right)  \right)  \right\vert  &  \leq
C\label{Q1E7a}\\
\left\vert \nabla_{x}\left(  \Delta_{y}\left(  Z\left(  y\right)  \log\left(
\left\vert \tau\left(  y\right)  -x\right\vert \right)  \right)  \right)
\right\vert  &  \leq C\left[  \frac{Z\left(  y\right)  }{\left\vert
\tau\left(  y\right)  -x\right\vert }+1\right]  \label{Q1E7}%
\end{align}
uniformly on $\left(  x,y\right)  \in\Omega\times\Omega.$

We can now prove the continuity of the function $K\left(  x,y\right)  $
defined in (\ref{M1E1}). Notice that (\ref{N1E1})-(\ref{N1E3}), (\ref{Q1E7a}),
(\ref{Q1E7}) imply:
\begin{align}
-\Delta_{y}\left(  K\left(  y,x\right)  \right)   &  =R_{1}\left(  y,x\right)
\;\;,\;\;\left\vert R_{1}\left(  y,x\right)  \right\vert \leq C\nonumber\\
\frac{\partial K}{\partial\nu_{y}}\left(  y,x\right)   &  =0\;\;,\;\;y\in
\partial\Omega\ \ ,\ \ x\in\Omega\nonumber\\
\left\vert \int_{\Omega}K\left(  y,x\right)  dy\right\vert  &  \leq
C_{1}\;\;,\;\;x\in\Omega\label{Q1E7b}%
\end{align}
with $C_{1}$ independent on $x.$%
\begin{equation}
-\Delta_{y}\left(  \nabla_{x}K\left(  y,x\right)  \right)  =R_{2}\left(
y,x\right)  \text{\ \ ,\ \ }\left\vert R_{2}\left(  y,x\right)  \right\vert
\leq C\left[  \frac{Z\left(  y\right)  }{\left\vert \tau\left(  y\right)
-x\right\vert }+1\right]  \label{Q1E8}%
\end{equation}%
\begin{align*}
\frac{\partial\left(  \nabla_{x}K\right)  }{\partial\nu_{y}}\left(
y,x\right)   &  =0\;\;,\;\;y\in\partial\Omega\ \ ,\ \ x\in\Omega\\
\left\vert \int_{\Omega}\nabla_{x}K\left(  y,x\right)  dy\right\vert  &  \leq
C_{1}\;\;,\;\;x\in\Omega
\end{align*}
with $C_{1}$ independent on $x.$

Therefore, multiplying (\ref{Q1E8}) by $\nabla_{x}K\left(  y,x\right)  ,$
integrating with respect to the $y$ variable and integrating by parts and
using Sobolev and H\"{o}lder inequalities we obtain:
\[
\int_{\Omega}\left\vert \nabla_{x}K\left(  y,x\right)  \right\vert ^{p}%
dy+\int_{\Omega}\left\vert \nabla_{y}\left(  \nabla_{x}K\left(  y,x\right)
\right)  \right\vert ^{2}dy\leq C\;\;,\;\text{for}\;\text{any }p<\infty
\]
uniformly on $x\in\Omega.\;$On the other hand, multiplying (\ref{Q1E7b}) by
$K\left(  y,x\right)  $ and using a similar argument we obtain:
\[
\int_{\Omega}\left\vert K\left(  y,x\right)  \right\vert ^{p}dy+\int_{\Omega
}\left\vert \nabla_{y}\left(  K\left(  y,x\right)  \right)  \right\vert
^{2}dy\leq C\;\;,\;\text{for}\;\text{any}\;p<\infty
\]
uniformly on $x\in\Omega.$ Classical regularity theory for (\ref{Q1E7b}) then
shows that $K\left(  y,x\right)  ,\;\nabla_{y}K\left(  y,x\right)  $ are
uniformly bounded in $\Omega\times\Omega.$

Since $\frac{1}{\left\vert \tau\left(  \cdot\right)  -x\right\vert }\in
L^{q}\left(  \Omega\right)  $ for any $q<2,$ it follows from classical
regularity theory that $\nabla_{x}K\left(  \cdot,x\right)  \in W^{2,q}\left(
\Omega\right)  $ for any $q<2$ with uniform bounds on $x\in\Omega.$ Therefore,
Sobolev embeddings yield $\nabla_{x}K\left(  \cdot,x\right)  \in W^{1,\sigma
}\left(  \Omega\right)  $ for any $\sigma<\infty$ and then $\nabla_{x}K\left(
\cdot,x\right)  $ is H\"{o}lder in the $y$ variable, uniformly on $x\in
\Omega.$ It remains to prove continuity in the $x$ variable. To this end we
use (\ref{Q1E8}) as well as (\ref{N1E1})-(\ref{N1E3}) to write the
representation formula:
\begin{equation}
\nabla_{x}K\left(  y,x\right)  =\int_{\Omega}G\left(  y,z\right)  R_{2}\left(
z,x\right)  dz=\int_{\Omega}G^{\ast}\left(  y,z\right)  R_{2}\left(
z,x\right)  dz+\int_{\Omega}K\left(  y,z\right)  R_{2}\left(  z,x\right)  dz
\label{Q1E9}%
\end{equation}

Using the inequalities:
\begin{align*}
\left\vert G^{\ast}\left(  y,z\right)  R_{2}\left(  z,x\right)  \right\vert
&  \leq C\left[  \frac{Z\left(  z\right)  }{\left\vert \tau\left(  z\right)
-x\right\vert }+1\right]  \left\vert \log\left(  \left\vert z-y\right\vert
\right)  \right\vert \\
\left\vert K\left(  y,z\right)  R_{2}\left(  z,x\right)  \right\vert  &  \leq
C\left[  \frac{Z\left(  z\right)  }{\left\vert \tau\left(  z\right)
-x\right\vert }+1\right]
\end{align*}
as well as the continuity of the function $R_{2}\left(  x,z\right)  $ with
respect to the $x$ variable for $x\neq z$ it then follows that the right-hand
side of (\ref{Q1E9}) is continuous and the result follows.
\end{proof}

\bigskip

\subsection{A geometric representation formula for $\nabla_{x}G.$}

We write the Green's function $\nabla_{x}G\left(  y,x\right)  $ in a more
convenient form. To this end we introduce the following notation to denote the
closest point to $x$ at the boundary $\partial\Omega.$
\begin{equation}
P_{\partial}\left(  x\right)  =x+d\left(  x\right)  \nu\left(  x\right)
\;\;,\;\;\operatorname*{dist}\left(  x,\partial\Omega\right)  \leq\sigma_{0}
\label{M2E0}%
\end{equation}

The next lemma provides a suitable approximation for $\nabla_{x}G$ near
$\partial\Omega.$

\bigskip

\begin{lemma}
\label{LRep}Assume that $G$ is as in Lemma \ref{L1}. Then we can write:
\begin{align}
&  \nabla_{x}G\left(  y,x\right) \nonumber\\
&  =-\frac{1}{2\pi}\frac{\left(  x-y\right)  }{\left\vert x-y\right\vert ^{2}%
}-\frac{Z\left(  y\right)  }{2\pi}\frac{P_{\partial}\left(  x\right)
-P_{\partial}\left(  y\right)  -\left[  d\left(  x\right)  \nu\left(
x\right)  +d\left(  y\right)  \nu\left(  y\right)  \right]  }{D}-\nonumber\\
&  -\frac{Z\left(  y\right)  h\left(  y\right)  }{2\pi}\left[  \mathcal{G}%
_{t}\left(  Y\left(  x,y\right)  ,\lambda_{1}\left(  x,y\right)  ,\lambda
_{2}\left(  x,y\right)  \right)  +\mathit{g}_{n}\left(  Y\left(  x,y\right)
,\lambda_{1}\left(  x,y\right)  ,\lambda_{2}\left(  x,y\right)  \right)
\nu\left(  y\right)  \right]  +W\left(  x,y\right)  \label{M2E3}%
\end{align}
where the operator $P_{\partial}$ is defined in (\ref{M2E0}), $d\left(
\cdot\right)  $ is as in Lemma \ref{Lgeom},$\ W\left(  x,y\right)  $ is
continuous in $\bar{\Omega}\times\bar{\Omega}$, and the functions
$\mathcal{G}_{t},\;\mathit{g}_{n},\;Y,\;\lambda_{1},\;\lambda_{2}$ are given
as:
\begin{align}
\mathcal{G}_{t}\left(  Y,\lambda_{1},\lambda_{2}\right)   &  =-2\left(
\lambda_{1}+\lambda_{2}\right)  \lambda_{2}^{2}Y+\left(  \lambda_{1}%
-\lambda_{2}\right)  \left\vert Y\right\vert ^{2}Y\label{M2E3a}\\
\mathit{g}_{n}\left(  Y,\lambda_{1},\lambda_{2}\right)   &  =\left[
-\lambda_{2}^{2}+2\lambda_{2}^{2}\left(  \lambda_{1}+\lambda_{2}\right)
^{2}+\left(  \lambda_{2}^{2}-\lambda_{1}^{2}\right)  \left\vert Y\right\vert
^{2}\right] \label{M2E3b}\\
D  &  =\left\vert P_{\partial}\left(  x\right)  -P_{\partial}\left(  y\right)
\right\vert ^{2}+\left(  d\left(  x\right)  +d\left(  y\right)  \right)  ^{2}
\label{M2E3c}%
\end{align}%
\begin{align*}
Y\left(  x,y\right)   &  =\frac{P_{\partial}\left(  x\right)  -P_{\partial
}\left(  y\right)  }{\sqrt{\left\vert P_{\partial}\left(  x\right)
-P_{\partial}\left(  y\right)  \right\vert ^{2}+\left(  d\left(  y\right)
+d\left(  y\right)  \right)  ^{2}}}\;\\
\lambda_{1}\left(  x,y\right)   &  =\frac{d\left(  x\right)  }{\sqrt
{\left\vert P_{\partial}\left(  x\right)  -P_{\partial}\left(  y\right)
\right\vert ^{2}+\left(  d\left(  y\right)  +d\left(  y\right)  \right)  ^{2}%
}}\;\\
\lambda_{2}\left(  x,y\right)   &  =\frac{d\left(  y\right)  }{\sqrt
{\left\vert P_{\partial}\left(  x\right)  -P_{\partial}\left(  y\right)
\right\vert ^{2}+\left(  d\left(  y\right)  +d\left(  y\right)  \right)  ^{2}%
}}\;.
\end{align*}

\end{lemma}

\begin{remark}
The first two terms in (\ref{M2E3}) are homogeneous functions of order $-1.$
The terms $\mathcal{G}_{t},\;\mathit{g}_{n}$\texttt{\ }are homogeneous
functions of order zero that in the limit $\left\vert x-\tau\left(  y\right)
\right\vert \rightarrow0$ yield respectively a tangential component and a
normal component to $\partial\Omega.$
\end{remark}

\begin{proof}
Our goal is to approximate $\nabla_{x}G$ that is given as (cf. (\ref{M1E1})):
\begin{equation}
\nabla_{x}G\left(  y,x\right)  =-\frac{1}{2\pi}\frac{\left(  x-y\right)
}{\left\vert x-y\right\vert ^{2}}-\frac{Z\left(  y\right)  }{2\pi}\frac
{x-\tau\left(  y\right)  }{\left\vert x-\tau\left(  y\right)  \right\vert
^{2}}+\nabla_{x}K\left(  y,x\right)  \ . \label{M2E1}%
\end{equation}

Using (\ref{M2E0}) we obtain:
\begin{equation}
x-\tau\left(  y\right)  =P_{\partial}\left(  x\right)  -P_{\partial}\left(
y\right)  -\left[  d\left(  x\right)  \nu\left(  x\right)  +d\left(  y\right)
\nu\left(  y\right)  \right]  -h\left(  y\right)  \left(  d\left(  y\right)
\right)  ^{2}\nu\left(  y\right)  \ , \label{M2E1aa}%
\end{equation}
whence it follows, after some computations:
\begin{align}
\left\vert x-\tau\left(  y\right)  \right\vert ^{2}  &  =\left\vert
P_{\partial}\left(  x\right)  -P_{\partial}\left(  y\right)  \right\vert
^{2}+\left(  d\left(  x\right)  +d\left(  y\right)  \right)  ^{2}-2\left(
P_{\partial}\left(  x\right)  -P_{\partial}\left(  y\right)  \right)
\cdot\left(  d\left(  x\right)  \nu\left(  x\right)  +d\left(  y\right)
\nu\left(  y\right)  \right)  +\nonumber\\
&  +2h\left(  y\right)  \left(  d\left(  y\right)  \right)  ^{2}\left(
d\left(  x\right)  +d\left(  y\right)  \right)  +O\left(  \left(  d\left(
x\right)  \right)  ^{4}+\left(  d\left(  y\right)  \right)  ^{4}+\left\vert
x-y\right\vert ^{4}\right)  \ . \label{M2E1a}%
\end{align}

Let us define $\ell\left(  x,y\right)  =\left(  P_{\partial}\left(  x\right)
-P_{\partial}\left(  y\right)  \right)  \cdot t\left(  y\right)  .$ Using also
$\nu\left(  y\right)  =\nu\left(  P_{\partial}\left(  y\right)  \right)
,$\ $t\left(  y\right)  =t\left(  P_{\partial}\left(  y\right)  \right)  $ we
obtain:
\[
\left(  P_{\partial}\left(  x\right)  -P_{\partial}\left(  y\right)  \right)
=\ell\left(  x,y\right)  t\left(  y\right)  -\frac{h\left(  P_{\partial
}\left(  y\right)  \right)  \left(  \ell\left(  x,y\right)  \right)  ^{2}}%
{2}\nu\left(  y\right)  +O\left(  \left(  \ell\left(  x,y\right)  \right)
^{3}\right)  \ ,
\]%
\begin{align*}
d\left(  y\right)  \left(  P_{\partial}\left(  x\right)  -P_{\partial}\left(
y\right)  \right)  \cdot\nu\left(  y\right)   &  =-\frac{h\left(  P_{\partial
}\left(  y\right)  \right)  \left(  \ell\left(  x,y\right)  \right)
^{2}d\left(  y\right)  }{2}+O\left(  \left(  \ell\left(  x,y\right)  \right)
^{4}\right)  \ ,\\
d\left(  x\right)  \left(  P_{\partial}\left(  x\right)  -P_{\partial}\left(
y\right)  \right)  \cdot\nu\left(  x\right)   &  =\frac{h\left(  P_{\partial
}\left(  x\right)  \right)  \left(  \ell\left(  x,y\right)  \right)
^{2}d\left(  x\right)  }{2}+O\left(  \left(  \ell\left(  x,y\right)  \right)
^{4}\right)  \ ,\\
\left(  \ell\left(  x,y\right)  \right)  ^{2}  &  =\left\vert P_{\partial
}\left(  x\right)  -P_{\partial}\left(  y\right)  \right\vert ^{2}+O\left(
\left\vert x-y\right\vert ^{3}\right)  \ .
\end{align*}

Then:
\begin{align}
&  2\left(  P_{\partial}\left(  x\right)  -P_{\partial}\left(  y\right)
\right)  \cdot\left(  d\left(  x\right)  \nu\left(  x\right)  +d\left(
y\right)  \nu\left(  y\right)  \right) \label{M2E1b}\\
&  =-h\left(  y\right)  \left(  d\left(  y\right)  -d\left(  x\right)
\right)  \left(  \ell\left(  x,y\right)  \right)  ^{2}+O\left(  \left\vert
x-y\right\vert ^{4}+\left(  d\left(  x\right)  \right)  ^{4}+\left(  d\left(
y\right)  \right)  ^{4}\right)  \ .\nonumber
\end{align}

Plugging (\ref{M2E1b}) into (\ref{M2E1a}) and using Taylor's expansion we
obtain:
\begin{align}
\frac{x-\tau\left(  y\right)  }{\left\vert x-\tau\left(  y\right)  \right\vert
^{2}}  &  =\frac{P_{\partial}\left(  x\right)  -P_{\partial}\left(  y\right)
-\left[  d\left(  x\right)  \nu\left(  x\right)  +d\left(  y\right)
\nu\left(  y\right)  \right]  }{D}+\nonumber\\
&  +\left[  \mathcal{G}_{t}\left(  Y\left(  x,y\right)  ,\lambda_{1}\left(
x,y\right)  ,\lambda_{2}\left(  x,y\right)  \right)  +\mathit{g}_{n}\left(
Y\left(  x,y\right)  ,\lambda_{1}\left(  x,y\right)  ,\lambda_{2}\left(
x,y\right)  \right)  \nu\left(  y\right)  \right]  +\nonumber\\
&  +\tilde{W}\left(  x,y\right)  \ , \label{M2E1c}%
\end{align}
where $\mathcal{G}_{t},\;\mathit{g}_{n}$ are as in (\ref{M2E3a}),
(\ref{M2E3b}) and $\tilde{W}\left(  x,y\right)  $ is a continuous function in
$\bar{\Omega}\times\bar{\Omega}$ satisfying:
\[
\tilde{W}\left(  x,y\right)  =O\left(  d\left(  x\right)  +d\left(  y\right)
+\left\vert x-y\right\vert \right)  \ .
\]

Combining (\ref{M2E1}), (\ref{M2E1c}), (\ref{M2E3}) and Lemma \ref{LRep}
follows with $W\left(  x,y\right)  =-\frac{Z\left(  y\right)  }{2\pi}\tilde
{W}\left(  x,y\right)  +\nabla_{x}K\left(  x,y\right)  .$
\end{proof}

\section{Local mass change estimates.}

In this Section we derive some crucial estimates for the local change of mass
of $u$ for the solutions of (\ref{S2E1}), (\ref{S2E2}) or (\ref{S2E4}),
(\ref{S2E5}). To this end we use the symmetrization argument as introduced in
\cite{SenbaSuzuki} and used also in \cite{DS}, \cite{Poupaud}. We will
consider separately the cases of points $x_{0}$ that are at the interior of
$\Omega$ and the points that are close to the boundary.

\subsection{Interior estimates.}

We will use an auxiliary test function $\varphi\in C^{1,1}\left(
\mathbb{R}^{+}\right)  $ defined as:
\begin{align}
\varphi\left(  r\right)   &  =1-\frac{r^{2}}{2}\;\;,\;\;0\leq r\leq
1\;\;,\;\;\varphi\left(  r\right)  =\frac{1}{2}-\log\left(  r\right)
\;\;,\;\;1\leq r\leq e^{\frac{1}{4}}\nonumber\\
\varphi\left(  r\right)   &  =\frac{1}{e^{\frac{1}{2}}}\left(  \frac
{3e^{\frac{1}{4}}}{2}-r\right)  ^{2}\;\;,\;e^{\frac{1}{4}}\leq r\leq
\frac{3e^{\frac{1}{4}}}{2}\;\;,\;\;\varphi\left(  r\right)  =0\;\;,\;\;r\geq
\frac{3e^{\frac{1}{4}}}{2} \label{S3E2A}%
\end{align}

Given $\rho>0$ and $x_{0}\in\Omega,$ such that $d\left(  x_{0}\right)
\geq\frac{3e^{\frac{1}{4}}\rho}{2}$ we define:
\begin{equation}
\psi_{\rho}\left(  x\right)  =\varphi\left(  \frac{\left|  x-x_{0}\right|
}{\rho}\right)  \label{S3E3a}%
\end{equation}

Notice that (\ref{S3E2A}) implies:
\begin{align}
\Delta\psi_{\rho}\left(  x\right)   &  =-\frac{2}{\rho^{2}}\;\text{for}%
\;\left\vert x-x_{0}\right\vert <\rho\;,\;\Delta\psi_{\rho}\left(  x\right)
=0\;\text{for}\;\rho<\left\vert x-x_{0}\right\vert <e^{\frac{1}{4}}%
\rho\;,\label{S3E3}\\
\;\Delta\psi_{\rho}\left(  x\right)   &  \geq0\;\text{for}\;e^{\frac{1}{4}%
}\rho<\left\vert x-x_{0}\right\vert \nonumber
\end{align}

The letters $\varphi,\;\psi$ will denote generic test functions that will
change along the paper, but will be used consistently in each argument.

We have:

\begin{proposition}
\label{P1}Suppose that $u$ solves one of the problems (\ref{S2E1}),
(\ref{S2E2}) or (\ref{S2E4}), (\ref{S2E5}). Let us fix $\rho>0$ and let us
assume that $\operatorname*{dist}\left(  x_{0},\partial\Omega\right)
\geq2\rho.$ Let $\psi_{\rho}$ be as in (\ref{S3E3a}). Then:
\begin{equation}
\left\vert \partial_{t}\left(  \int_{\Omega}\psi_{\rho}u\right)  \right\vert
\leq\frac{\kappa}{\rho^{2}} \label{S3E1}%
\end{equation}
if $u$ solves (\ref{S2E1}), (\ref{S2E2}),\thinspace and:
\begin{equation}
\partial_{t}\left(  \int_{\Omega}\psi_{\rho}u\right)  \geq-\frac{\kappa}%
{\rho^{2}}-\frac{2\varepsilon}{\rho^{2}}\int_{B_{\rho}\left(  x_{0}\right)
}u^{\frac{7}{6}} \label{S3E2}%
\end{equation}
if $u$ solves (\ref{S2E4}), (\ref{S2E5}). The constant $\kappa$ depends on
$\left\Vert u_{0}\right\Vert _{L^{1}\left(  \Omega\right)  }$, but it is
independent on $\varepsilon$ and $\rho.$
\end{proposition}

\begin{proof}
Suppose that $u$ solves (\ref{S2E1}), (\ref{S2E2}). Then, integrating by parts
and using (\ref{S2E1}) we obtain:
\begin{equation}
\partial_{t}\left(  \int_{\Omega}\psi_{\rho}udy\right)  -\int_{\Omega}%
\Delta\psi_{\rho}\left(  y\right)  udy-\int_{\Omega}f_{\varepsilon}\left(
u\right)  \nabla\psi_{\rho}\left(  y\right)  \nabla v\left(  y,t\right)  dy=0
\label{S3E4}%
\end{equation}

We rewrite the fundamental solution $G\left(  x,y\right)  $ using Lemma
\ref{L1}. It then follows that:
\[
G\left(  y,x\right)  =-\frac{1}{2\pi}\log\left(  \left\vert x-y\right\vert
\right)  +G_{0}\left(  y,x\right)
\]
where, using that $\frac{3e^{\frac{1}{4}}}{2}<2:$%
\begin{equation}
\left\vert \nabla_{y}G_{0}\left(  y,x\right)  \right\vert \leq\frac{C}{\rho
}\ \ ,\ \ \left\vert y-x_{0}\right\vert \leq\frac{3e^{\frac{1}{4}}\rho}%
{2}\ \ . \label{S3E4a}%
\end{equation}

Then the following representation formula for $\nabla v$ follows from
(\ref{S2E2}) and Lemma \ref{L1}:
\[
\nabla_{y}v\left(  y,t\right)  =\frac{1}{2\pi}\int_{\Omega}\frac{\left(
x-y\right)  }{\left\vert x-y\right\vert ^{2}}f_{\varepsilon}\left(  u\left(
x,t\right)  \right)  dx+\int_{\Omega}\nabla_{y}G_{0}\left(  y,x\right)
f_{\varepsilon}\left(  u\left(  x,t\right)  \right)  dx
\]
and plugging this formula into (\ref{S3E4}) we obtain:
\begin{align}
&  \partial_{t}\left(  \int_{B_{2\rho}\left(  x_{0}\right)  }\psi_{\rho
}udy\right)  -\int_{B_{2\rho}\left(  x_{0}\right)  }\Delta\psi_{\rho
}udy-\label{S3E5}\\
&  -\frac{1}{2\pi}\int_{\Omega}\int_{\Omega}f_{\varepsilon}\left(  u\left(
x,t\right)  \right)  f_{\varepsilon}\left(  u\left(  y,t\right)  \right)
\frac{\left(  x-y\right)  }{\left\vert x-y\right\vert ^{2}}\nabla\psi_{\rho
}\left(  y\right)  dxdy-\nonumber\\
&  -\int_{\Omega}\int_{\Omega}f_{\varepsilon}\left(  u\left(  x,t\right)
\right)  f_{\varepsilon}\left(  u\left(  y,t\right)  \right)  \nabla_{y}%
G_{0}\left(  y,x\right)  \nabla\psi_{\rho}\left(  y\right)  dxdy\nonumber\\
&  =0\nonumber
\end{align}

The the third one in (\ref{S3E5}) can be estimated using the symmetrization
argument introduced in \cite{SenbaSuzuki}. Notice that:
\[
\left\vert \frac{\left(  x-y\right)  }{\left\vert x-y\right\vert ^{2}}%
\cdot\left[  \nabla\psi_{\rho}\left(  y\right)  -\nabla\psi_{\rho}\left(
x\right)  \right]  \right\vert \leq\frac{C}{\rho^{2}}\ .
\]

Then:
\begin{align}
&  \int_{\Omega}\int_{\Omega}f_{\varepsilon}\left(  u\left(  x,t\right)
\right)  f_{\varepsilon}\left(  u\left(  y,t\right)  \right)  \frac{\left(
x-y\right)  }{\left\vert x-y\right\vert ^{2}}\nabla\psi_{\rho}\left(
y\right)  dxdy\nonumber\\
&  =\frac{1}{2}\int_{\Omega}\int_{\Omega}f_{\varepsilon}\left(  u\left(
x,t\right)  \right)  f_{\varepsilon}\left(  u\left(  y,t\right)  \right)
\frac{\left(  x-y\right)  }{\left\vert x-y\right\vert ^{2}}\cdot\left[
\nabla\psi_{\rho}\left(  y\right)  -\nabla\psi_{\rho}\left(  x\right)
\right]  dxdy\leq\frac{C}{\rho^{2}}\ . \label{S3E8}%
\end{align}

On the other hand, the linear term due to the laplacian can be estimated as:%
\begin{equation}
\left\vert \int_{B_{2\rho}\left(  x_{0}\right)  }\Delta\psi_{\rho
}udy\right\vert \leq\frac{C}{\rho^{2}} \label{S3E8a}%
\end{equation}
and (\ref{S3E4a}) yields:
\begin{equation}
\left\vert \int_{\Omega}\int_{\Omega}f_{\varepsilon}\left(  u\left(
x,t\right)  \right)  f_{\varepsilon}\left(  u\left(  y,t\right)  \right)
\nabla_{y}G_{0}\left(  y,x\right)  \nabla\psi_{\rho}\left(  y\right)
dxdy\right\vert \leq\frac{C}{\rho^{2}}\ . \label{S3E9}%
\end{equation}

Combining (\ref{S3E5})-(\ref{S3E9}) we obtain (\ref{S3E1}). If $u$ solves
(\ref{S2E4}), (\ref{S2E5}) a similar computation yields:
\begin{align}
&  \partial_{t}\left(  \int_{\Omega}\psi_{\rho}udy\right)  -\int_{\Omega
}\Delta\psi_{\rho}udy-\varepsilon\int_{\Omega}\Delta\psi_{\rho}u^{\frac{7}{6}%
}dy+\label{S3E6}\\
&  +\frac{1}{2\pi}\int_{\Omega}\int_{\Omega}u\left(  x,t\right)  u\left(
y,t\right)  \frac{\left(  x-y\right)  }{\left\vert x-y\right\vert ^{2}}%
\nabla\psi_{\rho}\left(  y\right)  dxdy-\nonumber\\
&  -\int_{\Omega}\int_{\Omega}u\left(  x,t\right)  u\left(  y,t\right)
\nabla_{y}G_{0}\left(  y,x\right)  \nabla\psi_{\rho}\left(  y\right)
dxdy\nonumber\\
&  =0\ .\nonumber
\end{align}

The last two terms on the left-hand side of (\ref{S3E6}) can be estimated as
in the previous case. The main difference is in the nonlinear term in the
laplacian that can be estimated using (cf. (\ref{S2E6})):
\begin{equation}
\int_{\Omega}\Delta\psi_{\rho}u^{\frac{7}{6}}dy\geq-\frac{2}{\rho^{2}}%
\int_{B_{\rho}\left(  x_{0}\right)  }u^{\frac{7}{6}}dy \label{S3E7}%
\end{equation}
whence (\ref{S3E2}) and therefore Proposition \ref{P1} follows.
\end{proof}

\subsection{Boundary estimates.}

We now derive the local mass growth estimate if the point $x_{0}$ is near the
boundary. To this end we need to construct an auxiliary test function that
will play a role analogous to the function $\psi_{\rho}$ in Proposition
\ref{P1}. This will be made in the following lemma:

\begin{lemma}
\label{L2}Let $\sigma_{0},\ \nu\left(  x\right)  ,\ d\left(  x\right)  $ be as
in Lemma \ref{Lgeom}. There exists $\rho_{0}\leq\sigma_{0}$ small enough,
$\Lambda>1,$ $C>0$ depending only on $\Omega$ such that, for any $\rho
\in\left(  0,\rho_{0}\right)  $ and any $x_{0}\in\bar{\Omega}$ with
$\operatorname*{dist}\left(  x_{0},\partial\Omega\right)  \leq2\rho$ there
exists a function $\psi_{\rho}\in C^{1,1}\left(  \bar{\Omega}\right)  $ with
the following properties:
\begin{align*}
\Delta\psi_{\rho}  &  =-\frac{2}{\rho^{2}}\;\;\text{in\ \ }B_{\rho}\left(
x_{0}\right)  \cap\Omega\\
\Delta\psi_{\rho}  &  \geq0\;\;\text{in\ \ }\left[  B_{\Lambda\rho}\left(
x_{0}\right)  \setminus\overline{B_{\rho}\left(  x_{0}\right)  }\right]
\cap\Omega\\
\partial_{\nu}\psi_{\rho}  &  =0\;\text{in\ \ }B_{\Lambda\rho}\left(
x_{0}\right)  \cap\partial\Omega\ \ ,\ \ \psi_{\rho}=0\;\text{in\ }\bar
{\Omega}\setminus B_{\Lambda\rho}\left(  x_{0}\right) \\
\psi_{\rho}  &  \geq\frac{1}{2}\;\text{in\ }B_{\rho}\left(  x_{0}\right)
\cap\bar{\Omega}\ \ \ \ ,\ \ 0\leq\psi_{\rho}\leq1\;\text{in\ \ }%
B_{\Lambda\rho}\left(  x_{0}\right)  \cap\bar{\Omega}%
\end{align*}%
\[
\rho\left\vert \nabla\psi_{\rho}\right\vert +\rho^{2}\left\vert \nabla^{2}%
\psi_{\rho}\right\vert +\frac{\rho^{2}\left\vert \nu\left(  x\right)
\cdot\nabla\psi_{\rho}\left(  x\right)  \right\vert }{d\left(  x\right)  }\leq
C\;\text{in\ }B_{\Lambda\rho}\left(  x_{0}\right)  \cap\Omega
\]

\end{lemma}

\begin{proof}
The main idea is that for $\rho_{0}$ sufficiently small the problem can be
treated as a perturbation of the problem in the half-plane. We introduce a
rescaled system of coordinates:%
\[
X=\frac{x-P_{\partial}\left(  x_{0}\right)  }{\rho}\ \ ,\ \ X_{0}=\frac
{x_{0}-P_{\partial}\left(  x_{0}\right)  }{\rho}%
\]
where the operator $P_{\partial}\left(  x_{0}\right)  $ is defined as in
(\ref{M2E0}). Notice that the assumption $\operatorname*{dist}\left(
x_{0},\partial\Omega\right)  \leq2\rho$ implies $\left\vert X_{0}\right\vert
\leq2.$ Rotating the coordinate system we can assume that the normal vector
$\nu\left(  P_{\partial}\left(  x_{0}\right)  \right)  $ is $\left(
0,-1\right)  .$ We construct $\tilde{\Psi}\left(  X\right)  $ in the
half-plane solving the problem:%
\begin{align}
\Delta_{X}\tilde{\Psi}\left(  X\right)   &  =-1\ \ ,\ \ X\in B_{1}\left(
X_{0}\right)  \cap\left\{  X=\left(  X_{1},X_{2}\right)  :X_{2}>0\right\}
\label{U2E1}\\
\partial_{\nu_{0}}\tilde{\Psi}\left(  X\right)   &  =0\ \ ,\ \ X\in
\partial\left[  B_{1}\left(  X_{0}\right)  \cap\left\{  X=\left(  X_{1}%
,X_{2}\right)  :X_{2}>0\right\}  \right]  \label{U2E2}%
\end{align}
where $\nu_{0}=\left(  0,-1\right)  $. This problem can be solved using the
reflection method. We can obtain a family of solutions for it in the form:%
\begin{equation}
\tilde{\Psi}\left(  X\right)  =A+\hat{\Psi}\left(  X\right)  \ \ ,\ \ \hat
{\Psi}\left(  X\right)  =-\frac{1}{2\pi}\int_{B_{1}\left(  X_{0}\right)  \cup
B_{1}\left(  X_{0}+2\nu_{0}\right)  }\log\left(  \left\vert X-Y\right\vert
\right)  dY \label{U1E0}%
\end{equation}
where $A$ is an arbitrary constant to be precised. Notice that $\hat{\Psi
}\left(  X\right)  $ is bounded in $\left\vert X\right\vert \leq1$ and it
satisfies:%
\begin{align}
\left\vert \hat{\Psi}\left(  X\right)  +\frac{m}{2\pi}\log\left(  \left\vert
X\right\vert \right)  \right\vert  &  \leq\frac{C}{\left\vert X\right\vert
^{2}}\ \ \text{for\ \ }\left\vert X\right\vert =\lambda_{0}\label{U1E1}\\
\left\vert \nabla_{X}\hat{\Psi}\left(  X\right)  +\frac{m}{2\pi}\frac
{X}{\left\vert X\right\vert ^{2}}\right\vert  &  \leq\frac{C}{\left\vert
X\right\vert ^{3}}\ \ \text{for\ \ }\left\vert X\right\vert =\lambda_{0}
\label{U1E2}%
\end{align}
with $\lambda_{0}$ sufficiently large and $C$ independent on $\lambda_{0}$ and
where $m=\left\vert B_{1}\left(  X_{0}\right)  \cup B_{1}\left(  X_{0}%
+2\nu_{0}\right)  \right\vert $. Notice that $m$ is bounded above and below by
constants independent on $X_{0}.$ In the derivation of (\ref{U1E1}),
(\ref{U1E2}) we have used the fact that $\int_{B_{1}\left(  X_{0}\right)  \cup
B_{1}\left(  X_{0}+2\nu_{0}\right)  }YdY=0.$ We then define:%
\[
\Phi\left(  X\right)  =\frac{m}{4\pi\lambda_{0}}\left(  \lambda_{0}%
+1-\left\vert X\right\vert \right)  _{+}^{2}\ \ \ \ \ \text{for\ \ }\left\vert
X\right\vert \geq\lambda_{0}%
\]
and choose $A$ in (\ref{U1E0}) as:%
\[
A=\frac{m}{4\pi\lambda_{0}}+\frac{m}{2\pi}\log\left(  \lambda_{0}\right)
\]

It then follows from (\ref{U1E0})-(\ref{U1E2}) that:%
\begin{align*}
\left\vert \tilde{\Psi}\left(  X\right)  -\Phi\left(  X\right)  \right\vert
&  \leq\frac{C}{\lambda_{0}^{2}}\ \ \ ,\ \ \left\vert \nabla_{X}\tilde{\Psi
}\left(  X\right)  -\nabla_{X}\Phi\left(  X\right)  \right\vert \leq\frac
{C}{\lambda_{0}^{3}}\ \ \text{for\ \ }\left\vert X\right\vert =\lambda_{0}\\
\Delta_{X}\Phi\left(  X\right)   &  \geq\frac{m}{4\pi\lambda_{0}%
}\ \ \text{for\ \ }\lambda_{0}\leq\left\vert X\right\vert \leq\lambda_{0}+1\\
\tilde{\Psi}\left(  X\right)   &  \geq1\ \ \text{for\ \ }\left\vert
X\right\vert \leq1
\end{align*}
if $\lambda_{0}$ is sufficiently large. Let us consider an function $W\in
C^{2}\left(  \mathbb{R}^{2}\setminus B_{\lambda_{0}}\left(  0\right)  \right)
$ and satisfying:%
\begin{align*}
W\left(  X\right)   &  =\tilde{\Psi}\left(  X\right)  -\Phi\left(  X\right)
\ \ ,\ \ \nabla_{X}W\left(  X\right)  =\nabla_{X}\tilde{\Psi}\left(  X\right)
-\nabla_{X}\Phi\left(  X\right)  \ \ \ \text{for\ \ }\left\vert X\right\vert
=\lambda_{0},\ \\
\ \ W\left(  X\right)   &  =0\ \ \text{for }\left\vert X\right\vert
\geq\lambda_{0}+1\ \ ,\ \ \left\vert \Delta_{X}W\left(  X\right)  \right\vert
\leq\frac{C}{\lambda_{0}^{2}}%
\end{align*}

Then, the function $\Psi\in C^{1,1}\left(  \left\{  X_{2}\geq0\right\}
\right)  $ defined as:%
\begin{align*}
\Psi\left(  X\right)   &  =\tilde{\Psi}\left(  X\right)  \ \ \text{for
\ }\left\vert X\right\vert <\lambda_{0}\\
\Psi\left(  X\right)   &  =\Phi\left(  X\right)  +W\left(  X\right)
\ \ \text{for \ }\left\vert X\right\vert \geq\lambda_{0}%
\end{align*}
satisfies (\ref{U2E1}), (\ref{U2E2}) \ as well as:%
\begin{align}
\Delta_{X}\Psi\left(  X\right)   &  \geq\frac{m}{8\pi\lambda_{0}%
}\ \ \text{for\ }\lambda_{0}\leq\left\vert X\right\vert <\lambda
_{0}+1,\ \text{\ }\left\vert X\right\vert \neq\lambda_{0}\label{U1E3}\\
\Psi\left(  X\right)   &  =0\ \ \text{for }\left\vert X\right\vert \geq
\lambda_{0}+1\nonumber
\end{align}
if $\lambda_{0}$ is sufficiently large.

The function $\Psi$ would provide a solution of the desired problem for planar
$\partial\Omega.$ In order to take into account curvature effects we study the
family of problems:%
\begin{align*}
\Delta_{X}\bar{\Psi}  &  =\Delta_{X}\Psi\left(  X\right)  \ \ \text{in\ \ }%
\left(  \frac{\Omega-P_{\partial}\left(  x_{0}\right)  }{\rho}\right)  \cap
B_{\lambda_{0}+2}\left(  0\right) \\
\partial_{\nu}\bar{\Psi}  &  =0\ \ \text{in\ \ }\left[  \partial\left(
\frac{\Omega-P_{\partial}\left(  x_{0}\right)  }{\rho}\right)  \right]  \cap
B_{\lambda_{0}+2}\left(  0\right) \\
\bar{\Psi}  &  =0\ \ \text{in\ \ }\left(  \frac{\Omega-P_{\partial}\left(
x_{0}\right)  }{\rho}\right)  \cap\partial B_{\lambda_{0}+2}\left(  0\right)
\end{align*}

Classical continuous dependence results on the domain show that $\left\vert
\bar{\Psi}-\Psi\right\vert $ can be made arbitrarily small for $\rho\leq
\rho_{0}$ small. We now construct $\tilde{W}$ satisfying:%
\[
\tilde{W}=0\ \ \left(  \frac{\Omega-P_{\partial}\left(  x_{0}\right)  }{\rho
}\right)  \cap\partial B_{\lambda_{0}+2}\left(  0\right)  \ \ ,\ \ \partial
_{\nu}\tilde{W}=0\ \ \text{in\ \ }\left[  \partial\left(  \frac{\Omega
-P_{\partial}\left(  x_{0}\right)  }{\rho}\right)  \right]  \cap
B_{\lambda_{0}+2}\left(  0\right)
\]
and%
\[
\partial_{\nu}\tilde{W}=\partial_{\nu}\bar{\Psi}\ \ \text{in\ \ }\left(
\frac{\Omega-P_{\partial}\left(  x_{0}\right)  }{\rho}\right)  \cap\partial
B_{\lambda_{0}+2}\left(  0\right)  \ \ ,\ \ \tilde{W}=0\ \ \text{in\ \ }%
\left(  \frac{\Omega-P_{\partial}\left(  x_{0}\right)  }{\rho}\right)  \cap
B_{\lambda_{0}}\left(  0\right)
\]
as well as $\left\vert \Delta\tilde{W}\right\vert $ small in $\left(
\frac{\Omega-P_{\partial}\left(  x_{0}\right)  }{\rho}\right)  \cap
B_{\lambda_{0}+2}\left(  0\right)  ,$ something that it is possible for
$\rho\leq\rho_{0}$ small. Then the function $\Psi_{c}=\bar{\Psi}-\tilde{W}$
satisfies:%
\begin{align*}
\Delta_{X}\Psi_{c}  &  =-1\ \ \text{in\ \ }\left(  \frac{\Omega-P_{\partial
}\left(  x_{0}\right)  }{\rho}\right)  \cap B_{1}\left(  0\right)
\ \ ,\ \ \Delta_{X}\Psi_{c}\geq0\ \ \text{in\ \ }\left(  \frac{\Omega
-P_{\partial}\left(  x_{0}\right)  }{\rho}\right)  \setminus B_{1}\left(
0\right) \\
\partial_{\nu}\Psi_{c}  &  =0\ \ \text{in\ \ }\left[  \partial\left(
\frac{\Omega-P_{\partial}\left(  x_{0}\right)  }{\rho}\right)  \right] \\
\Psi_{c}  &  \in C^{1,1}\left(  \frac{\Omega-P_{\partial}\left(  x_{0}\right)
}{\rho}\right) \\
\Psi_{c}\left(  X\right)   &  =0\ \ \text{for \ \ }\left\vert X\right\vert
\geq\lambda_{0}+2
\end{align*}

The function $\psi_{\rho}\left(  x\right)  =\Psi_{c}\left(  \frac
{x-P_{\partial}\left(  x_{0}\right)  }{\rho}\right)  $ then satisfies all the
properties required in Lemma \ref{L2} for $\rho\leq\rho_{0}.$
\end{proof}

\begin{proposition}
\label{mBound} Suppose that $u$ solves one of the problems (\ref{S2E1}),
(\ref{S2E2}) or (\ref{S2E4}), (\ref{S2E5}). Let us fix $0<\rho<\rho_{0}$ with
$\rho_{0}$ as in Lemma \ref{L2} and let us assume that $\operatorname*{dist}%
\left(  x_{0},\partial\Omega\right)  \leq4\rho.$ Then:
\begin{equation}
\left\vert \partial_{t}\left(  \int_{\Omega}\psi_{\rho}u\right)  \right\vert
\leq\frac{\kappa}{\rho^{2}}\ \ ,\ \ 0<t<\infty\label{B1E1}%
\end{equation}
if $u$ solves (\ref{S2E1}), (\ref{S2E2}),\thinspace and:
\begin{equation}
\partial_{t}\left(  \int_{\Omega}\psi_{\rho}u\right)  \geq-\frac{\kappa}%
{\rho^{2}}-\frac{2\varepsilon}{\rho^{2}}\int_{B_{\rho}\left(  x_{0}\right)
}u^{\frac{7}{6}}\ \ ,\ \ 0<t<\infty\label{B1E2}%
\end{equation}
if $u$ solves (\ref{S2E4}), (\ref{S2E5}). The constant $\kappa$ depends only
on $\left\Vert u_{0}\right\Vert _{L^{1}\left(  \Omega\right)  }$, but it is
independent on $\varepsilon$ and $\rho.$
\end{proposition}

\begin{proof}
Arguing as in the derivation of (\ref{S3E5}), (\ref{S3E6}) and using Lemmas
\ref{Lsymm} and \ref{L1}\ we obtain:
\begin{align}
&  \partial_{t}\left(  \int_{\Omega}\psi_{\rho}udx\right)  -\int_{\Omega
}\Delta\psi_{\rho}udx+\label{S4E1}\\
&  +\frac{1}{2\pi}\int_{\Omega}\int_{\Omega}f_{\varepsilon}\left(  u\left(
x,t\right)  \right)  f_{\varepsilon}\left(  u\left(  y,t\right)  \right)
\frac{\left(  x-y\right)  }{\left\vert x-y\right\vert ^{2}}\nabla\psi_{\rho
}\left(  x\right)  dxdy-\nonumber\\
&  +\frac{1}{2\pi}\int_{\Omega}\int_{\Omega}f_{\varepsilon}\left(  u\left(
x,t\right)  \right)  f_{\varepsilon}\left(  u\left(  y,t\right)  \right)
\frac{\left(  x-\tau\left(  y\right)  \right)  }{\left\vert x-\tau\left(
y\right)  \right\vert ^{2}}\nabla\psi_{\rho}\left(  x\right)  dxdy-\nonumber\\
&  -\int_{\Omega}\int_{\Omega}f_{\varepsilon}\left(  u\left(  x,t\right)
\right)  f_{\varepsilon}\left(  u\left(  y,t\right)  \right)  \nabla
_{x}K\left(  y,x\right)  \nabla\psi_{\rho}\left(  x\right)  dxdy\nonumber\\
&  =0\nonumber
\end{align}
where $\psi_{\rho}$ is now chosen as in Lemma \ref{L2}. Using this lemma we
can estimate all the terms in (\ref{S4E1}) as in the proof of Proposition
\ref{P1} except the fourth term in (\ref{S4E1}). We estimate first the
contribution to this term of the region where $\left\vert x-\tau\left(
y\right)  \right\vert \geq\rho$ using Lemma \ref{L2} as well as the mass
conservation property (\ref{S2E6}):
\begin{equation}
\left\vert \int_{\Omega\times\Omega\cap\left\{  \left\vert x-\tau\left(
y\right)  \right\vert \geq\rho\right\}  }f_{\varepsilon}\left(  u\left(
x,t\right)  \right)  f_{\varepsilon}\left(  u\left(  y,t\right)  \right)
\frac{\left(  x-\tau\left(  y\right)  \right)  }{\left\vert x-\tau\left(
y\right)  \right\vert ^{2}}\nabla\psi_{\rho}\left(  x\right)  dxdy\right\vert
\leq\frac{C}{\rho^{2}}\label{S4E4}%
\end{equation}

In order to estimate the contribution of the region where $\left\vert
x-\tau\left(  y\right)  \right\vert \leq\rho$ we use the fact that for
$\rho_{0}$ sufficiently small
\begin{equation}
\frac{1}{4}\left[  \left\vert x-\tau\left(  x\right)  \right\vert +\left\vert
y-\tau\left(  y\right)  \right\vert \right]  \leq d\left(  x\right)  +d\left(
y\right)  \leq3\left\vert x-\tau\left(  y\right)  \right\vert . \label{S4E4a}%
\end{equation}

Symmetrizing (\ref{S4E4}) we obtain:
\begin{align}
&  f_{\varepsilon}\left(  u\left(  x,t\right)  \right)  f_{\varepsilon}\left(
u\left(  y,t\right)  \right)  \frac{\left(  x-\tau\left(  y\right)  \right)
}{\left\vert x-\tau\left(  y\right)  \right\vert ^{2}}\nabla\psi_{\rho}\left(
x\right) \label{S4E3}\\
&  =\frac{f_{\varepsilon}\left(  u\left(  x,t\right)  \right)  f_{\varepsilon
}\left(  u\left(  y,t\right)  \right)  }{2\left\vert x-\tau\left(  y\right)
\right\vert ^{2}}\left[  \left(  x-\tau\left(  y\right)  \right)  \nabla
\psi_{\rho}\left(  x\right)  +\left(  y-\tau\left(  x\right)  \right)
\nabla\psi_{\rho}\left(  y\right)  \right] \nonumber
\end{align}

Notice that:
\begin{align}
&  \left[  \left(  x-\tau\left(  y\right)  \right)  \nabla\psi_{\rho}\left(
x\right)  +\left(  y-\tau\left(  x\right)  \right)  \nabla\psi_{\rho}\left(
y\right)  \right] \label{S4E3a}\\
&  =\left(  x-\tau\left(  x\right)  \right)  \nabla\psi_{\rho}\left(
x\right)  +\left(  y-\tau\left(  y\right)  \right)  \nabla\psi_{\rho}\left(
y\right)  +\left(  \tau\left(  x\right)  -\tau\left(  y\right)  \right)
\left[  \nabla\psi_{\rho}\left(  x\right)  -\nabla\psi_{\rho}\left(  y\right)
\right] \nonumber
\end{align}

Lemma \ref{L2} as well as the fact that $\left\vert \tau\left(  x\right)
-\tau\left(  y\right)  \right\vert \leq2\left\vert x-y\right\vert
\leq3\left\vert x-\tau\left(  y\right)  \right\vert $ yields:
\begin{equation}
\left\vert \left(  \tau\left(  x\right)  -\tau\left(  y\right)  \right)
\left[  \nabla\psi_{\rho}\left(  x\right)  -\nabla\psi_{\rho}\left(  y\right)
\right]  \right\vert \leq\frac{C}{\rho^{2}}\left\vert x-\tau\left(  y\right)
\right\vert ^{2} \label{S4E3b}%
\end{equation}

On the other hand, using Lemma \ref{L2} we obtain:
\begin{equation}
\left\vert \left(  x-\tau\left(  x\right)  \right)  \nabla\psi_{\rho}\left(
x\right)  \right\vert +\left\vert \left(  y-\tau\left(  y\right)  \right)
\nabla\psi_{\rho}\left(  y\right)  \right\vert \leq\frac{C}{\rho^{2}}\left(
\left\vert x-\tau\left(  x\right)  \right\vert d\left(  x\right)  +\left\vert
y-\tau\left(  y\right)  \right\vert d\left(  y\right)  \right)  \label{S4E3c}%
\end{equation}

Combining (\ref{S4E4}) with (\ref{S4E3})-(\ref{S4E3c}) we obtain:
\[
\left\vert \int_{\Omega\times\Omega}f_{\varepsilon}\left(  u\left(
x,t\right)  \right)  f_{\varepsilon}\left(  u\left(  y,t\right)  \right)
\frac{\left(  x-\tau\left(  y\right)  \right)  }{\left\vert x-\tau\left(
y\right)  \right\vert ^{2}}\nabla\psi_{\rho}\left(  x\right)  dxdy\right\vert
\leq\frac{C}{\rho^{2}}%
\]

This concludes the proof of (\ref{B1E1}). The proof of (\ref{B1E2}) is similar.
\end{proof}

\section{An entropy estimate.}

Entropy estimates for the study of Keller-Segel models were introduced in
\cite{Gajewsky} and they have been extensively used for the analysis of
chemotaxis models. We will use the following estimate for the solutions of the
second regularization considered above (\ref{S2E4}), (\ref{S2E5}).

\begin{lemma}
Let us assume that $\left(  u,v\right)  $ solves (\ref{S2E4}), (\ref{S2E5})
with bounded initial data $u\left(  x,0\right)  =u_{0}\left(  x\right)  $ and
$\varepsilon>0.$ Then, for any $\alpha>0$ there exists $C$ depending only on
$\alpha,u_{0},\Omega$ such that:
\begin{equation}
\varepsilon^{1+\alpha}\int_{\Omega}u^{\frac{7}{6}}dx\leq C\ \ ,\ \ 0<t<\infty
\label{S7E1}%
\end{equation}

\end{lemma}

\begin{proof}
We use the following entropy formula that can be easily checked integrating by
parts for the solutions of (\ref{S2E4}), (\ref{S2E5}):
\[
\partial_{t}\left(  \int_{\Omega}\left[  u\left(  \log\left(  u\right)
-1\right)  +6\varepsilon u^{\frac{7}{6}}-\frac{\left\vert \nabla v\right\vert
^{2}}{2}\right]  dx\right)  =-\int u\left[  \nabla\left(  \log\left(
u\right)  +7\varepsilon u^{\frac{1}{6}}\right)  -\nabla v\right]  ^{2}dx\leq0
\]

Then, since $\int_{\Omega}u\log u\geq C:$
\[
\varepsilon\int_{\Omega}u^{\frac{7}{6}}dx\leq C+\frac{1}{2}\int_{\Omega
}\left\vert \nabla v\right\vert ^{2}dx
\]
where $C$ depends only on $u_{0}$ and $\Omega.$

Classical regularity theory for the Poisson equation yields:
\[
\frac{1}{2}\int_{\Omega}\left\vert \nabla v\right\vert ^{2}dx\leq C\left(
\int_{\Omega}u^{p}dx\right)  ^{\frac{2}{p}}%
\]
for any $p>1,$ with $C$ depending only on $p$ and $\Omega.$ Then:
\[
\frac{1}{2}\int_{\Omega}\left\vert \nabla v\right\vert ^{2}dx\leq C\left(
\int_{\Omega}u^{p}dx\right)  ^{\frac{2}{p}}\leq C\left(  \int_{\Omega}%
u^{\frac{pq-1}{q-1}}dx\right)  ^{\frac{2\left(  q-1\right)  }{pq}}\left(
\int_{\Omega}udx\right)  ^{\frac{2}{pq}}%
\]
for any $q>1.$ Choosing $1<p<\frac{7}{6}$ and $q=\frac{1}{7-6p}$ we obtain:
\[
\varepsilon\int_{\Omega}u^{\frac{7}{6}}dx\leq C\left(  \int_{\Omega}%
u^{\frac{7}{6}}dx\right)  ^{\frac{12\left(  p-1\right)  }{p}}\left(
\int_{\Omega}u_{0}dx\right)  ^{\frac{2\left(  7-6p\right)  }{p}}\leq C\left(
\int_{\Omega}u^{\frac{7}{6}}dx\right)  ^{\frac{12\left(  p-1\right)  }{p}}%
\]
where $p>1$ can be chosen arbitrarily close to one. Young's inequality then
implies:
\[
\varepsilon^{1+\alpha}\int_{\Omega}u^{\frac{7}{6}}dx\leq C
\]
where $C$ depends on $\alpha,$ $u_{0}$ and $\Omega.$
\end{proof}

\section{\label{RegL2}$L^{2}$ estimates.}

We now prove some estimates ensuring that the solutions of (\ref{S2E1}),
(\ref{S2E2}) or (\ref{S2E4}), (\ref{S2E5}) are smooth in regions where the
amount of mass of $u$ is small.

\bigskip

\subsection{Interior estimates: First regularization.}

\bigskip We consider first the regularization of Keller-Segel system in
(\ref{S2E1}), (\ref{S2E2}).

\begin{proposition}
\label{I1}Given $M>0,$ $\kappa>0\;$there exist $m_{0}>0,$ independent of
$M,\;\kappa,\;\varepsilon$ and positive constants $c_{i},\;i=1,2$ depending on
$M,\;\kappa$ but independent of $\varepsilon,$ such that for each $0<\rho
\leq1$ and any solution $\left(  u,v\right)  $ of
\begin{align}
\partial_{t}u-\Delta u+\nabla\left(  f_{\varepsilon}\left(  u\right)  \nabla
v\right)   &  =0\;\;\text{in\ \ }\left(  x,t\right)  \in B_{4\rho}\left(
0\right)  \times\left(  \bar{t}-c_{1}\rho^{2},\bar{t}\right) \label{S5E1}\\
-\Delta v  &  =f_{\varepsilon}\left(  u\right)  -h\left(  t\right)
\;\;\text{in\ \ }\left(  x,t\right)  \in B_{4\rho}\left(  0\right)
\times\left(  \bar{t}-c_{1}\rho^{2},\bar{t}\right)  \label{S5E2}%
\end{align}
satisfying:
\begin{align}
\partial_{t}\left(  \int_{\mathbb{R}^{2}}\varphi\left(  \frac{\left\vert
x\right\vert }{2\rho}\right)  u\left(  x,t\right)  dx\right)   &  \geq
-\frac{\kappa}{\rho^{2}}\;\;,\;\;t\in\left(  \bar{t}-c_{1}\rho^{2},\bar
{t}\right) \label{S5E3a}\\
\int_{B_{4\rho}\left(  0\right)  }u\left(  x,\bar{t}\right)  dx  &  \leq
m_{0}\label{S5E3b}\\
\sup_{t\in\left(  \bar{t}-c_{1}\rho^{2},\bar{t}\right)  }\left\Vert v\left(
\cdot,t\right)  \right\Vert _{L^{6}\left(  B_{4\rho}\left(  0\right)  \right)
}  &  \leq M \label{S5E3c}%
\end{align}%
\begin{equation}
0\leq h\left(  t\right)  \leq M \label{S5E3dd}%
\end{equation}
with $\varphi$ as in (\ref{S3E2A}).

Then, the following inequality holds:
\[
\sup_{s\in\left[  \bar{t}-c_{1}\rho^{2},\bar{t}\right]  }\int_{B_{\rho}\left(
0\right)  }u^{2}\left(  x,s\right)  dx\leq\frac{c_{2}}{\rho^{4}}\;.
\]

\end{proposition}

An essential ingredient in the proof of Proposition \ref{I1} is the following
lemma that has been obtained before in slightly different forms, but that we
prove here by the reader's convenience.

\begin{lemma}
\label{L3}For any $\delta>0$ there exists $C>0$ independent on $\delta$ such
that for any $u\in W_{loc}^{1,2}\left(  \mathbb{R}^{2}\right)  $, any
compactly supported function $\eta\in C^{\infty}\left(  \mathbb{R}^{2}\right)
$ there holds:
\[
\int u^{3}\eta^{6}dx\leq\frac{9\left(  1+\delta\right)  }{16\pi}\left[
\int\left\vert \nabla u\right\vert ^{2}\eta^{6}dx\right]  \left[
\int_{\operatorname*{supp}\left(  \eta\right)  }udx\right]  +\frac{C}%
{\delta^{5}}\left\Vert \nabla\eta\right\Vert _{L^{\infty}}^{6}\left(
\int_{\operatorname*{supp}\left(  \eta\right)  }udx\right)  ^{3}\left(
\int_{\operatorname*{supp}\left(  \eta\right)  }dx\right)
\]

\end{lemma}

\begin{proof}
We apply the following classical Sobolev estimate in the critical case
\[
\int w^{2}dx\leq\frac{1}{4\pi}\left(  \int\left\vert \nabla w\right\vert
dx\right)  ^{2}\;\;,\;\;w\in W^{1,1}\left(  \mathbb{R}^{2}\right)
\]
to the particular function $w=u^{\frac{3}{2}}\eta^{3}.$ Then:
\[
\int u^{3}\eta^{6}dx\leq\frac{9\left(  1+\frac{\delta}{2}\right)  }{16\pi
}\left[  \int\left\vert \nabla u\right\vert ^{2}\eta^{6}dx\right]  \left[
\int_{\operatorname*{supp}\left(  \eta\right)  }udx\right]  +\frac{C}{\delta
}\left[  \int u^{\frac{3}{2}}\eta^{2}\left\vert \nabla\eta\right\vert
dx\right]  ^{2}%
\]

Applying H\"{o}lder and Young's inequalities:
\[
\left[  \int u^{\frac{3}{2}}\eta^{2}\left\vert \nabla\eta\right\vert
dx\right]  ^{2}\leq\delta^{2}\int u^{3}\eta^{6}dx+\frac{C}{\delta^{4}%
}\left\Vert \nabla\eta\right\Vert _{L^{\infty}}^{6}\left(  \int
_{\operatorname*{supp}\left(  \eta\right)  }udx\right)  ^{3}\left(
\int_{\operatorname*{supp}\left(  \eta\right)  }dx\right)
\]
and the result follows.
\end{proof}

\bigskip

\begin{proof}
[Proof of Proposition \ref{I1}]Let $\eta=\eta\left(  x\right)  $ be a cutoff
function satisfying $\eta\left(  x\right)  =1$ for $\left\vert x\right\vert
\leq\rho,\;\eta\left(  x\right)  =0$ for $\left\vert x\right\vert \geq2\rho,$
$\eta\in C^{\infty}$, decreasing on $\left\vert x\right\vert $ and satisfying
$\rho\left\vert \nabla\eta\right\vert +\rho^{2}\left\vert \nabla^{2}%
\eta\right\vert \leq C.$ Let us denote $t_{0}=\bar{t}-2c_{1}\rho^{2},$ where
$c_{1}$ will be precised later. Multiplying (\ref{S5E1}) by the test function
$u\eta^{6}\left(  t-t_{0}\right)  ^{\beta}$ with$\;\beta\geq2$ we obtain,
after integrating by parts:
\begin{align*}
\lefteqn{\partial_{t}\left(  \int\frac{u^{2}}{2}\eta^{6}\left(  t-t_{0}%
\right)  ^{\beta}dx\right)  }\\
&  =-\int\left\vert \nabla u\right\vert ^{2}\eta^{6}\left(  t-t_{0}\right)
^{\beta}dx+\frac{\beta}{2}\int u^{2}\eta^{6}\left(  t-t_{0}\right)  ^{\beta
-1}dx-6\int u\nabla u\eta^{5}\nabla\eta\left(  t-t_{0}\right)  ^{\beta}dx+\\
&  +\int f_{\varepsilon}\left(  u\right)  \nabla u\nabla v\eta^{6}\left(
t-t_{0}\right)  ^{\beta}dx+6\int uf_{\varepsilon}\left(  u\right)  \nabla
v\eta^{5}\nabla\eta\left(  t-t_{0}\right)  ^{\beta}dx
\end{align*}

We integrate by parts again to bring the eliminate the derivatives of $u$ in
the fourth term on the right. Then, if we define $F_{\varepsilon}\left(
u\right)  =\int_{0}^{u}f_{\varepsilon}\left(  s\right)  ds$ and use
(\ref{S5E2}) we obtain:
\begin{align*}
\lefteqn{\partial_{t}\left(  \int\frac{u^{2}}{2}\eta^{6}\left(  t-t_{0}%
\right)  ^{\beta}dx\right)  }\\
&  =-\int\left\vert \nabla u\right\vert ^{2}\eta^{6}\left(  t-t_{0}\right)
^{\beta}dx+\frac{\beta}{2}\int u^{2}\eta^{6}\left(  t-t_{0}\right)  ^{\beta
-1}dx-6\int u\nabla u\eta^{5}\nabla\eta\left(  t-t_{0}\right)  ^{\beta}dx+\\
&  +\int F_{\varepsilon}\left(  u\right)  \left[  f_{\varepsilon}\left(
u\right)  -h\left(  t\right)  \right]  \eta^{6}\left(  t-t_{0}\right)
^{\beta}dx+6\int\left[  uf_{\varepsilon}\left(  u\right)  -F_{\varepsilon
}\left(  u\right)  \right]  \eta^{5}\nabla v\nabla\eta\left(  t-t_{0}\right)
^{\beta}dx
\end{align*}

Eliminating the derivatives of $v$ in the last integral we arrive at:
\begin{align}
\lefteqn{\partial_{t}\left(  \int\frac{u^{2}}{2}\eta^{6}\left(  t-t_{0}%
\right)  ^{\beta}dx\right)  }\label{S5E3}\\
&  =-\int\left\vert \nabla u\right\vert ^{2}\eta^{6}\left(  t-t_{0}\right)
^{\beta}dx+\frac{\beta}{2}\int u^{2}\eta^{6}\left(  t-t_{0}\right)  ^{\beta
-1}dx-6\int u\nabla u\eta^{5}\nabla\eta\left(  t-t_{0}\right)  ^{\beta
}dx+\nonumber\\
&  +\int F_{\varepsilon}\left(  u\right)  f_{\varepsilon}\left(  u\right)
\eta^{6}\left(  t-t_{0}\right)  ^{\beta}dx-6\int\nabla\left[  uf_{\varepsilon
}\left(  u\right)  -F_{\varepsilon}\left(  u\right)  \right]  \eta^{5}%
v\nabla\eta\left(  t-t_{0}\right)  ^{\beta}dx-\nonumber\\
&  -30\int\left[  uf_{\varepsilon}\left(  u\right)  -F_{\varepsilon}\left(
u\right)  \right]  \eta^{4}v\left(  \nabla\eta\right)  ^{2}\left(
t-t_{0}\right)  ^{\beta}dx-\int F_{\varepsilon}\left(  u\right)  h\left(
t\right)  \eta^{6}\left(  t-t_{0}\right)  ^{\beta}dx\nonumber\\
&  -6\int\left[  uf_{\varepsilon}\left(  u\right)  -F_{\varepsilon}\left(
u\right)  \right]  \eta^{5}v\Delta\eta\left(  t-t_{0}\right)  ^{\beta
}dx\nonumber
\end{align}

The last two terms can be estimated easily:
\begin{align}
\left\vert \int\left[  uf_{\varepsilon}\left(  u\right)  -F_{\varepsilon
}\left(  u\right)  \right]  \eta^{4}v\left(  \nabla\eta\right)  ^{2}\left(
t-t_{0}\right)  ^{\beta}dx\right\vert  &  \leq\frac{C}{\rho^{2}}\left(  \int
u^{3}\eta^{6}\left(  t-t_{0}\right)  ^{\beta}dx\right)  ^{\frac{2}{3}}\left(
\int v^{3}\left(  t-t_{0}\right)  ^{\beta}dx\right)  ^{\frac{1}{3}}%
\leq\label{S5E4}\\
&  \leq\delta\int u^{3}\eta^{6}\left(  t-t_{0}\right)  ^{\beta}dx+\frac
{C}{\rho^{6}\delta^{2}}\int_{\operatorname*{supp}\left(  \eta\right)  }%
v^{3}\left(  t-t_{0}\right)  ^{\beta}dx\nonumber\\
\left\vert \int\left[  uf_{\varepsilon}\left(  u\right)  -F_{\varepsilon
}\left(  u\right)  \right]  \eta^{5}v\Delta\eta\left(  t-t_{0}\right)
^{\beta}dx\right\vert  &  \leq\frac{C}{\rho^{2}}\left(  \int u^{3}\eta
^{6}\left(  t-t_{0}\right)  ^{\beta}dx\right)  ^{\frac{2}{3}}\left(  \int
v^{3}\eta^{3}\left(  t-t_{0}\right)  ^{\beta}dx\right)  ^{\frac{1}{3}%
}\label{S5E5}\\
&  \leq\delta\int u^{3}\eta^{6}\left(  t-t_{0}\right)  ^{\beta}dx+\frac
{C}{\rho^{6}\delta^{2}}\int v^{3}\eta^{3}\left(  t-t_{0}\right)  ^{\beta
}dx\nonumber
\end{align}

On the other hand, using that $\left\vert f_{\varepsilon}^{\prime}\right\vert
\leq1,\;\left\vert F_{\varepsilon}^{\prime}\right\vert =\left\vert
f_{\varepsilon}\right\vert \leq\left\vert u\right\vert $ we obtain:
\begin{align}
&  \left\vert \int u\nabla u\eta^{5}\nabla\eta\left(  t-t_{0}\right)  ^{\beta
}dx\right\vert +\left\vert \int\nabla\left[  uf_{\varepsilon}\left(  u\right)
-F_{\varepsilon}\left(  u\right)  \right]  \eta^{5}v\nabla\eta\left(
t-t_{0}\right)  ^{\beta}dx\right\vert \nonumber\\
&  \leq\frac{C}{\rho}\left(  \int u^{3}\eta^{6}\left(  t-t_{0}\right)
^{\beta}dx\right)  ^{\frac{1}{3}}\left(  \int\left\vert \nabla u\right\vert
^{2}\eta^{6}\left(  t-t_{0}\right)  ^{\beta}dx\right)  ^{\frac{1}{2}}\left(
\int\left(  1+\left\vert v\right\vert \right)  ^{6}\left(  t-t_{0}\right)
^{\beta}dx\right)  ^{\frac{1}{6}}\label{S5E6}\\
&  \leq\delta\int u^{3}\eta^{6}\left(  t-t_{0}\right)  ^{\beta}dx+\delta
\int\left\vert \nabla u\right\vert ^{2}\eta^{6}\left(  t-t_{0}\right)
^{\beta}dx+\frac{C}{\rho^{6}\delta^{5}}\int\left(  1+\left\vert v\right\vert
\right)  ^{6}\left(  t-t_{0}\right)  ^{\beta}dx\nonumber
\end{align}

We also have, using $\left\vert F_{\varepsilon}\left(  u\right)  \right\vert
\leq\frac{u^{2}}{2}$ and H\"{o}lder's inequality%
\begin{align}
&  \left\vert \int u^{2}\eta^{6}\left(  t-t_{0}\right)  ^{\beta-1}%
dx\right\vert +\left\vert \int F_{\varepsilon}\left(  u\right)  h\left(
t\right)  \eta^{6}\left(  t-t_{0}\right)  ^{\beta}dx\right\vert \label{S5E7}\\
&  \leq\delta\int u^{3}\eta^{6}\left(  t-t_{0}\right)  ^{\beta}dx+\frac
{C\left(  t-t_{0}\right)  ^{\beta-2}}{\delta}\int u\eta^{6}dx\nonumber
\end{align}
where $C>0$ depends only on $M.$

The most delicate term is the fourth one on the right-hand side of
(\ref{S5E3}). This term can be estimated using Lemma \ref{L3} as:
\begin{align}
\left\vert \int F_{\varepsilon}\left(  u\right)  f_{\varepsilon}\left(
u\right)  \eta^{6}\left(  t-t_{0}\right)  ^{\beta}dx\right\vert  &  \leq
\frac{9\left(  1+\delta\right)  }{32\pi}\left[  \int\left\vert \nabla
u\right\vert ^{2}\eta^{6}\left(  t-t_{0}\right)  ^{\beta}dx\right]  \left[
\int_{\operatorname*{supp}\left(  \eta\right)  }udx\right] \label{S5E8}\\
&  +\frac{C\left(  t-t_{0}\right)  ^{\beta}}{\delta^{5}}\left\Vert \nabla
\eta\right\Vert _{L^{\infty}}^{6}\left(  \int_{\operatorname*{supp}\left(
\eta\right)  }udx\right)  ^{3}\left(  \int_{\operatorname*{supp}\left(
\eta\right)  }dx\right) \nonumber
\end{align}

Combining (\ref{S5E3})-(\ref{S5E8}) we obtain:
\begin{align}
\lefteqn{\partial_{t}\left(  \int\frac{u^{2}}{2}\eta^{6}\left(  t-t_{0}%
\right)  ^{\beta}dx\right)  }\label{S5E9}\\
&  \leq-\left(  1-\delta\right)  \int\left\vert \nabla u\right\vert ^{2}%
\eta^{6}\left(  t-t_{0}\right)  ^{\beta}dx+\frac{9\left(  1+\delta\right)
}{32\pi}\left[  \int\left\vert \nabla u\right\vert ^{2}\eta^{6}\left(
t-t_{0}\right)  ^{\beta}dx\right]  \left[  \int_{\operatorname*{supp}\left(
\eta\right)  }udx\right]  +\nonumber\\
&  +\frac{C\left(  t-t_{0}\right)  ^{\beta}}{\delta^{5}\rho^{6}}\left(
\int_{\operatorname*{supp}\left(  \eta\right)  }udx\right)  ^{3}\left(
\int_{\operatorname*{supp}\left(  \eta\right)  }dx\right)  +\frac{C}{\rho
^{6}\delta^{2}}\int_{\operatorname*{supp}\left(  \eta\right)  }v^{3}\left(
t-t_{0}\right)  ^{\beta}dx+\nonumber\\
&  +\frac{C}{\rho^{6}\delta^{5}}\int\left(  1+\left\vert v\right\vert \right)
^{6}\left(  t-t_{0}\right)  ^{\beta}dx+\frac{C\left(  t-t_{0}\right)
^{\beta-2}}{\delta}\int_{\operatorname*{supp}\left(  \eta\right)
}udx\nonumber
\end{align}
where we have estimated all the terms $\delta\int u^{3}\eta^{6}\left(
t-t_{0}\right)  ^{\beta}dx$ on the right-hand side of (\ref{S5E4}%
)-(\ref{S5E7}) using Lemma \ref{L3}. The values of $\delta$ and $C$ have been
then changed.

Let us write $m_{0}=\frac{8\pi}{27}.$ Using assumptions (\ref{S5E3a}),
(\ref{S5E3b}) as well as the definition of $\varphi$ it follows that, if
$c_{1}$ is chosen sufficiently small (although independent on $\rho$), we
have:
\[
\int_{B_{2\rho}\left(  0\right)  }u\left(  x,t\right)  dx\leq3m_{0}%
\;\;\text{for\ \ }t\in\left(  \bar{t}-2c_{1}\rho^{2},\bar{t}\right)
\]

Then, if $\delta$ is small enough, it follows from (\ref{S5E3c}), (\ref{S5E9})
that:
\begin{align}
\lefteqn{\partial_{t}\left(  \int\frac{u^{2}}{2}\eta^{6}\left(  t-t_{0}%
\right)  ^{\beta}dx\right)  \leq\frac{C\left(  t-t_{0}\right)  ^{\beta}%
}{\delta^{5}\rho^{6}}m_{0}^{3}\rho^{2}+\frac{C\left(  t-t_{0}\right)  ^{\beta
}}{\rho^{6}\delta^{2}}M^{3}\rho+}\nonumber\\
&  +\frac{C\left(  t-t_{0}\right)  ^{\beta}}{\rho^{6}\delta^{5}}\left[
\rho^{2}+M^{6}\right]  +\frac{Cm_{0}\left(  t-t_{0}\right)  ^{\beta-2}}%
{\delta}\nonumber
\end{align}
and assuming that $\rho_{0}$ is small enough and $M$ is of order one, without
loss of generality, we obtain:
\begin{equation}
\lefteqn{\partial_{t}\left(  \int\frac{u^{2}}{2}\eta^{6}\left(  t-t_{0}%
\right)  ^{\beta}dx\right)  \leq\frac{K\left(  t-t_{0}\right)  ^{\beta}}%
{\rho^{6}\delta^{5}}+\frac{Cm_{0}\left(  t-t_{0}\right)  ^{\beta-2}}{\delta}%
}\nonumber
\end{equation}
where $K$ depends on $M.$ Integrating this formula, with $\beta=2$ in the
interval $t\in\left(  t_{0},\bar{t}\right)  ,$ and using that $t_{0}=\bar
{t}-2c_{1}\rho^{2}$ it follows that:
\[
\int\left(  u\left(  x,s\right)  \right)  ^{2}\eta^{6}dx\leq K\left[
\frac{\left(  \bar{t}-t_{0}\right)  }{\rho^{6}\delta^{5}}+\frac{1}{\delta
}\frac{1}{\left(  \bar{t}-t_{0}\right)  }\right]  \leq\frac{K}{\delta^{5}%
\rho^{4}}\;,\;s\in\left[  \bar{t}-c_{1}\rho^{2},\bar{t}\right]
\]
and since $\delta$ is of order one (although small) the result follows just
changing $\frac{c_{1}}{2}$ by $c_{1}.$
\end{proof}

\bigskip

\subsection{Interior estimates: Second regularization.}

\bigskip

We now derive interior estimates for the regularization in (\ref{S2E4}),
(\ref{S2E5}).

\begin{proposition}
\label{I2}Given $M>1,$ $\kappa>0\;$there exist $m_{0}>0$ independent of
$M,\;\kappa,\;\varepsilon,$ and positive constants $c_{i}%
,\;i=1,2,\;\varepsilon_{0}>0,\;\rho_{0}>0$ depending on $M,\;\kappa$ but
independent of $\varepsilon,$ such that for each $0<\rho\leq\rho
_{0},\;0<\varepsilon\leq\varepsilon_{0}$ and any solution $\left(  u,v\right)
$ of
\begin{equation}
\partial_{t}u-\Delta\left(  u+\varepsilon u^{\frac{7}{6}}\right)
+\nabla\left(  u\nabla v\right)  =0\;\;\text{in\ \ }\left(  x,t\right)  \in
B_{4\rho}\left(  0\right)  \times\left(  \bar{t}-c_{1}\rho^{2},\bar{t}\right)
\label{S6E1}%
\end{equation}%
\begin{equation}
\Delta v=u-h\left(  t\right)  \;\;\text{in\ \ }\left(  x,t\right)  \in
B_{4\rho}\left(  0\right)  \times\left(  \bar{t}-c_{1}\rho^{2},\bar{t}\right)
\label{S6E2a}%
\end{equation}
satisfying:
\begin{align}
\partial_{t}\left(  \int_{\Omega}\varphi\left(  \frac{\left\vert x\right\vert
}{2\rho}\right)  u\left(  x,t\right)  dx\right)   &  \geq-\frac{\kappa}%
{\rho^{2}}-\frac{2\varepsilon}{\rho^{2}}\int_{B_{\rho}\left(  0\right)
}u^{\frac{7}{6}}\left(  x,t\right)  dx\;\;,\;t\in\left(  \bar{t}-c_{1}\rho
^{2},\bar{t}\right) \label{S6E3}\\
\int_{B_{4\rho}\left(  0\right)  }u\left(  x,\bar{t}\right)  dx  &  \leq
m_{0}\;\label{S6E4}\\
\sup_{t\in\left[  \bar{t}-2c_{1}\rho^{2},\bar{t}\right]  }\left[
\varepsilon^{\frac{3}{2}}\int_{B_{\rho}\left(  0\right)  }u^{\frac{7}{6}%
}\left(  x,t\right)  dx\right]   &  \leq M\;\;\label{S6E5}\\
\sup_{t\in\left(  \bar{t}-c_{1}\rho^{2},\bar{t}\right)  }\left\Vert v\left(
\cdot,t\right)  \right\Vert _{L^{6}\left(  B_{4\rho}\left(  0\right)  \right)
}  &  \leq M\label{S6E6}\\
0  &  \leq h\left(  t\right)  \leq M \label{S6E6a}%
\end{align}
with $\varphi$ as in (\ref{S3E2A}).

Then, the following inequality holds:
\begin{equation}
\sup_{s\in\left[  \bar{t}-c_{1}\rho^{2},\bar{t}\right]  }\int_{B_{\rho}\left(
0\right)  }u^{\frac{7}{6}}\left(  x,s\right)  dx\leq\frac{c_{2}}{\rho^{4}}\;
\label{S6E7}%
\end{equation}

\end{proposition}

\begin{proof}
\bigskip Let us assume that $\eta$ is the same cutoff function used in the
proof of Proposition \ref{I1}. Arguing as there we obtain the inequality
(\ref{S5E9}) for any $t_{0}\in\left[  \bar{t}-2c_{1}\rho^{2},\bar{t}\right]
.$ We can estimate the terms in (\ref{S5E9}) containing $v$ as in the proof of
Proposition \ref{I1}. This gives:
\begin{align}
\partial_{t}\left(  \int\frac{u^{2}}{2}\eta^{6}\left(  t-t_{0}\right)
^{\beta}dx\right)   &  \leq\left[  -1+\delta+\frac{9\left(  1+\delta\right)
}{32\pi}\left[  \int_{\operatorname*{supp}\left(  \eta\right)  }udx\right]
\right]  \int\left\vert \nabla u\right\vert ^{2}\eta^{6}\left(  t-t_{0}%
\right)  ^{\beta}dx+\nonumber\\
&  +\frac{K\left(  t-t_{0}\right)  ^{\beta}}{\rho^{6}\delta^{5}}%
+\frac{C\left(  t-t_{0}\right)  ^{\beta}}{\delta^{5}\rho^{6}}\left(
\int_{\operatorname*{supp}\left(  \eta\right)  }udx\right)  ^{3}\left(
\int_{\operatorname*{supp}\left(  \eta\right)  }dx\right)  +\nonumber\\
&  +\frac{C\left(  t-t_{0}\right)  ^{\beta-2}}{\delta}\int
_{\operatorname*{supp}\left(  \eta\right)  }udx-\frac{7\varepsilon}{6}\int
\eta^{6}\left(  t-t_{0}\right)  ^{\beta}u^{\frac{1}{6}}\left(  \nabla
u\right)  ^{2}dx\label{S6E9a}\\
&  -6\varepsilon\int\eta^{5}\left(  t-t_{0}\right)  ^{\beta}u\nabla\eta
\nabla\left(  u^{\frac{7}{6}}\right)  dx\nonumber
\end{align}
with $K$ depending only on $M.$ The last two terms are due to the regularizing
term $\varepsilon\Delta\left(  u^{\frac{7}{6}}\right)  .$ The term
$-\frac{7\varepsilon}{6}\int\eta^{6}\left(  t-t_{0}\right)  ^{\beta}%
u^{\frac{1}{6}}\left(  \nabla u\right)  ^{2}dx$ is nonpositive and therefore
it can be estimated above by zero. It remains to estimate the additional term
$-6\varepsilon\int\eta^{5}\left(  t-t_{0}\right)  ^{\beta}u\nabla\eta
\nabla\left(  u^{\frac{7}{6}}\right)  dx.$ To this end we use the estimate:
\[
\left\vert \int\eta^{5}\left(  t-t_{0}\right)  ^{\beta}u\nabla\eta
\nabla\left(  u^{\frac{7}{6}}\right)  dx\right\vert \leq\frac{C\left(
t-t_{0}\right)  ^{\beta}}{\rho}\left[  \int_{\operatorname*{supp}\left(
\eta\right)  }udx\right]  ^{\frac{1}{6}}\left[  \int\eta^{6}u^{3}dx\right]
^{\frac{1}{3}}\left[  \int\eta^{6}\left\vert \nabla u\right\vert
^{2}dx\right]  ^{\frac{1}{2}}%
\]

Using now Lemma \ref{L3} we obtain:
\[
\left\vert \int\eta^{5}\left(  t-t_{0}\right)  ^{\beta}u\nabla\eta
\nabla\left(  u^{\frac{7}{6}}\right)  dx\right\vert \leq\delta\int\left\vert
\nabla u\right\vert ^{2}\eta^{6}\left(  t-t_{0}\right)  ^{\beta}%
dx+\frac{C\left(  t-t_{0}\right)  ^{\beta}}{\rho^{6}\delta^{5}}\left[
1+\left[  \int_{\operatorname*{supp}\left(  \eta\right)  }udx\right]
^{3}\right]
\]

Using this estimate in (\ref{S6E9a}) we arrive at:
\begin{align}
\partial_{t}\left(  \int\frac{u^{2}}{2}\eta^{6}\left(  t-t_{0}\right)
^{\beta}dx\right)   &  \leq\left[  -1+2\delta+\frac{9\left(  1+\delta\right)
}{16\pi}\left[  \int_{\operatorname*{supp}\left(  \eta\right)  }udx\right]
\right]  \int\left\vert \nabla u\right\vert ^{2}\eta^{6}\left(  t-t_{0}%
\right)  ^{\beta}dx+\label{S6E9}\\
&  +\frac{K\left(  t-t_{0}\right)  ^{\beta}}{\rho^{6}\delta^{5}}%
+\frac{C\left(  t-t_{0}\right)  ^{\beta}}{\delta^{5}\rho^{6}}\left(
\int_{\operatorname*{supp}\left(  \eta\right)  }udx\right)  ^{3}+\nonumber\\
&  +\frac{C\left(  t-t_{0}\right)  ^{\beta-2}}{\delta}\int
_{\operatorname*{supp}\left(  \eta\right)  }udx\nonumber
\end{align}
with $K$ depending only on $M$ and $C$ just a numerical constant.

For any $\bar{t},$ let us define $t^{\ast}$ as:
\[
t^{\ast}=\inf\left\{  t\in\left[  \bar{t}-2c_{1}\rho^{2},\bar{t}\right]
:\sup_{s\in\left[  t,\bar{t}\right]  }\int_{B_{2\rho}\left(  0\right)
}u\left(  x,s\right)  dx\leq3m_{0}\right\}
\]
with $m_{0}=\frac{4\pi}{27}.$

Our goal is to show that $t^{\ast}=\bar{t}-2c_{1}\rho^{2}$ if $c_{1}$ is
sufficiently small (with $c_{1}$ independent on $\varepsilon,\rho$). First we
notice that the assumptions of Proposition \ref{I2} imply $t^{\ast}\leq
\min\left\{  \bar{t}-\frac{m_{0}\rho^{2}\sqrt{\varepsilon}}{2\left(
\kappa\sqrt{\varepsilon}+2M\right)  },\bar{t}-2c_{1}\rho^{2}\right\}  .$
Indeed, this is an easy consequence of the fact that the definition of
$\varphi$ combined with (\ref{S6E3})-(\ref{S6E5}) imply:
\begin{equation}
\int_{B_{2\rho}\left(  0\right)  }u\left(  x,t\right)  dx\leq2m_{0}+2\left(
\kappa+\frac{2M}{\sqrt{\varepsilon}}\right)  \frac{\left(  \bar{t}-t\right)
}{\rho^{2}} \label{S6E8}%
\end{equation}
and the left hand side of (\ref{S6E8}) is smaller than $3m_{0}$ for any $t$
satisfying $\left(  \bar{t}-t\right)  \leq\frac{m_{0}\rho^{2}}{2\left(
\kappa+\frac{2M}{\sqrt{\varepsilon}}\right)  }.$

Let us suppose that $t^{\ast}>\bar{t}-c_{1}\rho^{2}.$ Then, (\ref{S6E9})
implies that, if $\delta=\frac{1}{4},$ for $t_{0}=t^{\ast}$ and $t\in\left[
t^{\ast},\bar{t}\right]  :$%
\[
\partial_{t}\left(  \int\frac{u^{2}}{2}\eta^{6}\left(  t-t^{\ast}\right)
^{\beta}dx\right)  \leq K\left[  \frac{\left(  t-t^{\ast}\right)  ^{\beta}%
}{\rho^{6}}+\left(  t-t^{\ast}\right)  ^{\beta-2}\right]
\]
where $K$ whose value can change, depends only on $M.$ Assuming that
$\beta=2,$ and integrating in $\left[  t^{\ast},\bar{t}\right]  $ we obtain:
\begin{equation}
\int_{B_{\rho}\left(  0\right)  }u^{2}\left(  x,t\right)  dx\leq K\left[
\frac{\left(  t-t^{\ast}\right)  }{\rho^{6}}+\frac{1}{\left(  t-t^{\ast
}\right)  }\right]  \;\;,\;\;t\in\left[  t^{\ast},\bar{t}\right]
\label{S6E10}%
\end{equation}

This estimate yields:
\[
\int_{B_{\rho}\left(  0\right)  }u^{\frac{7}{6}}\left(  x,t\right)
dx\leq\left[  \int_{B_{\rho}\left(  0\right)  }u^{2}\left(  x,t\right)
dx\right]  ^{\frac{1}{6}}\left[  \int_{B_{\rho}\left(  0\right)  }u\left(
x,t\right)  dx\right]  ^{\frac{5}{6}}\leq K\left[  \frac{\left(  t-t^{\ast
}\right)  ^{\frac{1}{6}}}{\rho}+\frac{1}{\left(  t-t^{\ast}\right)  ^{\frac
{1}{6}}}\right]
\]
with $K$ depending only on $M$. Using (\ref{S6E3}):
\[
\partial_{t}\left(  \int_{\Omega}\varphi\left(  \frac{\left\vert x\right\vert
}{2\rho}\right)  u\left(  x,t\right)  dx\right)  \geq-\frac{\kappa}{\rho^{2}%
}-\frac{K\varepsilon}{\rho^{2}}\left[  \frac{\left(  t-t^{\ast}\right)
^{\frac{1}{6}}}{\rho}+\frac{1}{\left(  t-t^{\ast}\right)  ^{\frac{1}{6}}%
}\right]
\]

Integrating this formula between $t$ and $\bar{t}$ and using that $t\geq
t^{\ast}$ we obtain:%

\[
\frac{1}{2}\int_{B_{2\rho}\left(  0\right)  }u\left(  x,t\right)  dx\leq
m_{0}+\frac{\kappa}{\rho^{2}}\left(  \bar{t}-t^{\ast}\right)  +\frac
{K\varepsilon}{\rho^{3}}\left(  \bar{t}-t^{\ast}\right)  ^{\frac{7}{6}}%
+\frac{K\varepsilon}{\rho^{2}}\left(  \bar{t}-t^{\ast}\right)  ^{\frac{5}{6}}%
\]

Suppose first that $\varepsilon\leq\rho^{\frac{2}{3}}.$ Then, since $\left(
\bar{t}-t^{\ast}\right)  \leq\rho^{2}$ we obtain:
\[
\frac{1}{2}\int_{B_{2\rho}\left(  0\right)  }u\left(  x,t\right)  dx\leq
m_{0}+2\kappa c_{1}+\frac{Kc_{1}^{\frac{7}{6}}\varepsilon}{\rho^{\frac{2}{3}}%
}+\frac{Kc_{1}^{\frac{5}{6}}\varepsilon}{\rho^{\frac{1}{3}}}\leq m_{0}+2\kappa
c_{1}+K\left[  c_{1}^{\frac{7}{6}}+c_{1}^{\frac{5}{6}}\right]
\]

Then:
\[
\int_{B_{2\rho}\left(  0\right)  }u\left(  x,t\right)  dx\leq2m_{0}+4\kappa
c_{1}+2K\left[  c_{1}^{\frac{7}{6}}+c_{1}^{\frac{5}{6}}\right]  \;,\;t\in
\left[  t^{\ast},\bar{t}\right]
\]

Choosing $c_{1}$ small enough we obtain $4\kappa c_{1}+2K\left[  c_{1}%
^{\frac{7}{6}}+c_{1}^{\frac{5}{6}}\right]  <m_{0}.$ This contradicts the
definition of $t^{\ast}$ unless $t^{\ast}=\bar{t}-2c_{1}\rho^{2}.$ Using
(\ref{S6E10}) with $t=\bar{t}$ it then follows that:
\begin{equation}
\sup_{s\in\left[  \bar{t}-c_{1}\rho^{2},\bar{t}\right]  }\int_{B_{\rho}\left(
0\right)  }u^{2}\left(  x,s\right)  dx\leq\frac{c_{2}}{\rho^{4}}%
\;\;\;\;\text{if\ \ }\varepsilon\leq\rho^{\frac{2}{3}} \label{S8E1}%
\end{equation}

Suppose now that $\varepsilon\geq\rho^{\frac{2}{3}}.$ Then (\ref{S6E5})
implies:
\begin{equation}
\sup_{s\in\left[  \bar{t}-c_{1}\rho^{2},\bar{t}\right]  }\int_{B_{\rho}\left(
0\right)  }u^{\frac{7}{6}}\left(  x,s\right)  dx\leq\frac{c_{2}}%
{\varepsilon^{\frac{3}{2}}}\leq\frac{c_{2}}{\rho}\;\;\text{if\ \ }%
\varepsilon\geq\rho^{\frac{2}{3}} \label{S8E2}%
\end{equation}

\end{proof}

\subsection{Boundary estimates.}

Estimates analogous to Propositions \ref{I1}, \ref{I2} can be obtained near
the boundary points. The proof is similar, with the only difference of using
test functions $\varphi$ and $\eta$\ with homogeneous boundary conditions at
$\partial\Omega.$ We formulate the results by completeness, but the details of
the proofs will be omitted.

\bigskip

In the case of the first regularization we have:

\begin{proposition}
\label{BI1}Let $\Lambda$ be as in Lemma \ref{L2}. Given $M>0,$ $\kappa>0$
there exist $m_{0}>0,$ independent of $M,\;\kappa,\;\varepsilon$ and positive
constants $c_{i},\;i=1,2$ depending on $M,\;\kappa$ but independent of
$\varepsilon,$ such that for each $0<\rho\leq1,$ any $\ x_{0}\in\bar{\Omega
}\;$with $d\left(  x_{0}\right)  \leq4\rho$ and any solution $\left(
u,v\right)  $of
\begin{align*}
\partial_{t}u-\Delta u+\nabla\left(  f_{\varepsilon}\left(  u\right)  \nabla
v\right)   &  =0\;\;\text{in\ \ }\left(  x,t\right)  \in\left[  B_{\Lambda
\rho}\left(  x_{0}\right)  \cap\Omega\right]  \times\left(  \bar{t}-c_{1}%
\rho^{2},\bar{t}\right) \\
-\Delta v  &  =f_{\varepsilon}\left(  u\right)  -h\left(  t\right)
\;\;\text{in\ \ }\left(  x,t\right)  \in\left[  B_{\Lambda\rho}\left(
x_{0}\right)  \cap\Omega\right]  \times\left(  \bar{t}-c_{1}\rho^{2},\bar
{t}\right)
\end{align*}
satisfying:
\begin{align*}
\partial_{t}\left(  \int_{\Omega}\psi_{\rho}\left(  x\right)  u\left(
x,t\right)  dx\right)   &  \geq-\frac{\kappa}{\rho^{2}}\;\;,\;\;t\in\left(
\bar{t}-c_{1}\rho^{2},\bar{t}\right) \\
\int_{\left[  B_{\Lambda\rho}\left(  x_{0}\right)  \cap\Omega\right]
}u\left(  x,\bar{t}\right)  dx  &  \leq m_{0}\\
\sup_{t\in\left(  \bar{t}-c_{1}\rho^{2},\bar{t}\right)  }\left\Vert v\left(
\cdot,t\right)  \right\Vert _{L^{6}\left(  \left[  B_{\Lambda\rho}\left(
x_{0}\right)  \cap\Omega\right]  \right)  }  &  \leq M\\
0  &  \leq h\left(  t\right)  \leq M
\end{align*}
with $\psi_{\rho}$ as in Lemma \ref{L2}.

Then, the following inequality holds:
\[
\sup_{s\in\left[  \bar{t}-c_{1}\rho^{2},\bar{t}\right]  }\int_{B_{\rho}\left(
x_{0}\right)  }u^{2}\left(  x,s\right)  dx\leq\frac{c_{2}}{\rho^{4}}\;.
\]

\end{proposition}

\bigskip In the case of the second regularization we have:

\begin{proposition}
\label{BI2}Let $\Lambda$ be as in Lemma \ref{L2}. Given $M>1,$ $\kappa
>0\;$there exist $m_{0}>0$ independent of $M,\;\kappa,\;\varepsilon,$ and
positive constants $c_{i},\;i=1,2,\;\varepsilon_{0}>0,\;\rho_{0}>0$ depending
on $M,\;\kappa$ but independent of $\varepsilon,$ such that for each
$0<\rho\leq\rho_{0},\;0<\varepsilon\leq\varepsilon_{0},$ any $\ x_{0}\in
\bar{\Omega}\;$with $d\left(  x_{0}\right)  \leq4\rho$ and any solution
$\left(  u,v\right)  $ of
\begin{align*}
\partial_{t}u-\Delta\left(  u+\varepsilon u^{\frac{7}{6}}\right)
+\nabla\left(  u\nabla v\right)   &  =0\;\;\text{in\ \ }\left(  x,t\right)
\in\left[  B_{\Lambda\rho}\left(  x_{0}\right)  \cap\Omega\right]
\times\left(  \bar{t}-c_{1}\rho^{2},\bar{t}\right) \\
-\Delta v  &  =u-h\left(  t\right)  \;\;\text{in\ \ }\left(  x,t\right)
\in\left[  B_{\Lambda\rho}\left(  x_{0}\right)  \cap\Omega\right]
\times\left(  \bar{t}-c_{1}\rho^{2},\bar{t}\right)
\end{align*}
satisfying:
\begin{align*}
\partial_{t}\left(  \int_{\Omega}\psi_{\rho}\left(  x\right)  u\left(
x,t\right)  dx\right)   &  \geq-\frac{\kappa}{\rho^{2}}-\frac{2\varepsilon
}{\rho^{2}}\int_{B_{\rho}\left(  x_{0}\right)  }u^{\frac{7}{6}}\left(
x,t\right)  dx\;\;,\;t\in\left(  \bar{t}-c_{1}\rho^{2},\bar{t}\right) \\
\int_{\left[  B_{\Lambda\rho}\left(  x_{0}\right)  \cap\Omega\right]
}u\left(  x,\bar{t}\right)  dx  &  \leq m_{0}\;\\
\sup_{t\in\left[  \bar{t}-2c_{1}\rho^{2},\bar{t}\right]  }\left[
\varepsilon^{\frac{3}{2}}\int_{B_{\rho}\left(  x_{0}\right)  }u^{\frac{7}{6}%
}\left(  x,t\right)  dx\right]   &  \leq M\;\;\\
\sup_{t\in\left(  \bar{t}-c_{1}\rho^{2},\bar{t}\right)  }\left\Vert v\left(
\cdot,t\right)  \right\Vert _{L^{6}\left(  \left[  B_{\Lambda\rho}\left(
x_{0}\right)  \cap\Omega\right]  \right)  }  &  \leq M\\
0  &  \leq h\left(  t\right)  \leq M
\end{align*}
with $\psi_{\rho}$ as in Lemma \ref{L2}.

Then, the following inequality holds:
\begin{equation}
\sup_{s\in\left[  \bar{t}-c_{1}\rho^{2},\bar{t}\right]  }\int_{B_{\rho}\left(
x_{0}\right)  }u^{\frac{7}{6}}\left(  x,s\right)  dx\leq\frac{c_{2}}{\rho^{4}}
\label{S6E7a}%
\end{equation}

\end{proposition}

\section{The limit $\varepsilon\rightarrow0:$ Measured valued solutions.}

The regularity results above allow to obtain convergence of the solutions of
the problems (\ref{S2E1}), (\ref{S2E2}) and (\ref{S2E4}), (\ref{S2E5}) to some
measures $\mu$ whose ''singular set'' is a finite set of points for each time
$t.$

We will denote in the following as $M^{+}\left(  \Omega\times\mathbb{R}%
^{+}\right)  $ the space of positive Radon measures in $\Omega\times
\mathbb{R}^{+}.$ We will also denote as $N$ the number:
\begin{equation}
N=\left[  \frac{4\int_{\Omega}u_{0}dx}{m_{0}}\right]  \label{T1E1}%
\end{equation}
where $\left[  x\right]  $ denotes the integer part of $x\geq0$ and $m_{0}$ is
as in Proposition \ref{I1}.

\bigskip

We begin considering the limit of the solutions of (\ref{S2E1}), (\ref{S2E2}):

\begin{proposition}
\label{M1}Suppose that $u^{\varepsilon}$ is the solution of (\ref{S2E1}),
(\ref{S2E2}) with initial value $u^{\varepsilon}\left(  x,0\right)
=u_{0}\left(  x\right)  ,\;x\in\Omega$ and $\varepsilon>0.$ Then, there exist
Radon measures $\mu,\mu^{-}\in M^{+}\left(  \Omega\times\mathbb{R}^{+}\right)
$ and a subsequence $\left\{  \varepsilon_{k}\right\}  _{k=1}^{\infty}$ such
that:
\begin{align*}
u^{\varepsilon_{k}} &  \rightharpoonup\mu\;\;\text{,\ \ }f_{\varepsilon_{k}%
}\left(  u^{\varepsilon_{k}}\right)  \rightharpoonup\mu^{-}\text{\ \ as\ \ }%
k\rightarrow\infty\\
\mu^{-} &  \leq\mu
\end{align*}
in the $\ast-$weak topology. Moreover, the measures $\mu,\ \mu^{-}\ $ can be
written as the product:
\[
d\mu=d\mu_{t}dt\ \ \ ,\ \ d\mu^{-}=d\mu_{t}^{-}dt
\]
with
\begin{equation}
\mu_{t}\left(  \Omega\right)  =\int_{\Omega}u_{0}\left(  x\right)
dx\ \ \ ,\ \ \mu_{t}^{-}\leq\mu_{t}\ .\label{T1E1e}%
\end{equation}

We define the singular set of $\mu,$ and denote it as $S,$ as the set of
points $\left(  x_{0},t_{0}\right)  \in\bar{\Omega}\times\left[
0,\infty\right)  $ where:
\begin{equation}
\lim_{\delta\rightarrow0}\inf\left[  \frac{\mu\left(  B_{\rho}\left(
x_{0}\right)  \times\left[  t_{0},t_{0}+\delta\right]  \right)  }{\delta
}\right]  =\lim_{\delta\rightarrow0}\inf\left[  \frac{1}{\delta}\int_{t_{0}%
}^{t_{0}+\delta}\mu_{t}\left(  B_{\rho}\left(  x_{0}\right)  \right)
dt\right]  \geq\frac{m_{0}}{2}\;\;\text{for any }\rho>0\label{T1E1a}%
\end{equation}
with $m_{0}$ as in (\ref{S5E3b}). Then we can write:
\begin{equation}
\mu=\bar{\mu}+u\ \ \ ,\ \ \ \mu^{-}=\bar{\mu}^{-}+u\label{T1E1b}%
\end{equation}
where $\bar{\mu},\bar{\mu}^{-}\in M^{+}\left(  \Omega\times\mathbb{R}%
^{+}\right)  $ are supported in the set $S$ and $u\in C^{\infty}\left(
\Omega\times\mathbb{R}^{+}\setminus S\right)  \cap L^{1}\left(  \Omega
\times\left[  0,T\right)  \right)  $ for any $T<\infty.$ Moreover, for
$a.e.\;t_{0}\in\left[  0,\infty\right)  $ the set $S_{t_{0}}\equiv
S\cap\left\{  \left(  x,t\right)  :t=t_{0}\right\}  $ contains at most $N$
points and the measures $\mu_{t},\ \mu_{t}^{-}$ can be represented as:
\begin{align}
\mu_{t} &  =\sum_{x_{j}\left(  t\right)  \in S_{t}}\alpha_{j}\left(  t\right)
\delta_{x_{j}\left(  t\right)  }+u\left(  \cdot,t\right)
dx\;\;,\;\;a.e.\;t\in\left[  0,\infty\right)  \label{T1E1c}\\
\mu_{t}^{-} &  =\sum_{x_{j}\left(  t\right)  \in S_{t}}\beta_{j}^{-}\left(
t\right)  \delta_{x_{j}\left(  t\right)  }+u\left(  \cdot,t\right)
dx\;\;,\;\;a.e.\;t\in\left[  0,\infty\right)  \label{T1E1d}%
\end{align}
with $\alpha_{j}\left(  t\right)  >0,\ \alpha_{j}\left(  t\right)  \geq
\beta_{j}^{-}\left(  t\right)  ,\;u\left(  \cdot,t\right)  \in L^{1}\left(
\bar{\Omega}\right)  ,\;\int_{\bar{\Omega}}u\left(  x,t\right)  dx\leq
\int_{\Omega}u_{0}\left(  x\right)  dx.$
\end{proposition}

\begin{proof}
We will just make the arguments for points at the interior of $\Omega,$ since
in the case of boundary points the arguments are similar. We define the family
of measures $\mu^{\varepsilon}\in M^{+}\left(  \bar{\Omega}\times
\mathbb{R}^{+}\right)  $ by means of:
\[
\mu^{\varepsilon}\left(  B\right)  =\int_{B}u^{\varepsilon}\left(  x,t\right)
dxdt
\]
for any Borel set $B\subset\bar{\Omega}\times\left[  0,\infty\right)  .$
Taking a subsequence we have $\mu^{\varepsilon_{k}}\rightharpoonup\mu$ as
$k\rightarrow\infty.$ Using the mass conservation property for (\ref{S2E4}),
(\ref{S2E5}) we then have:%
\begin{equation}
\left\vert \int_{\Omega\times A}\varphi d\mu\right\vert \leq\left\Vert
u_{0}\right\Vert _{L^{1}\left(  \Omega\right)  }\int_{A}\left\Vert
\varphi\left(  \cdot,t\right)  \right\Vert _{L^{\infty}\left(  \Omega\right)
}dt\ \ .\ \label{RN2}%
\end{equation}

For any $\varphi\in C_{0}\left(  \bar{\Omega}\times\mathbb{R}^{+}\right)  $ we
define a signed measure $\omega_{\varphi}\in M\left(  \mathbb{R}^{+}\right)  $
by means of:
\begin{equation}
\omega_{\varphi}\left(  A\right)  =\int_{A}\int_{\Omega}\varphi\left(
x,t\right)  d\mu\ \ . \label{RN2a}%
\end{equation}

Notice that (\ref{RN2}) implies:
\begin{equation}
\left\vert \omega_{\varphi}\left(  A\right)  \right\vert \leq\left\Vert
u_{0}\right\Vert _{L^{1}\left(  \Omega\right)  }\left\Vert \varphi\right\Vert
_{L^{\infty}\left(  \Omega\times\mathbb{R}^{+}\right)  }\int_{A}dt\ \ .
\label{RN3}%
\end{equation}

Similar estimates hold for $\omega_{\varphi}^{+},\;\omega_{\varphi}^{-}.$
Therefore the measure $\left\vert \omega_{\varphi}\right\vert =\omega
_{\varphi}^{+}+\omega_{\varphi}^{-}$ is absolutely continuous with respect to
the Lebesgue measure in $\left[  0,\infty\right)  $. It then follows from the
Radon-Nikodym theorem that:
\begin{equation}
\omega_{\varphi}\left(  A\right)  =\int_{A}g_{\varphi}dt \label{RN3a}%
\end{equation}
for some $g_{\varphi}\in L^{1}\left(  \mathbb{R}^{+}\right)  .$ Moreover, due
to (\ref{RN3}) we have
\begin{equation}
\left\Vert g_{\varphi}\right\Vert _{L^{\infty}\left(  \mathbb{R}^{+}\right)
}\leq\left\Vert u_{0}\right\Vert _{L^{1}\left(  \Omega\right)  }\left\Vert
\varphi\right\Vert _{L^{\infty}\left(  \Omega\times\mathbb{R}^{+}\right)  }
\label{RN4}%
\end{equation}

Notice that:
\begin{equation}
g_{\alpha_{1}\varphi_{1}+\alpha_{2}\varphi_{2}}=\alpha_{1}g_{\varphi_{1}%
}+\alpha_{2}g_{\varphi_{2}} \label{RN5}%
\end{equation}
for any $\alpha_{1},\alpha_{2}\in\mathbb{R}$ and $\varphi_{1},\varphi_{2}\in
C_{0}\left(  \bar{\Omega}\times\mathbb{R}^{+}\right)  .$

Let us consider a countable dense linear space $\mathcal{S}\subset C\left(
\bar{\Omega}\right)  .$ For any $\psi\in\mathcal{S},$ $T>0$ we define
$\varphi_{T}=\psi\zeta\left(  t-T\right)  ,$ $T>0,\;$with $\zeta\in C^{\infty
}\left(  \mathbb{R}\right)  ,$ $\zeta\left(  t\right)  =1,\;t\in\left(
-\infty,0\right]  ,\;\zeta\left(  t\right)  =0$ for $t\in\left[
1,\infty\right)  .$ We define also:
\[
L_{t,T}\left[  \psi\right]  =g_{\varphi_{T}}\left(  t\right)  \;\;,\;\;t\in
\mathcal{U}%
\]
for a set $\mathcal{U}\subset\left[  0,\infty\right)  $ whose complement has
zero measure.

Due to (\ref{RN4}), (\ref{RN5}) it follows that for each $t\in\mathcal{U},$
$L_{t,T}$ is a continuous linear functional on $\mathcal{S}$ that can be
extended by density to a continuous linear functional in $C\left(  \bar
{\Omega}\right)  .$ On the other hand, we have:
\[
L_{t,T_{1}}\left[  \psi\right]  =L_{t,T_{2}}\left[  \psi\right]
\;\;,\;\;t\in\left[  0,T_{1}\right]  \cap\mathcal{U},\text{\ }T_{1}\leq T_{2}%
\]
due to (\ref{RN2a}), (\ref{RN3a}). We will denote $L_{t}\left[  \psi\right]
=\lim_{T\rightarrow\infty}L_{t,T}\left[  \psi\right]  .$ This defines a family
of continuous linear functionals $L_{t}$ for $a.e.\;t\in\left[  0,\infty
\right)  .$ Therefore, Riesz-Markov Theorem implies that there exists a family
of signed measures $\mu_{t}\in M\left(  \bar{\Omega}\right)  $ defined
$a.e.\;t\in\left[  0,\infty\right)  $ such that:
\begin{align*}
L_{t}\left[  \psi\right]   &  =\int_{\Omega}\psi d\mu_{t}\;\;,\;\;\psi\in
C\left(  \bar{\Omega}\right)  \;,\;a.e.\;t\in\left[  0,\infty\right) \\
\left\vert \mu_{t}\right\vert  &  \leq\left\Vert u_{0}\right\Vert
_{L^{1}\left(  \Omega\right)  }\;\;,\;a.e.\;t\in\left[  0,\infty\right)
\end{align*}

We now remark that for any $\varepsilon>0,\;t_{0}\in\left[  0,\infty\right)  $
and any smooth cutoff function $\eta$ such that $\eta\left(  t\right)
=1,\;t\in\left[  t_{0}-\varepsilon,t_{0}+\varepsilon\right]  ,\;\eta\left(
t\right)  =0,\;\left\vert t-t_{0}\right\vert \geq2\varepsilon$ and any
$\psi\in C\left(  \bar{\Omega}\right)  $ we have:
\[
g_{\eta\psi}\left(  t\right)  =g_{\varphi_{T}}\left(  t\right)  =L_{t}\left[
\psi\right]  =\int_{\Omega}\psi d\mu_{t}\;\;,\;\;t\in\left[  t_{0}%
-\varepsilon,t_{0}+\varepsilon\right]  \cap\mathcal{U}\;
\]
for $T$ sufficiently large. This is just a consequence of (\ref{RN2a}),
(\ref{RN3a}). A density argument then yields:
\[
\int_{\left[  0,\infty\right)  }\int_{\bar{\Omega}}\varphi\left(  x,t\right)
d\mu=\int_{\left[  0,\infty\right)  }\int_{\bar{\Omega}}\varphi\left(
x,t\right)  d\mu_{t}dt\;\;,\;\;\varphi\in C_{0}\left(  \bar{\Omega}%
\times\mathbb{R}^{+}\right)
\]
or shortly $d\mu=d\mu_{t}dt.$

Moreover, assuming that $\varphi=\varphi\left(  t\right)  $ and using
\begin{align*}
\int_{\left[  0,\infty\right)  }\varphi\left(  t\right)  \int_{\bar{\Omega}%
}u_{0}\left(  x\right)  dxdt  &  =\int_{\bar{\Omega}\times\left[
0,\infty\right)  }\varphi\left(  t\right)  u^{\varepsilon_{k}}\left(
x,t\right)  dxdt\\
&  =\int_{\bar{\Omega}\times\left[  0,\infty\right)  }\varphi\left(  t\right)
d\mu^{\varepsilon_{k}}\rightarrow\int_{\left[  0,\infty\right)  }\mu
_{t}\left(  \bar{\Omega}\right)  \varphi\left(  t\right)  dt
\end{align*}
it follows that $\mu_{t}\left(  \Omega\right)  =\int_{\Omega}u_{0}\left(
x\right)  dx,$ $a.e.\;t\in\left[  0,\infty\right)  .$

We now define the singular set by means of (\ref{T1E1a}) and decompose $\mu$
as in (\ref{T1E1b}) with $\nu=\mu\chi_{S},$ with $\chi_{S}$ denoting the
characteristic function of $S.$ We now show that $u=\mu-\nu$ is a smooth
function. To this end, notice that by definition of $S$ we have, for any
$\left(  x_{0},t_{0}\right)  \in\left[  \bar{\Omega}\setminus S\right]
\times\left[  0,\infty\right)  $ there exists $\rho=\rho\left(  x_{0}%
,t_{0}\right)  $ such that:
\[
\lim\inf_{\delta\rightarrow0}\frac{1}{\delta}\int_{t_{0}}^{t_{0}+\delta}%
\mu_{t}\left(  B_{4\rho}\left(  x_{0}\right)  \right)  dt<\frac{m_{0}}{2}%
\]

Then, there exists a sequence $\delta_{n}\rightarrow0$ such that:%

\[
\frac{1}{\delta_{n}}\int_{t_{0}}^{t_{0}+\delta_{n}}\mu_{t}\left(  B_{4\rho
}\left(  x_{0}\right)  \right)  dt\leq\frac{3m_{0}}{4}%
\]
and the weak convergence $u^{\varepsilon_{k}}\rightharpoonup\mu$ yields:
\[
\frac{1}{\delta_{n}}\int_{t_{0}}^{t_{0}+\delta_{n}}\int_{B_{4\rho}\left(
x_{0}\right)  }u^{\varepsilon_{k}}\left(  x,t\right)  dxdt\leq m_{0}%
\]
for $k$ sufficiently large. Therefore, there exists a sequence $t_{n,k}%
\in\left[  t_{0},t_{0}+\delta_{n}\right]  $ such that:
\begin{equation}
\int_{B_{4\rho}\left(  x_{0}\right)  }u^{\varepsilon_{k}}\left(
x,t_{n,k}\right)  dx\leq m_{0} \label{T1E2}%
\end{equation}
for $k$ sufficiently large. We can now apply Proposition \ref{I1} to the
functions $u^{\varepsilon_{k}}$ for large $k.$ Indeed, notice that
(\ref{S5E3a}) holds due to (\ref{S3E1}) in Proposition \ref{P1} and, on the
other hand, $\left\|  v^{\varepsilon}\right\|  _{L^{6}\left(  \Omega\right)
}\leq C\int u_{0}$ due to classical regularity theory for the Poisson
equation, whence (\ref{S5E3c}) also holds. Finally (\ref{T1E2}) implies
(\ref{S5E3b}). Proposition \ref{I1} then yields:
\[
\sup_{s\in\left[  t_{n,k}-c_{1}\rho^{2},t_{n,k}\right]  }\int_{B_{\rho}\left(
x_{0}\right)  }\left(  u^{\varepsilon_{k}}\right)  ^{2}\left(  x,t\right)
dx\leq\frac{c_{2}}{\rho^{4}}%
\]
and since $\delta_{n}$ can be assumed to be arbitrarily close to $0$ we have:
\[
\sup_{s\in\left[  t_{0}-\frac{c_{1}\rho^{2}}{2},t_{0}\right]  }\int_{B_{\rho
}\left(  x_{0}\right)  }\left(  u^{\varepsilon_{k}}\right)  ^{2}\left(
x,t\right)  dx\leq\frac{c_{2}}{\rho^{4}}%
\]

Notice that the order of the limits is, first we fix $\rho,$ then we choose
$n$ to have $\delta_{n}\leq\frac{c_{1}\rho^{2}}{2},$ and then $k$ large.

Classical regularity results for parabolic equations (cf. \cite{DB})\texttt{
}then imply that\newline $u^{\varepsilon_{k}}\in C^{\infty}\left(
B_{\frac{\rho}{2}}\left(  x_{0}\right)  \times\left[  t_{0}-\frac{c_{1}%
\rho^{2}}{4},t_{0}\right]  \right)  $ as well as uniform estimates for the
derivatives of $u^{\varepsilon_{k}}$ in the same set. Then, $u\in C^{\infty
}\left(  \left[  \bar{\Omega}\setminus S\right]  \times\mathbb{R}^{+}\right)
.$ Moreover, using the estimate $\int_{\Omega\times\left[  0,T\right]  }u\leq
T\int_{\Omega}u_{0}$ for any $T>0,$ and the positivity of $u,$ we obtain $u\in
L^{1}\left(  \left(  \Omega\times\mathbb{R}^{+}\right)  \right)  $ for any
$T<\infty.$

We now prove that for $a.e.\;t_{0}\in\left[  0,\infty\right)  $ the set
$S_{t_{0}}$ contains at most $N$ points. Suppose that $S_{t_{0}}$ contains at
least $\left(  N+1\right)  $ points $x_{1},...,x_{N+1}.$ Let us choose
$\rho<\frac{1}{2}\min_{i\neq j}\left\{  \left|  x_{i}-x_{j}\right|  \right\}
.$ Due to the definition of the singular set there exists $\delta$ (depending
on $\rho$) such that:
\begin{equation}
\frac{1}{\delta}\mu\left(  B_{\rho}\left(  x_{i}\right)  \times\left[
t_{0},t_{0}+\delta\right]  \right)  =\frac{1}{\delta}\int_{t_{0}}%
^{t_{0}+\delta}\mu_{t}\left(  B_{\rho}\left(  x_{i}\right)  \right)
dt\geq\frac{m_{0}}{4} \label{I1E1}%
\end{equation}
for $i=1,...,\left(  N+1\right)  .$

Since the balls $\left\{  B_{\rho}\left(  x_{i}\right)  \right\}  _{i=1}%
^{N+1}$ are disjoint, we have:
\begin{align*}
\frac{1}{\delta}\sum_{i=1}^{N+1}\mu\left(  B_{\rho}\left(  x_{i}\right)
\times\left[  t_{0},t_{0}+\delta\right]  \right)  =\frac{1}{\delta}\mu\left(
\left[  \bigcup_{i=1}^{N+1}B_{\rho}\left(  x_{i}\right)  \right]
\times\left[  t_{0},t_{0}+\delta\right]  \right)  \leq\\
\leq\frac{1}{\delta}\mu\left(  \bar{\Omega}\times\left[  t_{0},t_{0}%
+\delta\right]  \right)  =\frac{1}{\delta}\int_{t_{0}}^{t_{0}+\delta}\mu
_{t}\left(  \bar{\Omega}\right)  dt=\int_{\bar{\Omega}}u_{0}dx
\end{align*}
due to (\ref{T1E1e}).

Using (\ref{T1E1}), (\ref{I1E1}) it then follows that:
\[
\frac{m_{0}\left(  N+1\right)  }{4}\leq\int_{\bar{\Omega}}u_{0}dx<\frac
{m_{0}\left(  N+1\right)  }{4}%
\]
that yields a contradiction. Then $S_{t_{0}}$ contains at most $N$ points for
$a.e.\;t_{0}\in\left[  0,\infty\right)  .$

Since a measure concentrated in a finite set of points is a sum of Dirac
masses\texttt{ }it then follows that:
\[
\mu_{t}\chi_{S_{t_{0}}}=\sum_{x_{j}\left(  t\right)  \in S_{t}}\alpha
_{j}\left(  t\right)  \delta_{x_{i}\left(  t\right)  }\;\;,\;\;a.e.\;t\in
\left[  0,\infty\right)
\]
for at most $N$ points $\left\{  x_{j}\left(  t\right)  \right\}  $ and
positive numbers $\left\{  \alpha_{j}\left(  t\right)  \right\}  .$ On the
other hand, if\ $\left(  x_{0},t_{0}\right)  $ is not in the singular set, we
can represent $\mu$ in $B_{\frac{\rho}{2}}\left(  x_{0}\right)  \times\left[
t_{0}-\frac{c_{1}\rho^{2}}{4},t_{0}\right]  $ as a smooth function. Then:
\[
\mu_{t}\left(  1-\chi_{S_{t_{0}}}\right)  =u\left(  \cdot,t\right)  dx
\]
and:
\[
\mu_{t}=\sum_{x_{j}\left(  t\right)  \in S_{t}}\alpha_{j}\left(  t\right)
\delta_{x_{i}\left(  t\right)  }+u\left(  \cdot,t\right)
dx\;\;,\;\;a.e.\;t\in\left[  0,\infty\right)
\]

This gives (\ref{T1E1c}) and concludes the proof of Proposition \ref{M1}.

The convergence of a subsequence $\left\{  f_{\varepsilon_{k}}\left(
u^{\varepsilon_{k}}\right)  \right\}  $ to $\mu^{-}$ as well as the properties
of this measure can be obtained exactly as for the measure $\mu.$ The property
$\mu^{-}\leq\mu$ is just a consequence of the inequalities $f_{\varepsilon
}\left(  u^{\varepsilon}\right)  \leq u^{\varepsilon}.$ The fact that $\mu
^{-}$ can be represented by means of $u$ at the regular points is a
consequence of the fact that $u^{\varepsilon}$ is bounded at the regular set,
and therefore $f_{\varepsilon}\left(  u^{\varepsilon}\right)  \rightarrow u,$ there.
\end{proof}

\bigskip

We now prove a result similar to Proposition \ref{M1} for the regularization
(\ref{S2E4}), (\ref{S2E5}).

\begin{proposition}
\label{M2}Suppose that $u^{\varepsilon}$ is the solution of (\ref{S2E4}),
(\ref{S2E5}) with initial value $u^{\varepsilon}\left(  x,0\right)
=u_{0}\left(  x\right)  ,\;x\in\Omega$ and $\varepsilon>0.$ Then, there exist
a Radon measure $\mu\in M^{+}\left(  \Omega\times\mathbb{R}^{+}\right)  $ and
a subsequence $\left\{  \varepsilon_{k}\right\}  _{k=1}^{\infty}$ such that:
\[
u^{\varepsilon_{k}}\rightharpoonup\mu\;\ ,\ \ u^{\varepsilon_{k}}%
+\varepsilon_{k}\left(  u^{\varepsilon_{k}}\right)  ^{\frac{7}{6}%
}\rightharpoonup\mu^{+}\ \ \text{as\ \ }k\rightarrow\infty\ \ ,\ \ \mu\leq
\mu^{+}%
\]
in the $\ast-$weak topology. Moreover, the measures $\mu,\ \mu^{+}$ can be
written as the product:
\[
d\mu=d\mu_{t}dt\ \ \ ,\ \ \ d\mu^{+}=d\mu_{t}^{+}dt\ \ ,\ \ \mu_{t}\leq\mu
_{t}^{+}%
\]
where the family of Radon measures $\left\{  \mu_{t}\right\}  _{t\geq0}$
satisfy $a.e.\;t\in\left[  0,\infty\right)  :$%
\[
\mu_{t}\left(  \Omega\right)  =\int_{\Omega}u_{0}\left(  x\right)  dx\
\]
We define the singular set of $\mu,$ and denote it as $S,$ as the set of
points $\left(  x_{0},t_{0}\right)  \in\bar{\Omega}\times\left[
0,\infty\right)  $ where:
\[
\lim_{\delta\rightarrow0}\inf\left[  \frac{\mu\left(  B_{\rho}\left(
x_{0}\right)  \times\left[  t_{0},t_{0}+\delta\right]  \right)  }{\delta
}\right]  =\lim_{\delta\rightarrow0}\inf\left[  \frac{1}{\delta}\int_{t_{0}%
}^{t_{0}+\delta}\mu_{t}\left(  B_{\rho}\left(  x_{0}\right)  \right)
dt\right]  \geq\frac{m_{0}}{2}\;\;\text{for any }\rho>0
\]
with $m_{0}$ as in (\ref{S5E3b}). Then, we can write:
\begin{equation}
\mu=\bar{\mu}+u\ \ \ \ ,\ \ \ \ \mu^{+}=\bar{\mu}^{+}+u\label{TN1}%
\end{equation}
where $\bar{\mu},\bar{\mu}^{+}\in M^{+}\left(  \Omega\times\mathbb{R}%
^{+}\right)  $ are supported in the set $S$ and $u\in C^{\infty}\left(
\Omega\times\mathbb{R}^{+}\setminus S\right)  \cap L^{1}\left(  \Omega
\times\left[  0,T\right)  \right)  $ for any $T<\infty.$ Moreover, for
$a.e.\;t_{0}\in\left[  0,\infty\right)  $ the set $S_{t_{0}}\equiv
S\cap\left\{  \left(  x,t\right)  :t=t_{0}\right\}  $ contains at most $N$
points and the measure $\mu_{t}$ can be represented as:
\begin{equation}
\mu_{t}=\sum_{x_{j}\left(  t\right)  \in S_{t}}\alpha_{j}\left(  t\right)
\delta_{x_{j}\left(  t\right)  }+u\left(  \cdot,t\right)
dx\;\;,\;\;a.e.\;t\in\left[  0,\infty\right)  \label{T2E1}%
\end{equation}%
\begin{equation}
\mu_{t_{0}}^{+}=\sum_{x_{j}\left(  t\right)  \in S_{t}}\beta_{j}^{+}\left(
t\right)  \delta_{x_{j}\left(  t\right)  }+u\left(  \cdot,t\right)
dx\;\;,\;\;a.e.\;t\in\left[  0,\infty\right)  \label{T2E1a}%
\end{equation}
with $\beta_{j}^{+}\left(  t\right)  \geq\alpha_{j}\left(  t\right)
>0,\;u\left(  \cdot,t\right)  \in L^{1}\left(  \bar{\Omega}\right)
,\;\int_{\bar{\Omega}}u\left(  x,t\right)  dx\leq\int_{\Omega}u_{0}\left(
x\right)  dx.$
\end{proposition}

\begin{proof}
Arguing as in the proof of Proposition \ref{M1} it follows that $d\mu=d\mu
_{t}dt$ as well as $\mu_{t}\left(  \Omega\right)  =\int_{\Omega}u_{0}\left(
x\right)  dx.$ We have also that for $\left(  x_{0},t_{0}\right)  \in\left[
\bar{\Omega}\times\left[  0,\infty\right)  \right]  \setminus S$ there exists
$\rho=\rho\left(  x_{0},t_{0}\right)  >0$ such that:
\begin{equation}
\lim_{\delta\rightarrow0}\inf\left[  \frac{1}{\delta}\int_{t_{0}}%
^{t_{0}+\delta}\mu_{t}\left(  B_{4\rho}\left(  x_{0}\right)  \right)
dt\right]  \leq\frac{m_{0}}{2} \label{T1E4}%
\end{equation}

Arguing as in the proof of Proposition \ref{M1} we can apply Proposition
\ref{I1} to the functions $u^{\varepsilon_{k}}$ for large $k.$ Indeed, notice
that (\ref{S6E3}) holds due to (\ref{S3E2}) in Proposition \ref{P1} and, on
the other hand, $\left\Vert v^{\varepsilon}\right\Vert _{L^{6}\left(
\Omega\right)  }\leq C\int u_{0}$ due to classical regularity theory for the
Poisson equation, whence (\ref{S6E6}) also holds. The entropy estimate
(\ref{S7E1}) implies (\ref{S6E5}).\texttt{ }Finally (\ref{T1E4}) implies the
existence of a subsequence $t_{n,k}\downarrow t_{0}$ such that $\int
_{B_{4\rho}\left(  x_{0}\right)  }u^{\varepsilon_{k}}\left(  x,t_{n,k}\right)
dx\leq m_{0}$ \ (cf. (\ref{S6E4})). Proposition \ref{I2} as well as the fact
that $t_{n,k}\downarrow t_{0}$ then yields:
\[
\sup_{s\in\left[  t_{0}-c_{1}\rho^{2},t_{0}\right]  }\int_{B_{\rho}\left(
x\right)  }\left(  u^{\varepsilon}\right)  ^{2}\left(  y,s\right)  dy\leq
\frac{c_{2}}{\rho^{4}}\ \ .
\]

Classical regularity results for parabolic equations\texttt{ }then imply that
$u^{\varepsilon}\in C^{\infty}\left(  B_{\frac{\rho}{2}}\left(  x_{0}\right)
\times\left[  t_{0}-\frac{c_{1}\rho^{2}}{2},t_{0}\right]  \right)  $ as well
as uniform estimates for the derivatives of $u^{\varepsilon}$ in
$B_{\frac{\rho}{2}}\left(  x\right)  \times\left[  t_{0}-\frac{c_{1}\rho^{2}%
}{2},t_{0}\right]  .$ Then, $u\in C^{\infty}\left(  \left[  \bar{\Omega}%
\times\mathbb{R}^{+}\right]  \setminus S\right)  $ (cf. \cite{KSu}). The
uniform estimate $\int_{\Omega}u\leq\int_{\Omega}u_{0}$ and the positivity of
$u$ imply that $u\in L^{1}\left(  \bar{\Omega}\times\mathbb{R}^{+}\right)  .$
We can now prove as in Proposition \ref{M1} that for $a.e.\;t_{0}\in\left[
0,\infty\right)  $ the set $S_{t_{0}}$ contains at most $N$ points. The
measures $\mu_{t}$ have then the structure (\ref{T2E1}). This concludes the
proof of Proposition \ref{M2}.

Finally we prove the convergence properties of $u^{\varepsilon}+\varepsilon
\left(  u^{\varepsilon}\right)  ^{\frac{7}{6}}.$ To this end we need to obtain
an estimate for $\int\int\varepsilon\left(  u^{\varepsilon}\right)  ^{\frac
{7}{6}}dxdt.$ This can be obtained as follows. Suppose that $\left(
x_{0},t_{0}\right)  \in\bar{\Omega}\times\left[  0,T\right]  ,$ with
$0<T<\infty.$ Either $\left(  x_{0},t_{0}\right)  \in S$ or $\left(
x_{0},t_{0}\right)  \in\bar{\Omega}\setminus S.$ In the second case, there
exists $\rho=\rho\left(  x_{0},t_{0}\right)  $ such that $u^{\varepsilon}$ is
bounded by some $C=C\left(  x_{0},t_{0},\rho\right)  $ but independent on
$\varepsilon$ for $\varepsilon\leq\varepsilon_{0}\left(  x_{0},t_{0}%
,\rho\right)  $ in $B_{2\rho\left(  x_{0},t_{0}\right)  }\left(  x_{0}\right)
\times\left[  t_{0}-2c\rho^{2},t_{0}\right]  $ (cf. Proposition \ref{I2})$.$
Then, $u$ is smooth in $B_{\rho\left(  x_{0},t_{0}\right)  }\left(
x_{0}\right)  \times\left[  t_{0}-c\rho^{2},t_{0}\right]  $ and we have:
\begin{equation}
\int_{t_{0}-c\rho^{2}}^{t_{0}}\int_{B_{\rho\left(  x_{0},t_{0}\right)
}\left(  x_{0}\right)  }\left[  u^{\varepsilon}\left(  x,t\right)
+\varepsilon\left(  u^{\varepsilon}\left(  x,t\right)  \right)  ^{\frac{7}{6}%
}\right]  dxdt\leq C\left(  x_{0},t_{0},\rho\right)  \label{H1E2}%
\end{equation}
with $C$ independent on $\varepsilon.$

Suppose that, on the contrary, $\left(  x_{0},t_{0}\right)  \in S.$ Since the
number of points in $S_{t_{0}}$ is finite, we can choose $\rho>0,$ such that
$\mathcal{D}=\left(  \left[  \overline{B_{2\rho}\left(  x_{0}\right)
}\setminus B_{\rho}\left(  x_{0}\right)  \right]  \times\left\{
t_{0}\right\}  \right)  \cap S=\emptyset$. It then follows that for any
$\left(  \bar{x},t_{0}\right)  \in\mathcal{D}$ there exists $\delta
=\delta\left(  \bar{x},t_{0}\right)  $ such that $u^{\varepsilon}$\ is
uniformly bounded in $B_{\delta\left(  \bar{x},t_{0}\right)  }\left(  \bar
{x}\right)  \times\left[  t_{0}-2c\delta^{2},t_{0}\right]  .$ Since
$\mathcal{D}$ is compact we can cover it with a finite number of sets of the
form $B_{\delta\left(  \bar{x},t_{0}\right)  }\left(  \bar{x},t_{0}\right)
\times\left[  t_{0}-2c\delta^{2},t_{0}\right]  .$ Therefore, there exists
$\delta_{0}=\delta_{0}\left(  x_{0},t_{0},\rho\right)  >0$ such that
$u^{\varepsilon}$ is uniformly bounded in $\mathcal{D}\times\left[
t_{0}-\delta_{0},t_{0}\right]  $ for $\varepsilon\leq\varepsilon_{0}\left(
x_{0},t_{0},\rho\right)  .$ We now consider a test function $\psi_{\rho}$ such
that $\Delta\psi_{\rho}=-\frac{1}{\rho^{2}}$ for $\left\vert x-x_{0}%
\right\vert \leq\rho$ and $\left\vert \Delta\psi_{\rho}\right\vert \leq
C_{\rho}$ for $\rho\leq\left\vert x-x_{0}\right\vert \leq2\rho,$ and
$\psi_{\rho}=0$ for $\left\vert x-x_{0}\right\vert \geq2\rho.$

Using $\psi_{\rho}$ as test function we obtain the following estimate:
\begin{equation}
\left\vert \partial_{t}\left(  \int_{B_{2\rho}\left(  x_{0}\right)
}u^{\varepsilon}\left(  x,t\right)  \psi_{\rho}\left(  x\right)  dx\right)
+\frac{1}{\rho^{2}}\int_{B_{\rho}\left(  x_{0}\right)  }\left[  u^{\varepsilon
}\left(  x,t\right)  +\varepsilon\left(  u^{\varepsilon}\left(  x,t\right)
\right)  ^{\frac{7}{6}}\right]  dx\right\vert \leq C_{\rho}\;,\;t\in\left[
t_{0}-\delta_{0},t_{0}\right]  \label{H1E1}%
\end{equation}
where the error term on the right is due to the contribution of $\int
_{\mathcal{D}}\Delta\psi_{\rho}\left(  x\right)  \left[  u^{\varepsilon
}\left(  x,t\right)  \right]  dx,$ as well as the nonlinear terms that can be
estimated using the symmetrization argument as in the Proof of Proposition
\ref{I2}. Notice that we use the smoothness of $u^{\varepsilon}$ in
$\mathcal{D}\times\left[  t_{0}-\delta_{0},t_{0}\right]  .$ Therefore,
integrating (\ref{H1E1}) in $\left[  t_{0}-\delta_{0},t_{0}\right]  $:
\begin{equation}
\int_{t_{0}-\delta_{0}}^{t_{0}}\int_{B_{\rho}\left(  x_{0}\right)  }\left[
u^{\varepsilon}\left(  x,t\right)  +\varepsilon\left(  u^{\varepsilon}\left(
x,t\right)  \right)  ^{\frac{7}{6}}\right]  dxdt\leq C_{\rho} \label{H1E3}%
\end{equation}

Therefore, (\ref{H1E2}), (\ref{H1E3}) imply the existence for any $\left(
x_{0},t_{0}\right)  \in\Omega\times\left[  0,T\right]  $ of a cylinder
$B_{\rho}\left(  x_{0}\right)  \times\left[  t_{0}-\delta_{0},t_{0}\right]  $
for which (\ref{H1E3}) holds. Since $\Omega$ is compact, we can find a finite
covering of it by means of some of these cylinders. Then:
\[
\int_{0}^{T}\int_{\Omega}\left[  u^{\varepsilon}\left(  x,t\right)
+\varepsilon\left(  u^{\varepsilon}\left(  x,t\right)  \right)  ^{\frac{7}{6}%
}\right]  dxdt\leq C
\]
with $C>0$ independent on $\varepsilon,$ assuming that $\varepsilon
\leq\varepsilon_{0}.$

We then have:
\[
u^{\varepsilon_{k}}+\varepsilon_{k}\left(  u^{\varepsilon_{k}}\right)
^{\frac{7}{6}}\rightharpoonup\mu^{+}\text{\ \ as\ \ }k\rightarrow\infty\ \ .
\]

The fact that the singular set of $\mu^{+}$ and $\mu$ are the same is a
consequence of the fact that for every regular point of $\mu$ we have
$\varepsilon_{k}\left(  u^{\varepsilon_{k}}\right)  ^{\frac{7}{6}}%
\rightarrow0$ in a neighbourhood of the regular point, as it can be seen from
the estimate (\ref{S6E7}).
\end{proof}

\begin{remark}
We have denoted the limit of the sequence $u^{\varepsilon}$ as $\mu$ for both
regularizations. Notice, however, that there is not any reason to expect both
limits to be the same measure.
\end{remark}

\bigskip

\subsection{A continuity result for the singular set.}

\begin{lemma}
\label{contS} Suppose that the measure $\mu_{t}$ and the set $S_{t}$ are as in
Proposition \ref{M1} or Proposition \ref{M2}. Let $T>0.$ For any
$\varepsilon>0$ there exists $\delta>0$ such that for any $t_{0}\in\left[
0,T\right]  ,$ $S_{t}\subset S_{t_{0}}+B_{\varepsilon}\left(  0\right)  $ if
$t\in\left[  t_{0}-\delta,t_{0}\right]  .$ Moreover, suppose that for $\left(
x_{0},t_{0}\right)  \in S$ we have $\int_{B_{R}\left(  x_{0}\right)
\setminus\left\{  x_{0}\right\}  }d\mu_{t_{0}}\leq\frac{m_{0}}{2},$ with
$m_{0}$ as in Propositions \ref{I1}, \ref{I2} and $R>0$ fixed. Then, there
exists $c>0,\ L>0$ depending only on $\int_{\Omega}d\mu_{0}\left(  x\right)  $
and $\Omega$ such that $S_{t}\cap B_{R}\left(  x_{0}\right)  \subset
B_{L\sqrt{\left\vert t-t_{0}\right\vert }}\left(  x_{0}\right)  $ for
$t\in\left[  t_{0}-cR^{2},t_{0}\right]  .$
\end{lemma}

\begin{proof}
It is just a consequence of Propositions \ref{I1}, \ref{I2}, \ref{BI1},
\ref{BI2}.\bigskip
\end{proof}

We include now some auxiliary results that will be required later.

\subsection{Mass continuity. Characterization of the limit of some quadratic
terms.}

A basic characteristic of the regularizations of the Keller-Segel system in
(\ref{S2E1})-(\ref{S2E2}), (\ref{S2E4})-(\ref{S2E5}) is the fact that the mass
in each neighbourhood changes in a continuous way. More precisely, we have the
following result:

\begin{lemma}
\label{MassChange}Suppose that $0<T<\infty.$ Let $u^{\varepsilon}$ be the
solution of (\ref{S2E1})-(\ref{S2E2}) or (\ref{S2E4})-(\ref{S2E5}). Given
$\left(  x_{0},t_{0}\right)  \in\bar{\Omega}\times\left[  0,T\right]  $ there
exists a cutoff function $\psi\left(  x;x_{0},t_{0}\right)  $ and $\rho
=\rho\left(  x_{0},t_{0}\right)  ,\ \tau=\tau\left(  x_{0},t_{0}\right)  $
independent on $\varepsilon$ satisfying:%
\begin{equation}
\left\vert \partial_{t}\left(  \int_{B_{2\rho}\left(  x_{0}\right)
}u^{\varepsilon}\left(  y,t\right)  \psi\left(  y;x_{0},t_{0}\right)
dy\right)  \right\vert \leq C\left(  x_{0},t_{0}\right)  \label{M5E1}%
\end{equation}
for $t\in\left[  t_{0}-\tau,t_{0}\right]  $ with $C\left(  x_{0},t_{0}\right)
$ independent on $\varepsilon.$
\end{lemma}

\begin{proof}
We just consider interior points, since boundary estimates can be obtained
similarly using Lemma \ref{L2}. If $\left(  x_{0},t_{0}\right)  \in\left(
\Omega\times\left[  0,T\right]  \right)  \setminus S$ there exists $\rho
=\rho\left(  x_{0},t_{0}\right)  >0,\ \tau=\tau\left(  x_{0},t_{0}\right)  $
such that $B_{2\rho}\left(  x_{0}\right)  \times\left[  t_{0}-\tau
,t_{0}\right]  \in\left(  \Omega\times\left[  0,T\right]  \right)  \setminus
S$ (cf. Propositions \ref{I1}, \ref{I2}). We take a cutoff function
$\psi\left(  y\right)  $ such that $\psi\left(  y\right)  =1$ for $\left\vert
y-x_{0}\right\vert \leq\rho$ and $\psi\left(  y\right)  =0$ for $\left\vert
y-x_{0}\right\vert \geq2\rho.$ Then, the estimates for the linear terms
$\Delta\left(  u^{\varepsilon}\right)  ,\ \Delta\left(  u^{\varepsilon
}+\varepsilon u^{\varepsilon}\right)  $ are immediate and the nonlinear terms
can be estimated using the symmetrization argument in the proof of
Propositions \ref{I1}, \ref{I2}. If $\left(  x_{0},t_{0}\right)  \in S$ we
have $\left[  B_{2\rho}\left(  x_{0}\right)  \setminus B_{\rho}\left(
x_{0}\right)  \right]  \times\left[  t_{0}-\tau,t_{0}\right]  \cap
S=\emptyset$ if $\rho>0$ and $\tau>0$ are sufficiently small (see Lemma
\ref{contS}). The result then follows choosing $\psi\left(  y\right)  $ with
$\nabla\psi\neq0$ only in $\left[  B_{2\rho}\left(  x_{0}\right)  \setminus
B_{\rho}\left(  x_{0}\right)  \right]  .$
\end{proof}

\bigskip

Lemma \ref{MassChange} allows to obtain the weak limit of some quadratic terms.

\begin{lemma}
Suppose that $u^{\varepsilon}$ is as in Lemma \ref{MassChange} and that, for
suitable subsequences $u^{\varepsilon_{k}}\rightharpoonup\mu,$ with $\mu$ as
in Proposition \ref{M1} or Proposition \ref{M2}. Let $\varphi\in C\left(
\bar{\Omega}\times\bar{\Omega}\times\mathbb{R}^{+}\right)  .$ Then:%
\begin{equation}
\int\int\int u^{\varepsilon_{k}}\left(  x,t\right)  u^{\varepsilon_{k}}\left(
y,t\right)  \varphi\left(  x,y,t\right)  dxdydt\rightarrow\int\int\int
\varphi\left(  x,y,t\right)  d\mu_{t}\left(  x\right)  d\mu_{t}\left(
y\right)  dt \label{DouLim}%
\end{equation}
as $k\rightarrow\infty.$
\end{lemma}

\begin{proof}
The compactness of $\bar{\Omega}\times\left[  0,T\right]  $ implies the
existence, for each $\delta_{0}>0$ of functions $\left\{  \psi_{\ell,m}\left(
y\right)  \right\}  $ satisfying $\sum_{\ell}\psi_{\ell,m}\left(  y\right)
=1$ for $y\in\bar{\Omega},$ and such that:
\begin{equation}
\left\vert \partial_{t}\left(  \int u^{\varepsilon}\left(  y,t\right)
\psi_{\ell,m}\left(  y\right)  dy\right)  \right\vert \leq C\;\;\;,\;\;t\in
\left[  m\delta_{0},\left(  m+1\right)  \delta_{0}\right]  \ \ ,\ \ m\delta
_{0}<T \label{H3E1}%
\end{equation}
whence:%
\begin{equation}
\left\vert \int u^{\varepsilon}\left(  y,t\right)  \psi_{\ell,m}\left(
y\right)  dy-\int u^{\varepsilon}\left(  y,t_{m}\right)  \psi_{\ell,m}\left(
y\right)  dy\right\vert \leq C\left\vert t-t_{m}\right\vert \;\;,\;\;t\in
\left[  m\delta_{0},\left(  m+1\right)  \delta_{0}\right]  \label{H3E2}%
\end{equation}

Combining Propositions \ref{M1} and \ref{M2} and (\ref{H3E1}) we obtain the
existence of $t_{m}\in\left[  m\delta_{0},\left(  m+1\right)  \delta
_{0}\right]  $ for any $m=0,1,...$ such that:%

\begin{equation}
\int u^{\varepsilon}\left(  y,t_{m}\right)  \psi_{\ell,m}\left(  y\right)
dy\rightarrow\int\psi_{\ell,m}\left(  y\right)  d\mu_{t_{m}} \label{H3E6}%
\end{equation}

We rewrite the left-hand side of (\ref{DouLim}) as:
\[
\int\int\int u^{\varepsilon}\left(  x,t\right)  u^{\varepsilon}\left(
y,t\right)  \varphi\left(  x,y,t\right)  dxdydt=\sum_{m=0}^{\left[  \frac
{T}{\delta_{0}}\right]  }\int_{m\delta_{0}}^{\left(  \left(  m+1\right)
\delta_{0}\right)  \wedge T}dt\sum_{\ell}\int dxu^{\varepsilon}\left(
x,t\right)  \int u^{\varepsilon}\left(  y,t\right)  \varphi\left(
x,y,t\right)  dy
\]

Using the continuity of $\varphi$ and (\ref{H3E2}) we then obtain for any
$\sigma_{0}>0:$%
\[
\left\vert \int\int\int u^{\varepsilon}\left(  x,t\right)  u^{\varepsilon
}\left(  y,t\right)  \varphi\left(  x,y,t\right)  dxdydt-\int_{0}^{T}dt\int
d\mu_{t}\left(  x\right)  \int d\mu_{t}\left(  y\right)  \varphi\left(
x,y,t\right)  dy\right\vert \leq C\left(  \sigma_{0}+\delta_{0}\right)
\]
with $C$ independent on $\sigma_{0},\delta_{0}$ whence the result
follows.\bigskip
\end{proof}

\section{Formulation of the limit problem.}

We now pass to the limit in the problems (\ref{S2E1})-(\ref{S2E2}),
(\ref{S2E4})-(\ref{S2E5}) to derive the problems satisfied by the pairs of
measures $\left(  \mu,\mu^{-}\right)  ,\ \left(  \mu,\mu^{+}\right)  $
respectively. As a matter of fact, in order to write the weak equations
satisfied by the measures $\mu$ we will need to introduce some auxiliary
measures $\hat{\mu}^{-},\hat{\mu}$ defined in a space larger than $\bar
{\Omega}\times\left[  0,\infty\right)  $ at the singular points. We begin with
ther first regularization (\ref{S2E1})-(\ref{S2E2}).

\subsection{First regularization.}

We begin rewriting the weak formulation of the regularized problem
(\ref{S2E1})-(\ref{S2E2}) in a more convenient form:

\begin{lemma}
\label{Lterms}Suppose that $u^{\varepsilon}$ solves (\ref{S2E1})-(\ref{S2E2}).
Then, for any test function $\psi\in C^{\infty}\left(  \bar{\Omega}%
\times\mathbb{R}^{+}\right)  $ satisfying%
\begin{equation}
\frac{\partial\psi}{\partial\nu_{x}}\left(  x,t\right)  =0\text{ \ ,\ \ }%
x\in\partial\Omega\label{G2E4}%
\end{equation}
we have:%
\begin{equation}
L_{1}+Q_{1}+Q_{2}+Q_{3}+Q_{4}+Q_{5}=0\ \ ,\ \label{G2E3}%
\end{equation}
where:
\begin{align}
L_{1}  &  =-\int\psi\left(  x,0\right)  u_{0}\left(  x\right)  dx-\int\int
\psi_{t}u^{\varepsilon}dxdt-\int\int u^{\varepsilon}\Delta\psi
dxdt\label{G2E3a}\\
Q_{1}  &  =\frac{1}{4\pi}\int\int\int\frac{\left[  \left(  x-y\right)
\cdot\left(  \nabla\psi\left(  x,t\right)  -\nabla\psi\left(  y,t\right)
\right)  \right]  }{\left\vert x-y\right\vert ^{2}}d\omega_{\varepsilon}%
^{-}\left(  x,y,t\right) \label{G2E3b}\\
Q_{2}  &  =\frac{1}{4\pi}\int\int\int\left[  \nabla\psi\left(  x,t\right)
Z\left(  y\right)  -\nabla\psi\left(  y,t\right)  Z\left(  x\right)  \right]
\left(  \frac{P_{\partial}\left(  x\right)  -P_{\partial}\left(  y\right)
}{D}\right)  d\omega_{\varepsilon}^{-}\left(  x,y,t\right) \label{G2E3c}\\
Q_{3}  &  =-\frac{1}{4\pi}\int\int\int\left[  \nabla\psi\left(  x,t\right)
Z\left(  y\right)  +\nabla\psi\left(  y,t\right)  Z\left(  x\right)  \right]
\frac{\left[  d\left(  x\right)  \nu\left(  x\right)  +d\left(  y\right)
\nu\left(  y\right)  \right]  }{D}d\omega_{\varepsilon}^{-}\left(
x,y,t\right) \label{G2E3d}\\
Q_{4}  &  =\int\int\int\nabla\psi\left(  x,t\right)  \frac{Z\left(  y\right)
h\left(  y\right)  }{2\pi}\cdot\label{G2E3e}\\
&  \cdot\left[  \mathcal{G}_{t}\left(  Y\left(  x,y\right)  ,\lambda
_{1}\left(  x,y\right)  ,\lambda_{2}\left(  x,y\right)  \right)
+\mathit{g}_{n}\left(  Y\left(  x,y\right)  ,\lambda_{1}\left(  x,y\right)
,\lambda_{2}\left(  x,y\right)  \right)  \nu\left(  y\right)  \right]
d\omega_{\varepsilon}^{-}\left(  x,y,t\right) \nonumber\\
Q_{5}  &  =-\int\int\int\nabla\psi\left(  x,t\right)  W\left(  x,y\right)
d\omega_{\varepsilon}^{-}\left(  x,y,t\right)  \label{G2E3f}%
\end{align}
and:%
\[
d\omega_{\varepsilon}^{-}\left(  x,y,t\right)  =f_{\varepsilon}\left(
u\left(  x,t\right)  \right)  f_{\varepsilon}\left(  u\left(  y,t\right)
\right)  dxdydt
\]

\end{lemma}

\begin{proof}
Multiplying the regularized equations by a test function $\psi\left(
x,t\right)  $ compactly supported in $\bar{\Omega}\times\left[  0,\infty
\right)  $, solving the Poisson equation using the corresponding Green's
function and integrating in $\bar{\Omega}\times\left[  0,\infty\right)  $ we
obtain:
\begin{align*}
0  &  =-\int\psi\left(  x,0\right)  u_{0}\left(  x\right)  dx-\int\int\psi
_{t}u^{\varepsilon}dxdt-\int\int u^{\varepsilon}\Delta\psi dxdt-\\
&  -\int\int\int f_{\varepsilon}\left(  u\left(  x,t\right)  \right)
f_{\varepsilon}\left(  u\left(  y,t\right)  \right)  \nabla\psi\left(
x\right)  \nabla_{x}G\left(  y,x\right)  dxdydt
\end{align*}
where $G$ is the Green's function for the Poisson equation described in Lemma
\ref{Lsymm}.

Using Lemma \ref{L1} it then follows that:
\begin{align}
&  L_{1}+\frac{1}{2\pi}\int\int\int\nabla\psi\left(  x,t\right)  \frac{\left(
x-y\right)  }{\left\vert x-y\right\vert ^{2}}d\omega_{\varepsilon}^{-}\left(
x,y,t\right)  +\label{M2E4a}\\
&  +\frac{1}{2\pi}\int\int\int\nabla\psi\left(  x,t\right)  Z\left(  y\right)
\frac{P_{\partial}\left(  x\right)  -P_{\partial}\left(  y\right)  -\left[
d\left(  x\right)  \nu\left(  x\right)  +d\left(  y\right)  \nu\left(
y\right)  \right]  }{D}d\omega_{\varepsilon}^{-}\left(  x,y,t\right)
+Q_{4}+Q_{5}=0\nonumber
\end{align}

Symmetrizing the second term in (\ref{M2E4a}) we can rewrite it as $Q_{1}.$ On
the other hand, symmetrizing the third term in (\ref{M2E4a}) it becomes
$Q_{2}+Q_{3}.$
\end{proof}

\bigskip

We now proceed to identify the limit of the quadratic terms $Q_{k}%
,\ k=1,...,5$ in (\ref{G2E3b})-(\ref{G2E3f}). The sequence $\left\{
\omega_{\varepsilon}^{-}\left(  x,y,t\right)  \right\}  $ has good properties
to pass to the limit in $M^{+}\left(  \bar{\Omega}\times\bar{\Omega}%
\times\mathbb{R}^{+}\right)  ,$ however these measures are multiplied by
functions like $\frac{\left[  \left(  x-y\right)  \cdot\left(  \nabla
\psi\left(  x,t\right)  -\nabla\psi\left(  y,t\right)  \right)  \right]
}{\left\vert x-y\right\vert ^{2}}$ that are bounded but not continuous near
the diagonal $\left\{  x=y\right\}  $. To deal with such a terms will require
to introduce measures defined in larger sets than $\bar{\Omega}\times
\bar{\Omega}\times\left[  0,\infty\right)  .$

\bigskip

\subsubsection{Limit of the nonlinear terms: The term $Q_{1}.$}

\bigskip

\begin{lemma}
\label{LQ1}There exist measures $\omega^{-}\in M^{+}\left(  \bar{\Omega}%
\times\bar{\Omega}\times\mathbb{R}^{+}\right)  ,\ \hat{\mu}^{-}\in
M^{+}\left(  \bar{\Omega}\times S^{1}\times\mathbb{R}^{+}\right)  $ satisfying
$\omega^{-}\left(  \left\{  x=y\right\}  \times\left[  0,\infty\right)
\right)  =0,$ $\operatorname*{supp}\left(  \hat{\mu}^{-}\right)  \subset
S\times S^{1}$ such that for any test function as in Lemma \ref{Lterms} we
have:
\begin{equation}
Q_{1}\rightarrow\int_{\left[  \bar{\Omega}\times S^{1}\right]  \times\left[
0,\infty\right)  }\frac{\left(  \nu\cdot\nabla^{2}\psi\left(  x,t\right)
\cdot\nu\right)  }{4\pi}d\hat{\mu}^{-}+\frac{1}{4\pi}\int_{\left[  \bar
{\Omega}\times\bar{\Omega}\times\left[  0,\infty\right)  \right]  \cap\left\{
x\neq y\right\}  }\frac{\left[  \left(  x-y\right)  \cdot\left(  \nabla
\psi\left(  x,t\right)  -\nabla\psi\left(  y,t\right)  \right)  \right]
}{\left\vert x-y\right\vert ^{2}}d\omega_{t}^{-}dt \label{G2E6}%
\end{equation}
for some subsequence $\varepsilon_{k}\rightarrow0,$ $k\rightarrow\infty.$

Moreover, we have:%
\begin{equation}
d\hat{\mu}^{-}=d\hat{\mu}_{t}^{-}dt\ \ \ ,\ \ \ d\omega^{-}=d\omega_{t}%
^{-}dt\label{G2E6Fa}%
\end{equation}
and:%
\begin{equation}
\int_{S^{1}}d\hat{\mu}_{t}^{-}\left(  \cdot,\nu\right)  =\sum_{x_{i}\left(
t\right)  \in S_{t}}\gamma_{i}^{-}\left(  t\right)  \delta_{x_{i}\left(
t\right)  }\ \ ,\ \ \gamma_{i}^{-}\left(  t\right)  \geq\left(  \beta_{i}%
^{-}\left(  t\right)  \right)  ^{2}\ \ ,\ \ a.e.\ t\in\left[  0,\infty\right)
\label{G2E6a}%
\end{equation}%
\begin{align}
d\omega_{t}^{-}\left(  x,y\right)   &  =\sum_{\left\{  x_{i}\left(  t\right)
\neq x_{j}\left(  t\right)  \right\}  }\lambda_{i,j}\left(  t\right)
\delta_{x_{i}\left(  t\right)  }\left(  x\right)  \delta_{x_{j}\left(
t\right)  }\left(  y\right)  +\label{G2E6b}\\
&  +\sum_{x_{i}\left(  t\right)  \in S_{t}}\beta_{i}\left(  t\right)  \left[
\delta_{x_{i}\left(  t\right)  }\left(  x\right)  u\left(  y,t\right)
dy+\delta_{x_{i}\left(  t\right)  }\left(  y\right)  u\left(  x,t\right)
dx\right]  +u\left(  x,t\right)  u\left(  y,t\right)  dxdy\nonumber
\end{align}
for some $\lambda_{i,j}\left(  t\right)  \geq0$ defined for $i\neq j.$
\end{lemma}

\begin{proof}
Let $\eta\in C^{\infty}\left(  \mathbb{R}^{+}\right)  $ be a cutoff function
satisfying:
\[
\eta\left(  s\right)  =1\;\;,\;\;\left\vert s\right\vert \leq\frac{1}%
{2}\;\;,\;\;\eta\left(  s\right)  =0\;\;,\;\;\left\vert s\right\vert
\geq1\;\;\;,\;\;0\leq\eta\leq1
\]
We then write:
\begin{align*}
Q_{1}  &  =\int\int\int H_{1}\left(  x,y,t\right)  \eta\left(  \frac
{\left\vert x-y\right\vert }{\delta}\right)  d\omega_{\varepsilon}^{-}\left(
x,y,t\right)  +\int\int\int H_{1}\left(  x,y,t\right)  \left[  1-\eta\left(
\frac{\left\vert x-y\right\vert }{\delta}\right)  \right]  d\omega
_{\varepsilon}^{-}\left(  x,y,t\right) \\
&  \equiv I_{1}+I_{2}%
\end{align*}
where:%
\begin{equation}
H_{1}\left(  x,y,t\right)  =\frac{1}{4\pi}\frac{\left[  \left(  x-y\right)
\cdot\left(  \nabla\psi\left(  x,t\right)  -\nabla\psi\left(  y,t\right)
\right)  \right]  }{\left\vert x-y\right\vert ^{2}} \label{H1A1}%
\end{equation}

Given $\varphi\in C^{\infty}\left(  \Omega\times S^{1}\times\mathbb{R}%
^{+}\right)  $ we define:
\[
\hat{\mu}_{\delta,\varepsilon}^{-}\left(  \varphi\right)  \equiv\int\int
\int\varphi\left(  x,\frac{x-y}{\left\vert x-y\right\vert },t\right)
\eta\left(  \frac{\left\vert x-y\right\vert }{\delta}\right)  d\omega
_{\varepsilon}^{-}\left(  x,y,t\right)
\]

The family of nonnegative measures $\hat{\mu}_{\delta,\varepsilon}^{-}\left(
\chi_{\left[  0,T\right]  }\varphi\right)  $ is compact in $M^{+}\left(
\Omega\times S^{1}\times\left[  0,T\right]  \right)  $ for each $T<\infty$
with the weak topology, since $\hat{\mu}_{\delta,\varepsilon}^{-}\left(
\chi_{\left[  0,T\right]  }\cdot1\right)  \leq T\left(  \int u_{0}dx\right)
^{2}.$\texttt{ }Taking a subsequence if needed we can define $\hat{\mu
}_{\delta,T}^{-}\left(  \varphi\right)  =\lim_{k\rightarrow\infty}\hat{\mu
}_{\delta,\varepsilon_{k}}^{-}\left(  \chi_{\left[  0,T\right]  }%
\varphi\right)  $ where the limit is taken in the weak topology. We can now
take the limit $\delta\rightarrow0$ for suitable subsequences. Then:
\begin{equation}
\hat{\mu}_{\delta_{k},T}^{-}\rightharpoonup\hat{\mu}_{T}^{-} \label{G2E7}%
\end{equation}

Moreover, we can write $d\hat{\mu}_{T}^{-}=d\hat{\mu}_{t}^{-}dt$ arguing as in
the Proof of Proposition \ref{M1}.

We can now compute the limit of the term $I_{1}$ using the measures $\hat{\mu
}_{T}^{-}.$ To this end, we approximate the test function $\frac{\left[
\left(  x-y\right)  \cdot\left(  \nabla\psi\left(  x,t\right)  -\nabla
\psi\left(  y,t\right)  \right)  \right]  }{\left\vert x-y\right\vert ^{2}}$
by a test function having the form $\varphi\left(  x,\frac{x-y}{\left\vert
x-y\right\vert },t\right)  .$ Indeed, Taylor's Theorem yields:
\[
\frac{\left[  \left(  x-y\right)  \cdot\left(  \nabla\psi\left(  x,t\right)
-\nabla\psi\left(  y,t\right)  \right)  \right]  }{\left\vert x-y\right\vert
^{2}}=\frac{\left(  x-y\right)  \cdot\nabla^{2}\psi\left(  x,t\right)
\cdot\left(  x-y\right)  }{\left\vert x-y\right\vert ^{2}}+O\left(
\delta\right)
\]
whence:
\[
I_{1}=\hat{\mu}_{\delta,\varepsilon}^{-}\left(  \chi_{\left[  0,T\right]
}\varphi\right)  +O\left(  \delta\right)  \ \ \ ,\ \ \ \varphi\left(
x,\nu,t\right)  =\frac{\nu\cdot\nabla^{2}\psi\left(  x,t\right)  \cdot\nu
}{4\pi}\;,\;x\in\Omega\;,\;\nu\in S^{1}%
\]

Therefore, the limit $\varepsilon\rightarrow0,$ $\delta\rightarrow0$, for
suitable subsequences, yields:
\begin{equation}
I_{1}\rightarrow\int\int_{\bar{\Omega}\times S^{1}}\left(  \frac{\nu
\cdot\nabla^{2}\psi\left(  x,t\right)  \cdot\nu}{4\pi}\right)  d\hat{\mu}%
_{t}^{-}\left(  x,\nu\right)  dt \label{G2E7a}%
\end{equation}

It only remains to compute the limit of $I_{2}$ as $\varepsilon\rightarrow0,$
$\delta\rightarrow0.$ Notice that the family $\left\{  \omega_{\varepsilon
}^{-}\left(  x,y,t\right)  \right\}  $ is weakly compact in $M^{+}\left(
\left[  \left(  \bar{\Omega}\times\bar{\Omega}\right)  \cap\left\{  \left\vert
x-y\right\vert \geq\delta\right\}  \right]  \times\left[  0,T\right]  \right)
.$ Therefore there exists $\omega_{\delta,T}^{-}$ such that $\omega
_{\varepsilon}^{-}\left[  1-\eta\left(  \frac{\left\vert x-y\right\vert
}{\delta}\right)  \right]  \rightharpoonup\omega_{\delta,T}^{-}.$ Taking then
the limit $\delta\rightarrow0,$ we finally obtain:%
\begin{equation}
I_{2}\rightarrow\frac{1}{4\pi}\int_{\left[  \bar{\Omega}\times\bar{\Omega
}\times\left[  0,\infty\right)  \right]  \cap\left\{  x\neq y\right\}  }%
\frac{\left[  \left(  x-y\right)  \cdot\left(  \nabla\psi\left(  x,t\right)
-\nabla\psi\left(  y,t\right)  \right)  \right]  }{\left\vert x-y\right\vert
^{2}}d\omega_{t}^{-}dt \label{G2E7b}%
\end{equation}

Combining (\ref{G2E7a}), (\ref{G2E7b}) we obtain (\ref{G2E6}). The
representation of $\omega^{-}$ given in (\ref{G2E6Fa}) follows as in
Proposition \ref{M1}.

To derive (\ref{G2E6a}) we compute the measure $\hat{\mu}_{t}^{-}$ acting over
test functions independent on $\nu.$ We then need to consider the limit as
$\varepsilon\rightarrow0,\ \delta\rightarrow0,$ (for suitable subsequences) of
integrals with the form:%
\begin{equation}
\int\int\int f_{\varepsilon}\left(  u^{\varepsilon}\left(  x,t\right)
\right)  f_{\varepsilon}\left(  u^{\varepsilon}\left(  y,t\right)  \right)
\varphi\left(  x,t\right)  \eta\left(  \frac{\left\vert x-y\right\vert
}{\delta}\right)  dxdydt \label{G2A1}%
\end{equation}

Writing:%
\[
f_{\varepsilon}\left(  u^{\varepsilon}\left(  x,t\right)  \right)  =\mu
^{-}\left(  x,t\right)  +\left[  f_{\varepsilon}\left(  u^{\varepsilon}\left(
x,t\right)  \right)  -\mu^{-}\left(  x,t\right)  \right]
\]
we obtain:%
\begin{align}
&  f_{\varepsilon}\left(  u^{\varepsilon}\left(  x,t\right)  \right)
f_{\varepsilon}\left(  u^{\varepsilon}\left(  y,t\right)  \right) \nonumber\\
&  =\mu^{-}\left(  x,t\right)  \mu^{-}\left(  y,t\right)  +\mu^{-}\left(
x,t\right)  \left[  f_{\varepsilon}\left(  u^{\varepsilon}\left(  y,t\right)
\right)  -\mu^{-}\left(  y,t\right)  \right]  +\nonumber\\
&  +\left[  f_{\varepsilon}\left(  u^{\varepsilon}\left(  x,t\right)  \right)
-\mu^{-}\left(  x,t\right)  \right]  \mu^{-}\left(  y,t\right)  +\nonumber\\
&  +\left[  f_{\varepsilon}\left(  u^{\varepsilon}\left(  x,t\right)  \right)
-\mu^{-}\left(  x,t\right)  \right]  \left[  f_{\varepsilon}\left(
u^{\varepsilon}\left(  y,t\right)  \right)  -\mu^{-}\left(  y,t\right)
\right]  \label{G2A2}%
\end{align}
{}

Plugging (\ref{G2A2}) into (\ref{G2A1}) it follows that the limit of the
second and third term approach zero as $\varepsilon\rightarrow0.$ On the other
hand, in order to estimate the contribution of the last term in (\ref{G2A2})
we remark that, estimating $f_{\varepsilon}\left(  u^{\varepsilon}\right)  $
by $u^{\varepsilon}$, and using that outside a ball of radius $\rho$ of the
singular set $u^{\varepsilon}$ converges uniformly to $u,$ we can estimate the
contribution outside the singular set by a $L^{1}$ function and the resulting
integral contribution approaches zero as $\delta\rightarrow0,$ since the
measure of the considered set approaches zero. Therefore, the integration in
(\ref{G2A1}) is restricted to $S+B_{\rho}\left(  0\right)  \times\left\{
0\right\}  $ with $\rho$ very small. In such a set we can assume that
$\eta\left(  \frac{\left\vert x-y\right\vert }{\delta}\right)  $ is constant,
and $\varphi\left(  x,t\right)  $ can be approximated by the values at the
singular set, therefore, by functions depending only on time. It then follows
that the last term in (\ref{G2A2}) gives a contribution with the form:%
\[
\int\sum_{x_{j}\left(  t\right)  \in S_{t}}\varphi\left(  x_{j}\left(
t\right)  ,t\right)  \left(  \int\left[  f_{\varepsilon}\left(  u^{\varepsilon
}\left(  x,t\right)  \right)  -\mu^{-}\left(  x,t\right)  \right]  dx\right)
^{2}dt\geq0
\]
except for a small error. $\varphi\left(  x,t\right)  $ in (\ref{G2A1}) can be
approximated as $\varepsilon\rightarrow0$ by the sum of the values at the
singular set.

It then follows that:%
\[
\int_{S^{1}}d\hat{\mu}_{t}^{-}\left(  \cdot,\nu\right)  \geq\left(
\mu_{\operatorname{sing}}^{-}\left(  \cdot\right)  \right)  ^{2}%
\]
whence using the fact that $f_{\varepsilon}\left(  u\right)  \leq u$ as well
as Corollary \ref{MassBound}, (\ref{G2E6a}) follows. The representation
formula (\ref{G2E6b}) is a consequence of the fact that the measures $\left\{
\omega_{\varepsilon}^{-}\left(  x,y,t\right)  \left[  1-\eta\left(
\frac{\left\vert x-y\right\vert }{\delta}\right)  \right]  \right\}  $ are
supported in the region $\left\{  \left\vert x-y\right\vert \geq
\delta\right\}  ,$ as well as (\ref{T1E1d}). If we consider points $\left(
x_{0},y_{0}\right)  $ at the singular set we can obtain smoothness of the
solutions in an neighbourhood, and obtain strong convergence of $f\left(
u^{\varepsilon}\left(  x,t\right)  \right)  f\left(  u^{\varepsilon}\left(
y,t\right)  \right)  .$ This gives the term $u\left(  x,t\right)  u\left(
y,t\right)  $ in (\ref{G2E6b}). If, say $x_{0}\in S_{t}$ and $y_{0}$ is a
regular point, we obtain strong convergence of the function $f\left(
u^{\varepsilon}\left(  y,t\right)  \right)  $ in a neighbourhood and, taking
the product of weak convergence with strong convergence to obtain the terms
$\sum_{x_{i}\left(  t\right)  \in S_{t}}\beta_{i}^{-}\left(  t\right)  \left[
\delta_{x_{i}\left(  t\right)  }\left(  x\right)  u\left(  y,t\right)
dy+\delta_{x_{i}\left(  t\right)  }\left(  y\right)  u\left(  x,t\right)
dx\right]  $ in (\ref{G2E6b}). Finally, in a neighbourhood of the points
$\left(  x,y\right)  =\left(  x_{i}\left(  t\right)  ,x_{j}\left(  t\right)
\right)  \in S_{t}\times S_{t}$ with $i\neq j$ we can only prove the existence
of a singular set with weights $\lambda_{i,j}\left(  t\right)  \geq0.$
Unfortunately it is not possible to ensure that $\lambda_{i,j}\left(
t\right)  =\beta_{i}^{-}\left(  t\right)  \beta_{j}^{-}\left(  t\right)  $
without a careful study of the possible fast oscillations in time of the
functions $f_{\varepsilon}\left(  u^{\varepsilon}\right)  $. A Young measure
formalism that allows to describe the possible effect of oscillations in short
time scales is given in Section \ref{Young}.
\end{proof}

\bigskip

A consequence of Lemmas \ref{MassChange} and \ref{LQ1} is the following:

\begin{corollary}
\label{MassBound}Suppose that $\mu,\ \omega^{-}$ are as in Proposition
\ref{M1}. Then:%
\begin{equation}
d\omega_{t}^{-}\left(  x,y\right)  \leq d\mu_{t}\left(  x\right)  d\mu
_{t}\left(  y\right)  \label{Wmass}%
\end{equation}
\bigskip
\end{corollary}

\subsubsection{Limit of the nonlinear terms: The term $Q_{2}.$}

\bigskip

In order to characterize the limit of the term $Q_{2}$ we need to define a
manifold that will play the role of $\Omega\times S^{1}$ for the points at the
boundary. We define also some auxiliary sets.

\begin{definition}
For any $y\in\partial\Omega$ we define the following manifold with boundary:
\[
\mathcal{M}_{2}\left[  y\right]  =\left\{  \left(  Y,\lambda_{1},\lambda
_{2}\right)  :Y\in\mathbb{TM}_{y}\left(  \partial\Omega\right)  ,\;\lambda
_{1}\geq0,\;\lambda_{2}\geq0,\;\left\vert Y\right\vert ^{2}+\left(
\lambda_{1}+\lambda_{2}\right)  ^{2}=1\right\}
\]

\end{definition}

Notice that $\mathcal{M}_{2}\left[  y\right]  $ is isomorphic to the
intersection of a two-dimensional cylinder with the quadrant $\left\{
\lambda_{1}\geq0,\;\lambda_{2}\geq0\right\}  .$

\begin{definition}
We will denote as $\mathcal{M}_{2}$ the set $\left\{  \left(  y,\sigma\right)
:y\in\partial\Omega,\ \sigma\in\mathcal{M}_{2}\left[  y\right]  \right\}  $
endowed naturally with a structure of three-dimensional manifold with boundary.
\end{definition}

\begin{definition}
We will denote as $\mathcal{M}_{2}^{\left(  \varepsilon_{0}\right)  }$ the set
$\left\{  \left(  y,\sigma\right)  :y\in\Omega,\ \operatorname*{dist}\left(
y,\partial\Omega\right)  \leq\varepsilon_{0},\ \sigma\in\mathcal{M}_{2}\left[
y_{0}\right]  \right\}  ,\ $where $y_{0}$ is the closest point to $y$ in
$\partial\Omega$ and $\varepsilon_{0}>0$ is fixed sufficiently small.
\end{definition}

\begin{lemma}
\label{LQ2}Let $Q_{2}$ as in (\ref{G2E3c}). There exists a measure $\hat{\mu
}_{b}^{-}\in M^{+}\left(  \mathcal{M}_{2}\times\mathbb{R}^{+}\right)
,\ d\hat{\mu}_{b}^{-}=d\hat{\mu}_{b,t}^{-}dt$ such that, for suitable
subsequences $\varepsilon_{k}\rightarrow0,\ \delta_{k}\rightarrow0,$
$k\rightarrow\infty:$%
\begin{align}
Q_{2}  &  \rightarrow\hat{\mu}_{b}^{-}\left(  \varphi_{1}\right)
+\label{G4E3}\\
&  +\frac{1}{4\pi}\int\int\int_{\left[  \bar{\Omega}\times\bar{\Omega}%
\times\left[  0,\infty\right)  \right]  \cap\left\{  x\neq y\right\}  }\left[
\nabla\psi\left(  x,t\right)  Z\left(  y\right)  -\nabla\psi\left(
y,t\right)  Z\left(  x\right)  \right]  \left(  \frac{P_{\partial}\left(
x\right)  -P_{\partial}\left(  y\right)  }{D}\right)  d\omega_{t}^{-}\left(
x,y\right)  dt\nonumber
\end{align}
with $\omega_{t}^{-}\left(  x,y\right)  $ as in Lemma \ref{LQ1} and:with:
\begin{equation}
\varphi_{1}\left(  y,Y,\lambda_{1},\lambda_{2},t\right)  =\left[  Y+\left(
\lambda_{2}-\lambda_{1}\right)  \nu\left(  y\right)  \right]  \cdot\nabla
^{2}\psi\left(  y,t\right)  \cdot Y\;\;\;,\;\;y\in\partial\Omega
\;\;,\;\;\left(  Y,\lambda_{1},\lambda_{2}\right)  \in\mathcal{M}_{2}\left[
y\right]  \;\;,\;\;t\in\left[  0,\infty\right)  \label{G4E2}%
\end{equation}

Moreover:%
\begin{equation}
\int_{\mathcal{M}_{2}\left[  \cdot\right]  }d\hat{\mu}_{b,t}^{-}\left(
\cdot,\sigma\right)  =\sum_{x_{i}\left(  t\right)  \in S_{t}}\gamma_{b,i}%
^{-}\left(  t\right)  \delta_{x_{i\left(  t\right)  }}\left(  \cdot\right)
\ \ ,\ \ \gamma_{b,i}^{-}\left(  t\right)  \geq\left(  \beta_{i}^{-}\left(
t\right)  \right)  ^{2}\ \ ,\ \ a.e.\ t\in\left[  0,\infty\right)
\label{G3A1}%
\end{equation}

\end{lemma}

\begin{proof}
We use the same argument as in the proof of Lemma \ref{LQ1}. Using the same
cutoff function $\eta$ we write:
\begin{align*}
&  \int\int\int H_{2}\left(  x,y,t\right)  d\omega_{\varepsilon}^{-}\left(
x,y,t\right) \\
&  =\int\int\int H_{2}\left(  x,y,t\right)  \eta\left(  \frac{\left\vert
x-\tau\left(  y\right)  \right\vert }{\delta}\right)  d\omega_{\varepsilon
}^{-}\left(  x,y,t\right)  +\\
&  +\int\int\int H_{2}\left(  x,y,t\right)  \left[  1-\eta\left(
\frac{\left\vert x-\tau\left(  y\right)  \right\vert }{\delta}\right)
\right]  d\omega_{\varepsilon}^{-}\left(  x,y,t\right) \\
&  \equiv I_{1}+I_{2}%
\end{align*}
with:%
\begin{equation}
H_{2}\left(  x,y,t\right)  =\frac{1}{4\pi}\left[  \nabla\psi\left(
x,t\right)  Z\left(  y\right)  -\nabla\psi\left(  y,t\right)  Z\left(
x\right)  \right]  \left(  \frac{P_{\partial}\left(  x\right)  -P_{\partial
}\left(  y\right)  }{D}\right)  \label{H1A2}%
\end{equation}

In order to define the measure $\hat{\mu}_{b}^{-}\in M^{+}\left(
\mathcal{M}_{2}\times\mathbb{R}^{+}\right)  $ we argue as follows. Given a
test function $\varphi\in C^{\infty}\left(  \mathcal{M}_{2}\times
\mathbb{R}^{+}\right)  $ we extend it to $\mathcal{M}_{2}^{\left(
\varepsilon_{0}\right)  }$ for small $\varepsilon_{0}$ as $\varphi\left(
y,\sigma\right)  =\varphi\left(  y_{0},\sigma\right)  $ with $\sigma
\in\mathcal{M}_{2}\left[  y_{0}\right]  $ and we define the auxiliary linear
functional for $\delta>0$ sufficiently small:%
\begin{equation}
\hat{\mu}_{b,\delta,\varepsilon}^{-}\left(  \varphi\right)  =\int\int
\int\varphi\left(  \frac{x+y}{2},\left[  \frac{\left(  P_{\partial}\left(
x\right)  -P_{\partial}\left(  y\right)  \right)  \cdot t\left(  y\right)
}{\kappa\sqrt{D}}\right]  t\left(  y\right)  ,\frac{d\left(  x\right)  }%
{\sqrt{D}},\frac{d\left(  y\right)  }{\sqrt{D}},t\right)  \eta\left(
\frac{\left\vert x-\tau\left(  y\right)  \right\vert }{\delta}\right)
d\omega_{\varepsilon}^{-}\left(  x,y,t\right)  \label{G4E7}%
\end{equation}
where $\kappa=\frac{\left\vert \left(  P_{\partial}\left(  x\right)
-P_{\partial}\left(  y\right)  \right)  \cdot t\left(  y\right)  \right\vert
}{\left\vert P_{\partial}\left(  x\right)  -P_{\partial}\left(  y\right)
\right\vert }.$ Notice that for any $x,y\ $satisfying $\left\vert
x-\tau\left(  y\right)  \right\vert \leq\delta,$ we have also
$\operatorname*{dist}\left(  x,\partial\Omega\right)  \leq\delta,$
$\operatorname*{dist}\left(  y,\partial\Omega\right)  \leq2\delta$ for
$\delta$ sufficiently small. In particular $Z\left(  y\right)  =1.$ The
sequence of measures $\hat{\mu}_{b,\delta,\varepsilon}^{-}$ is weakly compact
in\footnote{See if it is possible to put $T=\infty$ directly.} $\left[
C\left(  \mathcal{M}_{2}\times\left[  0,T\right)  \right)  \right]  ^{\ast}$
for any $0<T<\infty$. Then, taking suitable subsequences:
\[
\hat{\mu}_{b,\delta,\varepsilon}^{-}\rightharpoonup\hat{\mu}_{b}^{-}%
\]
where $\hat{\mu}_{b}^{-}\in\mathbb{M}^{+}\left(  \mathcal{M}_{2}%
\times\mathbb{R}^{+}\right)  $.

We can then pass to the limit in $I_{1}$ as follows. Using Taylor's we can
write:
\begin{align*}
I_{1}  &  =\int\int\int\eta\left(  \frac{\left\vert x-\tau\left(  y\right)
\right\vert }{\delta}\right)  \cdot\\
&  \cdot\left[  \left[  \left(  P_{\partial}\left(  x\right)  -P_{\partial
}\left(  y\right)  \right)  \cdot t\left(  y\right)  \right]  t\left(
y\right)  +\left(  d\left(  y\right)  -d\left(  x\right)  \right)  \nu\left(
y\right)  \right]  \frac{\nabla^{2}\psi\left(  y,t\right)  }{4\pi}\left(
\frac{P_{\partial}\left(  x\right)  -P_{\partial}\left(  y\right)  }%
{D}\right)  d\omega_{\varepsilon}^{-}\left(  x,y,t\right) \\
&  +O\left(  \delta\right)
\end{align*}

Then:
\[
I_{1}=\hat{\mu}_{b,\delta,\varepsilon}^{-}\left(  \varphi_{1}\right)
+O\left(  \delta\right)
\]

Then, taking suitable subsequences:
\[
I_{1}\rightarrow\hat{\mu}_{b}^{-}\left(  \varphi_{1}\right)
\;\;,\;\;\varepsilon_{k}\rightarrow0\;\;,\;\;\delta_{k}\rightarrow
0\;\;\text{as}\;k\rightarrow\infty
\]

On the other hand, taking the limit $\varepsilon\rightarrow0$ and
$\delta\rightarrow0$, also for suitable subsequences:
\[
I_{2}\rightarrow\frac{1}{4\pi}\int\int\int_{\left[  \bar{\Omega}\times
\bar{\Omega}\times\left[  0,\infty\right)  \right]  \cap\left\{  x\neq
y\right\}  }\left[  \nabla\psi\left(  x,t\right)  Z\left(  y\right)
-\nabla\psi\left(  y,t\right)  Z\left(  x\right)  \right]  \left(
\frac{P_{\partial}\left(  x\right)  -P_{\partial}\left(  y\right)  }%
{D}\right)  d\omega_{t}^{-}\left(  x,y\right)  dt
\]
where the measure $\omega_{t}^{-}$ is as in (\ref{G2E7b}).

We can compute the action of the measure $\hat{\mu}_{b}^{-}$ over test
functions $\varphi\in C\left(  \mathcal{M}_{2}\times\mathbb{R}^{+}\right)  $
depending only on $\left(  x,t\right)  \in\partial\Omega\times\left[
0,\infty\right)  .$ To this end we need to consider the limit of:%
\[
\int\int\int\varphi\left(  \frac{x+y}{2},t\right)  \eta\left(  \frac
{\left\vert x-\tau\left(  y\right)  \right\vert }{\delta}\right)
d\omega_{\varepsilon}^{-}\left(  x,y,t\right)
\]
that converge in the limit $\varepsilon_{k}\rightarrow0\;,\delta
_{k}\rightarrow0\;$as$\;k\rightarrow\infty$ to:%
\[
\int_{0}^{\infty}\int_{\partial\Omega}\int_{\mathcal{M}_{2}\left[  x\right]
}\varphi\left(  x,t\right)  d\hat{\mu}_{b,t}^{-}\left(  x,\sigma\right)  dt
\]

Arguing as in the Proof of Lemma \ref{LQ1} we obtain (\ref{G3A1}).
\end{proof}

\bigskip

\subsubsection{Limit of the nonlinear terms: The term $Q_{3}.$}

We now compute the limit of $Q_{3}$.

\begin{lemma}
\label{LQ3}Suppose that $\psi\in C^{\infty}\left(  \bar{\Omega}\times
\mathbb{R}^{+}\right)  $ satisfies (\ref{G2E4}). Then, taking suitable
subsequences $\varepsilon_{k}\rightarrow0,\ \delta_{k}\rightarrow0$ as
$k\rightarrow\infty$ we have:%
\begin{align}
Q_{3}  &  \rightarrow\hat{\mu}_{b}^{-}\left(  \varphi_{2}+\varphi_{3}\right)
-\frac{1}{4\pi}\int\int\int_{\left[  \bar{\Omega}\times\bar{\Omega}%
\times\left[  0,\infty\right)  \right]  \cap\left\{  x\neq y\right\}  }%
\frac{\left[  d\left(  x\right)  \nu\left(  x\right)  +d\left(  y\right)
\nu\left(  y\right)  \right]  }{D}\cdot\nonumber\\
&  \cdot\left[  Z\left(  y\right)  \left[  \nabla\psi\left(  x,t\right)
-Z\left(  x\right)  \nabla\psi\left(  P_{\partial}\left(  x\right)  ,t\right)
\right]  +Z\left(  x\right)  \left[  \nabla\psi\left(  y,t\right)  -Z\left(
y\right)  \nabla\psi\left(  P_{\partial}\left(  y\right)  ,t\right)  \right]
\right]  d\omega_{t}^{-}\left(  x,y\right)  dt-\nonumber\\
&  -\frac{1}{4\pi}\int\int\int_{\left[  \bar{\Omega}\times\bar{\Omega}%
\times\left[  0,\infty\right)  \right]  \cap\left\{  x\neq y\right\}  }%
\frac{Z\left(  x\right)  Z\left(  y\right)  \left[  \nabla\psi\left(
P_{\partial}\left(  x\right)  ,t\right)  -\nabla\psi\left(  P_{\partial
}\left(  y\right)  ,t\right)  \right]  }{D}\cdot\label{G4E8}\\
&  \cdot\left[  d\left(  y\right)  \nu\left(  y\right)  -d\left(  x\right)
\nu\left(  x\right)  \right]  d\omega_{t}^{-}\left(  x,y\right)  dt\nonumber
\end{align}
where $Q_{3}$ is as in (\ref{G2E3d}), the measure $\hat{\mu}_{b}^{-}$ as in
Lemma \ref{LQ2}, $\omega_{t}^{-}\left(  x,y\right)  $ is as in Lemma \ref{LQ1}
and the test functions $\varphi_{2},\ \varphi_{3}$ are given by:%
\begin{align}
\varphi_{2}\left(  y,Y,\lambda_{1},\lambda_{2}\right)   &  =\frac{1}{4\pi
}\left[  \nu\left(  y\right)  \cdot\nabla^{2}\psi\left(  y,t\right)  \cdot
\nu\left(  y\right)  \right]  \left(  \lambda_{1}+\lambda_{2}\right)
^{2}\label{M3E1}\\
\varphi_{3}\left(  y,Y,\lambda_{1},\lambda_{2}\right)   &  =\frac{1}{4\pi
}\left(  \lambda_{1}-\lambda_{2}\right)  \left[  Y\cdot\nabla^{2}\psi\left(
y,t\right)  \cdot\nu\left(  y\right)  \right]  \label{M3E2}%
\end{align}
with $\left(  Y,\lambda_{1},\lambda_{2}\right)  \in\mathcal{M}_{2}\left[
y\right]  .$
\end{lemma}

\begin{proof}
In order to rewrite $Q_{3}$ we use the identity:%

\begin{align*}
&  \left[  \nabla\psi\left(  x,t\right)  Z\left(  y\right)  +\nabla\psi\left(
y,t\right)  Z\left(  x\right)  \right]  \left[  d\left(  x\right)  \nu\left(
x\right)  +d\left(  y\right)  \nu\left(  y\right)  \right] \\
&  =\left[  \left[  \nabla\psi\left(  x,t\right)  -Z\left(  x\right)
\nabla\psi\left(  P_{\partial}\left(  x\right)  ,t\right)  \right]  Z\left(
y\right)  +\nabla\psi\left(  P_{\partial}\left(  x\right)  ,t\right)  Z\left(
x\right)  Z\left(  y\right)  \right]  \left[  d\left(  x\right)  \nu\left(
x\right)  +d\left(  y\right)  \nu\left(  y\right)  \right] \\
&  +\left[  \left[  \nabla\psi\left(  y,t\right)  -Z\left(  y\right)
\nabla\psi\left(  P_{\partial}\left(  y\right)  ,t\right)  \right]  Z\left(
x\right)  +\nabla\psi\left(  P_{\partial}\left(  y\right)  ,t\right)  Z\left(
x\right)  Z\left(  y\right)  \right]  \left[  d\left(  x\right)  \nu\left(
x\right)  +d\left(  y\right)  \nu\left(  y\right)  \right]
\end{align*}

Using $\nu\left(  P_{\partial}\left(  x\right)  \right)  =\nu\left(  x\right)
$ as well as (\ref{G2E4}), we obtain, after symmetrizing:
\begin{align*}
&  \left[  \nabla\psi\left(  x,t\right)  Z\left(  y\right)  +\nabla\psi\left(
y,t\right)  Z\left(  x\right)  \right]  \left[  d\left(  x\right)  \nu\left(
x\right)  +d\left(  y\right)  \nu\left(  y\right)  \right] \\
&  =\left[  Z\left(  y\right)  \left[  \nabla\psi\left(  x,t\right)  -Z\left(
x\right)  \nabla\psi\left(  P_{\partial}\left(  x\right)  ,t\right)  \right]
+Z\left(  x\right)  \left[  \nabla\psi\left(  y,t\right)  -Z\left(  y\right)
\nabla\psi\left(  P_{\partial}\left(  y\right)  ,t\right)  \right]  \right]
\cdot\left[  d\left(  x\right)  \nu\left(  x\right)  +d\left(  y\right)
\nu\left(  y\right)  \right]  +\\
&  +Z\left(  x\right)  Z\left(  y\right)  \left[  \nabla\psi\left(
P_{\partial}\left(  x\right)  ,t\right)  -\nabla\psi\left(  P_{\partial
}\left(  y\right)  ,t\right)  \right]  \cdot\left[  d\left(  y\right)
\nu\left(  y\right)  -d\left(  x\right)  \nu\left(  x\right)  \right]
\end{align*}

Therefore, we can write $Q_{3}$ as:%
\[
Q_{3}=Q_{3,1}+Q_{3,2}%
\]
with:%
\[
Q_{3,1}=-\int\int\int H_{3,1}\left(  x,y,t\right)  d\omega_{\varepsilon}%
^{-}\left(  x,y,t\right)  \ \ \ ,\ \ \ Q_{3,2}=-\int\int\int H_{3,2}\left(
x,y,t\right)  d\omega_{\varepsilon}^{-}\left(  x,y,t\right)
\]%
\begin{align}
H_{3,1}\left(  x,y,t\right)   &  =\frac{\left[  d\left(  x\right)  \nu\left(
x\right)  +d\left(  y\right)  \nu\left(  y\right)  \right]  }{4\pi D}%
\cdot\left[  Z\left(  y\right)  \left[  \nabla\psi\left(  x,t\right)
-Z\left(  x\right)  \nabla\psi\left(  P_{\partial}\left(  x\right)  ,t\right)
\right]  \right.  +\nonumber\\
&  \left.  +Z\left(  x\right)  \left[  \nabla\psi\left(  y,t\right)  -Z\left(
y\right)  \nabla\psi\left(  P_{\partial}\left(  y\right)  ,t\right)  \right]
\right] \label{H1A3}\\
H_{3,2}\left(  x,y,t\right)   &  =\frac{Z\left(  x\right)  Z\left(  y\right)
\left[  \nabla\psi\left(  P_{\partial}\left(  x\right)  ,t\right)  -\nabla
\psi\left(  P_{\partial}\left(  y\right)  ,t\right)  \right]  \cdot\left[
d\left(  y\right)  \nu\left(  y\right)  -d\left(  x\right)  \nu\left(
x\right)  \right]  }{4\pi D} \label{H1A4}%
\end{align}

In order to compute the contribution near the diagonal $\left\{  x=y\right\}
$ and away from it we split $Q_{3,1},\ Q_{3,2}$ as in the Proof of Lemmas
\ref{LQ1}, \ref{LQ2}. More precisely:%
\begin{align*}
Q_{3,1}  &  =-\int\int\int H_{3,1}\left(  x,y,t\right)  \eta\left(
\frac{\left\vert x-\tau\left(  y\right)  \right\vert }{\delta}\right)
d\omega_{\varepsilon}^{-}\left(  x,y,t\right) \\
&  -\int\int\int H_{3,1}\left(  x,y,t\right)  \left[  1-\eta\left(
\frac{\left\vert x-\tau\left(  y\right)  \right\vert }{\delta}\right)
\right]  d\omega_{\varepsilon}^{-}\left(  x,y,t\right) \\
&  \equiv I_{1,1}+I_{1,2}%
\end{align*}%
\begin{align*}
Q_{3,2}  &  =-\int\int\int H_{3,2}\left(  x,y,t\right)  \eta\left(
\frac{\left\vert x-\tau\left(  y\right)  \right\vert }{\delta}\right)
d\omega_{\varepsilon}^{-}\left(  x,y,t\right) \\
&  -\int\int\int H_{3,2}\left(  x,y,t\right)  \left[  1-\eta\left(
\frac{\left\vert x-\tau\left(  y\right)  \right\vert }{\delta}\right)
\right]  d\omega_{\varepsilon}^{-}\left(  x,y,t\right) \\
&  \equiv I_{2,1}+I_{2,2}%
\end{align*}

Using Taylor's expansion as well as the definition of $P_{\partial}\left(
x\right)  ,\ P_{\partial}\left(  y\right)  :$%
\[
I_{1,1}=\frac{1}{4\pi}\int\int\int\frac{\left[  d\left(  x\right)  +d\left(
y\right)  \right]  ^{2}}{D}\left[  \nu\left(  y\right)  \cdot\nabla^{2}%
\psi\left(  y,t\right)  \cdot\nu\left(  y\right)  \right]  \eta\left(
\frac{\left\vert x-\tau\left(  y\right)  \right\vert }{\delta}\right)
d\omega_{\varepsilon}^{-}\left(  x,y,t\right)  +O\left(  \delta\right)
\]%
\[
I_{2,1}=\frac{1}{4\pi}\int\int\int\left[  \left(  P_{\partial}\left(
x\right)  -P_{\partial}\left(  y\right)  \right)  \cdot\nabla^{2}\psi\left(
y,t\right)  \cdot\nu\left(  y\right)  \right]  \frac{\left[  d\left(
x\right)  -d\left(  y\right)  \right]  }{D}\eta\left(  \frac{\left\vert
x-\tau\left(  y\right)  \right\vert }{\delta}\right)  d\omega_{\varepsilon
}^{-}\left(  x,y,t\right)
\]
Taking the limit $\varepsilon_{k}\rightarrow0,\;\delta_{k}\rightarrow0$ for
suitable subsequences we then obtain:
\[
I_{1,1}+I_{2,1}\rightarrow\hat{\mu}_{b}^{-}\left(  \varphi_{2}+\varphi
_{3}\right)
\]

On the other hand, using the boundedness of the integrands we can pass to the
limit as $\varepsilon_{k}\rightarrow0,\;\delta_{k}\rightarrow0$ in
$I_{1,2},\;I_{2,2}$ to obtain:
\begin{align*}
I_{1,2}  &  \rightarrow-\frac{1}{4\pi}\int\int\int_{\left[  \bar{\Omega}%
\times\bar{\Omega}\times\left[  0,\infty\right)  \right]  \cap\left\{  x\neq
y\right\}  }\frac{\left[  d\left(  x\right)  \nu\left(  x\right)  +d\left(
y\right)  \nu\left(  y\right)  \right]  }{D}\cdot\\
&  \cdot\left[  Z\left(  y\right)  \left[  \nabla\psi\left(  x,t\right)
-Z\left(  x\right)  \nabla\psi\left(  P_{\partial}\left(  x\right)  ,t\right)
\right]  +Z\left(  x\right)  \left[  \nabla\psi\left(  y,t\right)  -Z\left(
y\right)  \nabla\psi\left(  P_{\partial}\left(  y\right)  ,t\right)  \right]
\right]  d\omega_{t}^{-}\left(  x,y\right)  dt
\end{align*}%
\begin{align*}
I_{2,2}  &  \rightarrow-\frac{1}{4\pi}\int\int\int_{\left[  \bar{\Omega}%
\times\bar{\Omega}\times\left[  0,\infty\right)  \right]  \cap\left\{  x\neq
y\right\}  }\frac{Z\left(  x\right)  Z\left(  y\right)  \left[  \nabla
\psi\left(  P_{\partial}\left(  x\right)  ,t\right)  -\nabla\psi\left(
P_{\partial}\left(  y\right)  ,t\right)  \right]  }{D}\cdot\\
&  \cdot\left[  d\left(  y\right)  \nu\left(  y\right)  -d\left(  x\right)
\nu\left(  x\right)  \right]  d\omega_{t}^{-}\left(  x,y\right)  dt
\end{align*}

This concludes the proof of the lemma.
\end{proof}

\subsubsection{Limit of the nonlinear terms: The terms $Q_{4}$ and $Q_{5}.$}

We finally precise the limit of the terms $Q_{4}$ and $Q_{5}.$

\begin{lemma}
\label{LQ4}Taking suitable subsequences $\varepsilon_{k}\rightarrow
0,\ \delta_{k}\rightarrow0$ as $k\rightarrow\infty$ we have:%
\begin{align}
Q_{4}  &  \rightarrow\hat{\mu}_{b}^{-}\left(  \varphi_{4}\right)  +\int
\int\int_{\left[  \bar{\Omega}\times\bar{\Omega}\times\left[  0,\infty\right)
\right]  \cap\left\{  x\neq y\right\}  }\nabla\psi\left(  x,t\right)
\frac{Z\left(  y\right)  h\left(  y\right)  }{2\pi}\cdot\label{G5E3}\\
&  \cdot\left[  \mathcal{G}_{t}\left(  Y\left(  x,y\right)  ,\lambda
_{1}\left(  x,y\right)  ,\lambda_{2}\left(  x,y\right)  \right)
+\mathit{g}_{n}\left(  Y\left(  x,y\right)  ,\lambda_{1}\left(  x,y\right)
,\lambda_{2}\left(  x,y\right)  \right)  \nu\left(  y\right)  \right]
d\omega_{t}^{-}\left(  x,y\right)  dt\nonumber
\end{align}%
\begin{equation}
Q_{5}\rightarrow-\int\int\int_{\left[  \bar{\Omega}\times\bar{\Omega}%
\times\left[  0,\infty\right)  \right]  }\nabla\psi\left(  x,t\right)
W\left(  x,y\right)  d\omega_{t}^{-}\left(  x,y\right)  dt \label{G5E4}%
\end{equation}
where $Q_{4},\ Q_{5}$ are as in (\ref{G2E3e}), (\ref{G2E3f}) respectively,
$\varphi_{4}$ is:%
\begin{equation}
\varphi_{4}\left(  y,Y,\lambda_{1},\lambda_{2}\right)  =\frac{h\left(
y\right)  }{2\pi}\nabla\psi\left(  y,t\right)  \cdot\left[  \mathcal{G}%
_{t}\left(  Y,\lambda_{1},\lambda_{2}\right)  +\mathit{g}_{n}\left(
Y,\lambda_{1},\lambda_{2}\right)  \nu\left(  y\right)  \right]  \label{M3E3}%
\end{equation}
and $\mathcal{G}_{t},\ \mathit{g}_{n}$ are as in (\ref{M2E3a}), (\ref{M2E3b}).
The measures $\omega_{t}^{-}\left(  x,y\right)  ,\ \hat{\mu}_{b}^{-},\ $are
respectively as in Lemmas \ref{LQ1}, \ref{LQ2}.
\end{lemma}

\begin{proof}
The proof is essentially similar to the one in Lemmas \ref{LQ2}, \ref{LQ3}. We
split $Q_{4}$ using the cutoff function $\eta.$ Then:%
\[
Q_{4}=\int\int\int\left[  \cdot\cdot\cdot\right]  \eta\left(  \frac{\left\vert
x-\tau\left(  y\right)  \right\vert }{\delta}\right)  d\omega_{\varepsilon
}^{-}\left(  x,y,t\right)  +
\]%
\[
+\int\int\int\left[  \cdot\cdot\cdot\right]  \left[  1-\eta\left(
\frac{\left\vert x-\tau\left(  y\right)  \right\vert }{\delta}\right)
\right]  d\omega_{\varepsilon}^{-}\left(  x,y,t\right)  \equiv Q_{4,1}+Q_{4,2}%
\]

Using (\ref{G4E7}) we have:
\[
Q_{4,1}=\hat{\mu}_{b,\delta,\varepsilon}^{-}\left(  \varphi_{4}\right)
+O\left(  \delta\right)
\]
with $\varphi_{4}$ as in (\ref{M3E3}). Taking the limit $\varepsilon
_{k}\rightarrow0,\;\delta_{k}\rightarrow0$ we then obtain (\ref{G5E3}). On the
other hand, we can take directly the limit in $Q_{5}$ due to the continuity of
$W\left(  x,y\right)  ,$ whence (\ref{G5E4}) follows.
\end{proof}

\begin{remark}
It is important to notice that the domain of integration in (\ref{G5E4})
includes the diagonal $\left\{  x=y\right\}  $ differently from the other
formulas where the measures $\left\{  \omega_{t}^{-}\right\}  $ appear (cf.
(\ref{G2E6}), (\ref{G4E3}), (\ref{G4E8}), (\ref{G5E3})).
\end{remark}

\subsubsection{Weak formulation limit equation for the first regularization.}

\bigskip

We can now collect the previous results as follows.

\bigskip

\begin{theorem}
\label{WeakC1}Let $u^{\varepsilon},v^{\varepsilon}$ be a solution of
(\ref{S2E1}), (\ref{S2E2}). Suppose that the measures $\mu,\ \hat{\mu}%
^{-},\ \omega^{-},\ \hat{\mu}_{b}^{-}$ are defined as in Proposition \ref{M1}
and Lemmas \ref{LQ1}, \ref{LQ2} respectively. Then, for any test function
$\psi\in C^{\infty}$ satisfying (\ref{G2E4}) we have:%
\begin{align}
&  -\int\psi\left(  x,0\right)  u_{0}\left(  x\right)  dx-\int\int\psi_{t}%
d\mu_{t}dt-\int\int\Delta\psi d\mu_{t}dt+\nonumber\\
&  +\frac{1}{4\pi}\int\int_{\bar{\Omega}\times S^{1}}\left(  \nu\cdot
\nabla^{2}\psi\left(  x,t\right)  \cdot\nu\right)  d\hat{\mu}_{t}^{-}\left(
x,\nu\right)  dt+\hat{\mu}_{b}^{-}\left(  \varphi\right)  -\nonumber\\
&  -\int\int\int_{\left[  \bar{\Omega}\times\bar{\Omega}\times\left[
0,\infty\right)  \right]  \cap\left\{  x=y\right\}  }\nabla\psi\left(
x,t\right)  W\left(  x,y\right)  d\omega_{t}^{-}\left(  x,y\right)
dt-\nonumber\\
&  -\int\int\int_{\left[  \bar{\Omega}\times\bar{\Omega}\times\left[
0,\infty\right)  \right]  \cap\left\{  x\neq y\right\}  }\left[  \nabla
_{x}G\left(  x,y\right)  \cdot\nabla\psi\left(  x,t\right)  \right]
d\omega_{t}^{-}\left(  x,y\right)  dt\nonumber\\
&  =0 \label{W6E2}%
\end{align}
with the test function $\varphi$ given by:%
\begin{align}
\varphi\left(  y,Y,\lambda_{1},\lambda_{2},t\right)   &  =\frac{Y\cdot
\nabla^{2}\psi\left(  y,t\right)  \cdot Y}{4\pi}+\left[  \frac{\nu\left(
y\right)  \cdot\nabla^{2}\psi\left(  y,t\right)  \cdot\nu\left(  y\right)
}{4\pi}\right]  \left(  \lambda_{1}+\lambda_{2}\right)  ^{2}+\nonumber\\
&  +\frac{h\left(  y\right)  }{2\pi}\nabla\psi\left(  y,t\right)  \cdot\left[
\mathcal{G}_{t}\left(  Y,\lambda_{1},\lambda_{2}\right)  +\mathit{g}%
_{n}\left(  Y,\lambda_{1},\lambda_{2}\right)  \nu\left(  y\right)  \right]
\label{M4E1}%
\end{align}
Moreover, the measures $\hat{\mu}_{t}^{-},\ \hat{\mu}_{b,t}^{-}$ satisfy
(\ref{G2E6a}), (\ref{G3A1}).
\end{theorem}

\begin{proof}
The result just follows taking the limit in (\ref{G2E3}). The limit of the
term $L_{1}$ is inmediate. The limit of the terms $Q_{1},...,Q_{5}$ can be
obtained using Lemmas \ref{LQ1}-\ref{LQ4}. We have $\varphi=\varphi
_{1}+\varphi_{2}+\varphi_{3}+\varphi_{4}$ with the functions $\varphi
_{j},\ j=1,...,4$ as in (\ref{G4E2}), (\ref{M3E1}), (\ref{M3E2}),
(\ref{M3E3}). We then obtain:%
\begin{align}
&  -\int\psi\left(  x,0\right)  u_{0}\left(  x\right)  dx-\int\int\psi_{t}%
d\mu_{t}dt-\int\int\Delta\psi d\mu_{t}dt+\nonumber\\
&  +\int\int_{\bar{\Omega}\times S^{1}}\left(  \frac{\nu\cdot\nabla^{2}%
\psi\left(  x,t\right)  \cdot\nu}{4\pi}\right)  d\hat{\mu}_{t}^{-}\left(
x,\nu\right)  dt+\hat{\mu}_{b}^{-}\left(  \varphi\right)  +\nonumber\\
&  +\int\int_{\left[  \bar{\Omega}\times\bar{\Omega}\times\left[
0,\infty\right)  \right]  \cap\left\{  x\neq y\right\}  }\int H_{1}\left(
x,y,t\right)  d\omega_{t}^{-}\left(  x,y\right)  dt+\nonumber\\
&  +\int\int\int_{\left[  \bar{\Omega}\times\bar{\Omega}\times\left[
0,\infty\right)  \right]  \cap\left\{  x\neq y\right\}  }H_{2}\left(
x,y,t\right)  d\omega_{t}^{-}\left(  x,y\right)  dt-\nonumber\\
&  -\int\int\int_{\left[  \bar{\Omega}\times\bar{\Omega}\times\left[
0,\infty\right)  \right]  }\nabla\psi\left(  x,t\right)  W\left(  x,y\right)
d\omega_{t}^{-}\left(  x,y\right)  dt\nonumber\\
&  -\int\int\int_{\left[  \bar{\Omega}\times\bar{\Omega}\times\left[
0,\infty\right)  \right]  \cap\left\{  x\neq y\right\}  }H_{3,1}\left(
x,y,t\right)  d\omega_{t}^{-}\left(  x,y\right)  dt-\nonumber\\
&  -\int\int\int_{\left[  \bar{\Omega}\times\bar{\Omega}\times\left[
0,\infty\right)  \right]  \cap\left\{  x\neq y\right\}  }H_{3,2}\left(
x,y,t\right)  d\omega_{t}^{-}\left(  x,y\right)  dt+\nonumber\\
&  +\int\int\int_{\left[  \bar{\Omega}\times\bar{\Omega}\times\left[
0,\infty\right)  \right]  \cap\left\{  x\neq y\right\}  }\nabla\psi\left(
x,t\right)  \frac{Z\left(  y\right)  h\left(  y\right)  }{2\pi}\cdot
\nonumber\\
&  \cdot\left[  \mathcal{G}_{t}\left(  Y\left(  x,y\right)  ,\lambda
_{1}\left(  x,y\right)  ,\lambda_{2}\left(  x,y\right)  \right)
+\mathit{g}_{n}\left(  Y\left(  x,y\right)  ,\lambda_{1}\left(  x,y\right)
,\lambda_{2}\left(  x,y\right)  \right)  \nu\left(  y\right)  \right]
d\omega_{t}^{-}\left(  x,y\right)  dt\nonumber\\
&  =0 \label{W6E1}%
\end{align}
where the functions $H_{1},H_{2},H_{3,1},H_{3,2}\ $ are as in (\ref{H1A1}),
(\ref{H1A2}), (\ref{H1A3}), (\ref{H1A4}) respectively.

Formula (\ref{W6E1}) is particularly convenient in order to check the
convergence of the different terms arising in the integrals, because the
measures $\omega_{t}^{-}\left(  x,y\right)  $ are integrated against
continuous functions in the region of integration. However, it can be written
in a more convenient form reversing the computations in Lemma \ref{LQ3} in
order to rewrite $H_{3,1},H_{3,2}$, using (\ref{G2E4}) and Lemma \ref{LRep}.
\end{proof}

\begin{remark}
It is important to take into account the presence in (\ref{W6E2}) of the
integral term containing $W$ and integrated in $\left\{  x=y\right\}  .$ This
term gives a nonzero contribution in the singular points of the measure
$\omega_{t}^{-}\left(  x,y\right)  .$
\end{remark}

\bigskip

\subsection{Second regularization}

Arguing in a completely similar manner with the second regularization
(\ref{S2E4}), (\ref{S2E5}) we can obtain the following result:

\begin{theorem}
\label{WeakC2} Let $u^{\varepsilon},v^{\varepsilon}$ be a solution of
(\ref{S2E4}), (\ref{S2E5}). Suppose that $\mu,\ \mu^{+}$ are as in Proposition
\ref{M2}. There exists measures $\hat{\mu}\in M^{+}\left(  \Omega\times
S^{1}\times\mathbb{R}^{+}\right)  ,\ \omega\in M^{+}\left(  \Omega\times
\Omega\times\mathbb{R}^{+}\right)  ,\ \hat{\mu}_{b}\in M^{+}\left(
\mathcal{M}\times\mathbb{R}^{+}\right)  $ defined as in Lemmas \ref{LQ1},
\ref{LQ2} with the functions $f_{\varepsilon}\left(  u^{\varepsilon}\right)  $
replaced by $u^{\varepsilon}.$ For any test function $\psi\in C^{\infty}$
satisfying (\ref{G2E4}) this family of measures satisfy:%
\begin{align}
&  -\int\psi\left(  x,0\right)  u_{0}\left(  x\right)  dx-\int\int\psi_{t}%
d\mu_{t}dt-\int\int\Delta\psi d\mu_{t}^{+}dt+\nonumber\\
&  +\frac{1}{4\pi}\int\int_{\bar{\Omega}\times S^{1}}\left(  \nu\cdot
\nabla^{2}\psi\left(  x,t\right)  \cdot\nu\right)  d\hat{\mu}_{t}\left(
x,\nu\right)  dt+\hat{\mu}_{b}\left(  \varphi\right)  -\nonumber\\
&  -\int\int\int_{\left[  \bar{\Omega}\times\bar{\Omega}\times\left[
0,\infty\right)  \right]  \cap\left\{  x=y\right\}  }\nabla\psi\left(
x,t\right)  W\left(  x,y\right)  d\omega_{t}\left(  x,y\right)  dt-\nonumber\\
&  -\int\int\int_{\left[  \bar{\Omega}\times\bar{\Omega}\times\left[
0,\infty\right)  \right]  \cap\left\{  x\neq y\right\}  }\left[  \nabla
_{x}G\left(  x,y\right)  \cdot\nabla\psi\left(  x,t\right)  \right]
d\omega_{t}\left(  x,y\right)  dt\nonumber\\
&  =0 \label{W6E1a}%
\end{align}
with the test function $\varphi$ as in (\ref{M4E1}). Moreover, we have:%
\begin{equation}
\int_{S^{1}}d\hat{\mu}_{t}\left(  \cdot,\nu\right)  =\sum_{x_{i}\left(
t\right)  \in S_{t}}\gamma_{i}\left(  t\right)  \delta_{x_{i}\left(  t\right)
}\ \ ,\ \ \gamma_{i}\left(  t\right)  =\left(  \alpha_{i}\left(  t\right)
\right)  ^{2}\ \ ,\ \ a.e.\ t\in\left[  0,\infty\right)  \label{W6E1b}%
\end{equation}%
\begin{align*}
d\omega_{t}\left(  x,y\right)   &  =\sum_{\left\{  x_{i}\left(  t\right)  \neq
x_{j}\left(  t\right)  \right\}  }\alpha_{i}\left(  t\right)  \alpha
_{j}\left(  t\right)  \delta_{x_{i}\left(  t\right)  }\left(  x\right)
\delta_{x_{j}\left(  t\right)  }\left(  y\right)  +\\
&  +\sum_{x_{i}\left(  t\right)  \in S_{t}}\alpha_{i}\left(  t\right)  \left[
\delta_{x_{i}\left(  t\right)  }\left(  x\right)  u\left(  y,t\right)
dy+\delta_{x_{i}\left(  t\right)  }\left(  y\right)  u\left(  x,t\right)
dx\right]  +u\left(  x,t\right)  u\left(  y,t\right)  dxdy
\end{align*}%
\[
\int_{\mathcal{M}_{2}\left[  \cdot\right]  }d\hat{\mu}_{b,t}\left(
\cdot,\sigma\right)  =\sum_{x_{i}\left(  t\right)  \in S_{t}}\gamma
_{b,i}\left(  t\right)  \delta_{x_{i}\left(  t\right)  }\left(  \cdot\right)
\ \ ,\ \ \gamma_{b,i}\left(  t\right)  =\left(  \alpha_{i}\left(  t\right)
\right)  ^{2}\ \ ,\ \ a.e.\ t\in\left[  0,\infty\right)
\]

\end{theorem}

\bigskip

\section{On the connection between the measures $\mu,\ \mu^{-}$ and $\mu^{+}.$
Separation Lemma. Proof of the different evolutions for the two
regularizations.}

The main result in this Section is the following:

\begin{theorem}
\label{RepMeas}The measure $\bar{\mu}$ in Proposition \ref{M1} (cf.
(\ref{T1E1b})) is related with the measures $\hat{\mu}^{-},\ \hat{\mu}_{b}%
^{-}$ defined in Lemma \ref{LQ1} by means of:%
\[
\frac{1}{8\pi}\int_{S^{1}}d\hat{\mu}_{t}^{-}\left(  \cdot,\nu\right)
+\frac{1}{8\pi}\int_{\mathcal{M}_{2}}d\hat{\mu}_{b}^{-}\left(  \cdot
,\nu\right)  =d\bar{\mu}_{t}\left(  \cdot\right)  \ \ \ a.e\ \ t\in\left[
0,\infty\right)
\]

The measure $\bar{\mu}^{+}$ in Proposition \ref{M2} (cf. (\ref{TN1})) is
related with the measures $\hat{\mu},\ \hat{\mu}_{b}\ $ in Theorem
\ref{WeakC2} by means of:%
\[
\frac{1}{8\pi}\int_{S^{1}}d\hat{\mu}_{t}\left(  \cdot,\nu\right)  +\frac
{1}{8\pi}\int_{\mathcal{M}_{2}}d\hat{\mu}_{b}\left(  \cdot,\nu\right)
=d\bar{\mu}_{t}^{+}\left(  \cdot\right)  \ \ \ a.e\ \ t\in\left[
0,\infty\right)
\]

\end{theorem}

A relevant consequence of Theorem \ref{RepMeas} is the following.

\bigskip

\begin{corollary}
\label{CorM}Let $\mu,\ \mu^{-},\ \hat{\mu}^{-}$ be solutions of (\ref{W6E1})
(cf. Theorem \ref{WeakC1}) as in (\ref{T1E1c}), (\ref{T1E1d}), (\ref{G2E6a}).
Suppose that there exists a set of positive measure $\mathcal{A}\subset\left[
0,\infty\right)  $ such that for $t_{0}\in\mathcal{A},$ $\beta_{i}^{-}\left(
t_{0}\right)  >8\pi.$ Then $\alpha_{i}\left(  t_{0}\right)  >\beta_{i}%
^{-}\left(  t_{0}\right)  $ $a.e.\ t_{0}\in\mathcal{A}.$

Let $\mu,\ \mu^{+},\ \hat{\mu}$ be solutions of (\ref{W6E1a}) (cf. Theorem
\ref{WeakC2}) as in (\ref{T2E1}), (\ref{T2E1a}), (\ref{W6E1b}). Suppose that
ther exists a set of positive measure $\mathcal{A}\subset\left[
0,\infty\right)  $ such that for $t_{0}\in\mathcal{A},$ $\alpha_{i}\left(
t_{0}^{+}\right)  >8\pi.$ Then $\beta_{i}^{+}\left(  t_{0}\right)  >\alpha
_{i}\left(  t_{0}\right)  .$
\end{corollary}

\begin{proof}
In the case of the first regularization (cf. (\ref{W6E1})), Theorem
\ref{RepMeas} combined with (\ref{T1E1c}), (\ref{G2E6a}) implies:%
\[
\left(  \beta_{i}^{-}\left(  t_{0}\right)  \right)  ^{2}\leq\gamma_{i}%
^{-}\left(  t_{0}\right)  =8\pi\alpha_{i}\left(  t_{0}\right)
\]
whence $\alpha_{i}\left(  t_{0}\right)  \geq\beta_{i}^{-}\left(  t_{0}\right)
\left(  \frac{\beta_{i}^{-}\left(  t_{0}\right)  }{8\pi}\right)  >\beta
_{i}^{-}\left(  t_{0}\right)  $ and the result follows. In the case of the
second regularization we use (\ref{T2E1a}), (\ref{W6E1b}) and Theorem
\ref{RepMeas} to obtain:%
\[
\left(  \alpha_{i}\left(  t\right)  \right)  ^{2}=\gamma_{i}\left(  t\right)
=8\pi\beta_{i}^{+}\left(  t_{0}\right)
\]
whence the conclusion follows in a similar way.
\end{proof}

\begin{remark}
Notice that a consequence of Corollary \ref{CorM} is that the weak
formulations (\ref{W6E2}), (\ref{W6E1b}) cannot define the same evolution as
soon as one of the Dirac masses at any singular point becomes larger than
$8\pi.$ Notice that intuitively, Corollary \ref{CorM} states that the effect
of the cutoff $f_{\varepsilon}\left(  u^{\varepsilon}\right)  $ in the case of
the first regularization, or the term $\varepsilon u^{\frac{7}{6}}$ in the
case of the second regularization becomes visible as soon as the masses become
larger than $8\pi.$
\end{remark}

\bigskip

We will give the details of the proof of Theorem \ref{RepMeas} for the
measures $\hat{\mu}^{-},\ \bar{\mu}$ for the points placed at the interior of
$\Omega,$ since the points at the boundary $\partial\Omega$ or the case of the
measures $\hat{\mu},\ \bar{\mu}^{+}$ can be studied with similar arguments.

The starting point in the Proof of Theorem \ref{RepMeas} will be the following
inequality that measures the rate of change of the mass in the neighbourhood
of a singular point in terms of the values of the measures $\hat{\mu}%
^{-},\ \bar{\mu}$ near such a point.

\bigskip

\begin{lemma}
\label{LI1}Suppose that the measures $\hat{\mu}^{-},\ \bar{\mu}$ solve
(\ref{W6E2}) for any $\psi\in C^{\infty}\left(  \bar{\Omega}\times
\mathbb{R}^{+}\right)  $ satisfying (\ref{G2E4}). Let $x_{0}\in\Omega,$
$2\rho<R,$ $B_{2R}\left(  x_{0}\right)  \subset\Omega,$ $0\leq T_{1}\leq
T_{2}<\infty.$ Suppose that $\eta_{R}\left(  x\right)  =\eta\left(
\frac{\left\vert x-x_{0}\right\vert }{R}\right)  $ with $\eta\in C^{\infty
}\left(  \mathbb{R}\right)  ,$ $\eta\left(  r\right)  =1$ if$\;0\leq r\leq1,$
$\eta^{\prime}\left(  r\right)  \leq0,\;\eta\left(  r\right)  =0$ if $r\geq2.$
We define also the test function:%
\begin{equation}
\varphi_{\rho}\left(  x\right)  =\left\{
\begin{array}
[c]{c}%
1-\frac{\left\vert x-x_{0}\right\vert ^{2}}{2\rho^{2}}\;\;,\;\;\left\vert
x-x_{0}\right\vert <\rho\\
\frac{\left(  \left\vert x-x_{0}\right\vert -2\rho\right)  _{-}^{2}}{2\rho
^{2}}\;\;,\;\;\left\vert x-x_{0}\right\vert >\rho
\end{array}
\right.  \label{W1E7}%
\end{equation}
where $\left(  s\right)  _{-}=s$ if $s\leq0$ and $\left(  s\right)  _{-}=0$ if
$s>0.$ We define also:%
\begin{equation}
U_{\rho}\left(  t;x_{0}\right)  =\int_{B_{\rho}\left(  x_{0}\right)  \times
S^{1}}d\hat{\mu}_{t}^{-}\left(  x,\nu\right)  \label{W4E1}%
\end{equation}

Then, the following inequalities hold:%
\begin{align}
&  \left\vert \int\varphi_{\rho}\left(  x\right)  d\mu_{T_{2}}\left(
x\right)  -\int\varphi_{\rho}\left(  x\right)  d\mu_{T_{1}}\left(  x\right)
-\frac{1}{\rho^{2}}\int_{T_{1}}^{T_{2}}\left[  \frac{U_{\rho}\left(
t;x_{0}\right)  }{4\pi}-2\int_{B_{\rho}\left(  x_{0}\right)  }d\mu_{t}\left(
x\right)  \right]  dt\right\vert \nonumber\\
&  \leq\frac{C}{\rho^{2}}\int_{T_{1}}^{T_{2}}\int_{\Omega\setminus B_{\rho
}\left(  x_{0}\right)  }\eta_{R}\left(  x\right)  d\mu_{t}\left(  x\right)
dt+\frac{C}{\rho}\int_{T_{1}}^{T_{2}}\int\int\left[  \frac{\eta_{R}\left(
x\right)  }{R}+1\right]  d\mu_{t}\left(  x\right)  d\mu_{t}\left(  y\right)
dt+\nonumber\\
&  +\frac{1}{4\pi\rho^{2}}\int_{T_{1}}^{T_{2}}\int_{B_{\rho}\left(
x_{0}\right)  \times B_{\rho}\left(  x_{0}\right)  \cap\left\{  x\neq
y\right\}  }\eta_{R}\left(  x\right)  \eta_{R}\left(  y\right)  d\mu
_{t}\left(  x\right)  d\mu_{t}\left(  y\right)  dt \label{Z1E1}%
\end{align}
for some constant $C$ depending on $\int d\mu_{0}\left(  x\right)  ,$ but
independent on $\,x_{0},\ R,\ \rho,\ T_{1},\ T_{2}.$
\end{lemma}

\begin{proof}
We will give the details of the proof for the points placed at the interior of
$\Omega,$ since the boundary points can be treated similarly. Suppose that
$\psi$ is any test function supported in a ball $B_{2\rho}\left(
x_{0}\right)  $ with $2\rho<R.$ Using Lemma \ref{L1} and symmetrizing we can
rewrite the last term on the right-hand side of (\ref{W6E2}) as:%
\begin{align}
&  \left\vert \int\int\int_{\left[  \bar{\Omega}\times\bar{\Omega}%
\times\left[  0,\infty\right)  \right]  \cap\left\{  x\neq y\right\}  }\left[
\nabla_{x}G\left(  x,y\right)  \cdot\nabla\psi\left(  x,t\right)  \right]
d\omega_{t}^{-}\left(  x,y\right)  dt+\int\int_{\left\vert x-y\right\vert
>0}\int H_{1}\left(  x,y,t\right)  d\omega_{t}^{-}\left(  x,y\right)
dt\right\vert \nonumber\\
&  \leq C\left\Vert \nabla_{x}\psi\right\Vert _{L^{\infty}\left(  \Omega
\times\mathbb{R}^{+}\right)  }\int\int\int d\omega_{t}^{-}\left(  x,y\right)
\label{W1E0}%
\end{align}
with $H_{1}\left(  x,y,t\right)  $ as in (\ref{H1A1}). Using the test function
$\eta_{R}\left(  x\right)  $ we can write:%
\begin{align}
&  \int\int_{\left\vert x-y\right\vert >0}\int H_{1}\left(  x,y,t\right)
d\omega_{t}^{-}\left(  x,y\right)  dt\label{W1E2}\\
&  =\int\int_{\left\vert x-y\right\vert >0}\int H_{1}\left(  x,y,t\right)
\eta_{R}\left(  x\right)  \eta_{R}\left(  y\right)  d\omega_{t}^{-}\left(
x,y\right)  dt+\nonumber\\
&  +\int\int_{\left\vert x-y\right\vert >0}\int H_{1}\left(  x,y,t\right)
\left(  1-\eta_{R}\left(  x\right)  \right)  \eta_{R}\left(  y\right)
d\omega_{t}^{-}\left(  x,y\right)  dt+\nonumber\\
&  +\int\int_{\left\vert x-y\right\vert >0}\int H_{1}\left(  x,y,t\right)
\eta_{R}\left(  x\right)  \left(  1-\eta_{R}\left(  y\right)  \right)
d\omega_{t}^{-}\left(  x,y\right)  dt\nonumber
\end{align}
where we use that a term containing the product $\left(  1-\eta_{R}\left(
x\right)  \right)  \left(  1-\eta_{R}\left(  y\right)  \right)  $ vanishes due
to the choice of the supports of $\eta_{R},\ \psi.$

The last two terms on the right-hand side of (\ref{W1E2}) can be bounded by:
\[
\frac{C}{R-2\rho}\left\Vert \nabla_{x}\psi\right\Vert _{L^{\infty}\left(
\Omega\times\mathbb{R}^{+}\right)  }\int\int\eta_{R}\left(  x\right)
d\omega_{t}^{-}\left(  x,y\right)
\]

Then (\ref{W1E2}) becomes:
\begin{align}
&  \int\int_{\left\vert x-y\right\vert >0}\int H_{1}\left(  x,y,t\right)
d\omega_{t}^{-}\left(  x,y\right)  dt\nonumber\\
&  =\int\int_{\left\vert x-y\right\vert >0}\int H_{1}\left(  x,y,t\right)
\eta_{R}\left(  x\right)  \eta_{R}\left(  y\right)  d\omega_{t}^{-}\left(
x,y\right)  dt+\nonumber\\
&  +O\left(  \frac{\left\Vert \nabla\psi\right\Vert _{L^{\infty}\left(
\Omega\times R^{+}\right)  }}{R-2\rho}\int\int\eta_{R}\left(  x\right)
d\omega_{t}^{-}\left(  x,y\right)  \right)  \label{W1E3}%
\end{align}

Combining (\ref{W1E0}), (\ref{W1E3}), and using the fact that $\psi=0$ at
$\partial\Omega$ we can rewrite (\ref{W6E2}) as:
\begin{align}
&  -\left(  \int\psi\left(  x,0\right)  u_{0}\left(  x\right)  dx\right)
-\int\int\psi_{t}\mu_{t}\left(  x\right)  dt-\int\int\Delta\psi d\mu
_{t}\left(  x\right)  dt\nonumber\\
&  +\int\int_{\Omega\times S^{1}}\left(  \frac{\nu\cdot\nabla^{2}\psi\left(
x,t\right)  \cdot\nu}{4\pi}\right)  d\hat{\mu}_{t}^{-}\left(  x,\nu\right)
dt+\nonumber\\
&  +\frac{1}{4\pi}\int\int_{\left\vert x-y\right\vert >0}\int\frac{\left[
\left(  x-y\right)  \cdot\left(  \nabla\psi\left(  x,t\right)  -\nabla
\psi\left(  y,t\right)  \right)  \right]  }{\left\vert x-y\right\vert ^{2}%
}\eta_{R}\left(  x\right)  \eta_{R}\left(  y\right)  d\omega_{t}^{-}\left(
x,y\right)  dt-\nonumber\\
&  +O\left(  \int\int\int\left[  \frac{\eta_{R}\left(  x\right)  }{R-2\rho
}\left\Vert \nabla_{x}\psi\right\Vert _{L^{\infty}\left(  \Omega
\times\mathbb{R}^{+}\right)  }+\left\Vert \nabla_{x}\psi\right\Vert
_{L^{\infty}\left(  \Omega\times\mathbb{R}^{+}\right)  }\right]  d\omega
_{t}^{-}\left(  x,y\right)  dt\right) \nonumber\\
&  =0 \label{W1E5}%
\end{align}

Let us consider a function $H_{\varepsilon}\left(  t;T_{1,}T_{2}\right)
,\;H_{\varepsilon}\in C^{\infty},\;0\leq H_{\varepsilon}\leq1$ such that
$H_{\varepsilon}\rightarrow\chi_{\left[  T_{1},T_{2}\right]  }$ as
$\varepsilon\rightarrow0$ in $L^{1}\left(  \mathbb{R}^{+}\right)  ,$
$H_{\varepsilon}\rightharpoonup\delta_{T_{1}}-\delta_{T_{2}}$ in $\left(
C\left(  \mathbb{R}^{+}\right)  \right)  ^{\ast},$ where $\chi_{\left[
T_{1},T_{2}\right]  }$ is the characteristic function of the interval $\left[
T_{1},T_{2}\right]  $ and $\delta_{T}$ is a Dirac mass at $t=T.$ Replacing
$\psi\left(  x,t\right)  $ by $\psi\left(  x,t\right)  H_{\varepsilon}\left(
t;T_{1},T_{2}\right)  $ and taking the limit $\varepsilon\rightarrow0$ we
obtain for $a.e.\;T_{1},T_{2}>0,$ and then for all $T_{1},T_{2}>0$ due to the
absolute continuity of the integration on $t\ $%
\begin{align}
&  -\left(  \int\psi\left(  x,T_{1}\right)  d\mu_{T_{1}}\left(  x\right)
\right)  +\int\psi\left(  x,T_{2}\right)  d\mu_{T_{2}}\left(  x\right)
-\int_{T_{1}}^{T_{2}}\int\frac{\partial\psi}{\partial t}\mu_{t}\left(
x\right)  dt-\nonumber\\
&  -\int_{T_{1}}^{T_{2}}\int\Delta\psi d\mu_{t}\left(  x\right)
dt+\int_{T_{1}}^{T_{2}}\int_{\Omega\times S^{1}}\left(  \frac{\nu\cdot
\nabla^{2}\psi\left(  x,t\right)  \cdot\nu}{4\pi}\right)  d\hat{\mu}_{t}%
^{-}\left(  x,\nu\right)  dt+\nonumber\\
&  +\int_{T_{1}}^{T_{2}}\int_{\left\vert x-y\right\vert >0}\int H_{1}\left(
x,y,t\right)  \eta_{R}\left(  x\right)  \eta_{R}\left(  y\right)  d\omega
_{t}^{-}\left(  x,y\right)  dt+\nonumber\\
&  +O\left(  \int_{T_{1}}^{T_{2}}\left\Vert \nabla_{x}\psi\right\Vert
_{L^{\infty}\left(  \Omega\times R^{+}\right)  }\int\int\left[  \frac{\eta
_{R}\left(  x\right)  }{R-2\rho}+1\right]  d\omega_{t}^{-}\left(  x,y\right)
dt\right) \nonumber\\
&  =0 \label{W1E6}%
\end{align}

We now use in (\ref{W1E6}) the test function $\psi\left(  x,t\right)
=\varphi_{\rho}\left(  x\right)  $ (cf. (\ref{W1E7})). Notice that
$\varphi_{\rho}\notin C^{\infty},$ but since $\varphi_{\rho}\in C^{1,1}$ it
can be used by means of a density argument. Then:
\begin{align}
&  -\left(  \int\varphi_{\rho}\left(  x\right)  d\mu_{T_{1}}\left(  x\right)
\right)  +\int\varphi_{\rho}\left(  x\right)  d\mu_{T_{2}}\left(  x\right)
-\int_{T_{1}}^{T_{2}}\int\Delta\varphi_{\rho}d\mu_{t}\left(  x\right)
dt+\nonumber\\
&  +\int_{T_{1}}^{T_{2}}\int_{\Omega\times S^{1}}\left(  \frac{\nu\cdot
\nabla^{2}\varphi_{\rho}\left(  x\right)  \cdot\nu}{4\pi}\right)  d\hat{\mu
}_{t}^{-}\left(  x,\nu\right)  dt+\nonumber\\
&  +\frac{1}{4\pi}\int_{T_{1}}^{T_{2}}\int_{\left\vert x-y\right\vert >0}%
\int\frac{\left[  \left(  x-y\right)  \cdot\left(  \nabla\varphi_{\rho}\left(
x\right)  -\nabla\varphi_{\rho}\left(  y\right)  \right)  \right]
}{\left\vert x-y\right\vert ^{2}}\eta_{R}\left(  x\right)  \eta_{R}\left(
y\right)  d\omega_{t}^{-}\left(  x,y\right)  dt+\nonumber\\
&  +\int_{T_{1}}^{T_{2}}O\left(  \frac{1}{\rho}\int\int\left[  \frac{\eta
_{R}\left(  x\right)  }{R-2\rho}+1\right]  d\omega_{t}^{-}\left(  x,y\right)
\right)  dt\nonumber\\
&  =0 \label{A1E2}%
\end{align}

We now estimate the term containing the measure $\hat{\mu}^{-}.$ Notice that:
\begin{align*}
&  \int_{T_{1}}^{T_{2}}\int_{\Omega\times S^{1}}\left(  \frac{\nu\cdot
\nabla^{2}\varphi_{\rho}\left(  x\right)  \cdot\nu}{4\pi}\right)  d\hat{\mu
}_{t}^{-}\left(  x,\nu\right)  dt\\
&  =-\frac{1}{4\pi\rho^{2}}\int_{T_{1}}^{T_{2}}\int_{B_{\rho}\left(
x_{0}\right)  \times S^{1}}d\hat{\mu}_{t}^{-}\left(  x,\nu\right)  dt+\\
&  +\int_{T_{1}}^{T_{2}}\int_{\left[  B_{2\rho}\left(  x_{0}\right)  \setminus
B_{\rho}\left(  x_{0}\right)  \right]  \times S^{1}}\left(  \frac{\nu
\cdot\nabla^{2}\varphi_{\rho}\left(  x\right)  \cdot\nu}{4\pi}\right)
d\hat{\mu}_{t}^{-}\left(  x,\nu\right)  dt
\end{align*}
whence:
\begin{align}
&  \left\vert \int_{T_{1}}^{T_{2}}\int_{\Omega\times S^{1}}\left(  \frac
{\nu\cdot\nabla^{2}\varphi_{\rho}\left(  x\right)  \cdot\nu}{4\pi}\right)
d\hat{\mu}_{t}^{-}\left(  x,\nu\right)  dt+\frac{1}{4\pi\rho^{2}}\int_{T_{1}%
}^{T_{2}}\int_{B_{\rho}\left(  x_{0}\right)  \times S^{1}}d\hat{\mu}_{t}%
^{-}\left(  x,\nu\right)  dt\right\vert \nonumber\\
&  \leq\frac{C}{\rho^{2}}\int_{T_{1}}^{T_{2}}\int_{B_{2\rho}\left(
x_{0}\right)  \setminus B_{\rho}\left(  x_{0}\right)  }d\mu_{t}\left(
x\right)  dt \label{A1E3}%
\end{align}

Using the fact that $\left\vert \nabla\varphi_{\rho}\left(  x\right)
-\nabla\varphi_{\rho}\left(  y\right)  \right\vert \leq\frac{C}{\rho^{2}%
}\left\vert x-y\right\vert $ for $\left\vert x-x_{0}\right\vert \leq\rho,\;$
$\left\vert y-x_{0}\right\vert \geq\rho$ as well as (\ref{A1E2}), (\ref{A1E3})
it then follows that:
\begin{align}
&  \int\varphi_{\rho}\left(  x\right)  d\mu_{T_{2}}\left(  x\right)
-\int\varphi_{\rho}\left(  x\right)  d\mu_{T_{1}}\left(  x\right) \nonumber\\
&  =\frac{2}{\rho^{2}}\int_{T_{1}}^{T_{2}}\int_{B_{2\rho}\left(  x_{0}\right)
\setminus B_{\rho}\left(  x_{0}\right)  }\frac{\left\vert x-x_{0}\right\vert
-\rho}{\left\vert x-x_{0}\right\vert }d\mu_{t}\left(  x\right)  dt+\nonumber\\
&  +\frac{1}{\rho^{2}}\int_{T_{1}}^{T_{2}}\left[  \frac{U_{\rho}\left(
t;x_{0}\right)  }{4\pi}-2\int_{B_{\rho}\left(  x_{0}\right)  }d\mu_{t}\left(
x\right)  \right]  dt+\nonumber\\
&  +\int_{T_{1}}^{T_{2}}O\left(  \frac{1}{\rho^{2}}\int_{\Omega\setminus
B_{\rho}\left(  x_{0}\right)  }\int_{\Omega}\eta_{R}\left(  x\right)  \eta
_{R}\left(  y\right)  d\omega_{t}^{-}\left(  x,y\right)  +\frac{1}{\rho^{2}%
}\int_{B_{2\rho}\left(  x_{0}\right)  \setminus B_{\rho}\left(  x_{0}\right)
}d\mu_{t}\left(  x\right)  \right)  dt+\nonumber\\
&  +\int_{T_{1}}^{T_{2}}O\left(  \frac{1}{\rho}\int\int\left[  \frac{\eta
_{R}\left(  x\right)  }{R-2\rho}+1\right]  d\omega_{t}^{-}\left(  x,y\right)
\right)  dt+\nonumber\\
&  +\frac{1}{4\pi\rho^{2}}\int_{T_{1}}^{T_{2}}\int_{B_{\rho}\left(
x_{0}\right)  \times B_{\rho}\left(  x_{0}\right)  \cap\left\{  x\neq
y\right\}  }\eta_{R}\left(  x\right)  \eta_{R}\left(  y\right)  d\omega
_{t}^{-}\left(  x,y\right)  dt \label{W2E2}%
\end{align}

Using (\ref{Wmass}) as well as the inequality $\frac{\left\vert x-x_{0}%
\right\vert -\rho}{\left\vert x-x_{0}\right\vert }\leq\eta_{R}\left(
x\right)  $ for $\left\vert x-x_{0}\right\vert \leq2\rho$ we obtain:
\begin{align*}
&  \int\varphi_{\rho}\left(  x\right)  d\mu_{T_{2}}\left(  x\right)
-\int\varphi_{\rho}\left(  x\right)  d\mu_{T_{1}}\left(  x\right) \\
&  \leq\frac{1}{\rho^{2}}\int_{T_{1}}^{T_{2}}\left[  \frac{U_{\rho}\left(
t;x_{0}\right)  }{4\pi}-2\int_{B_{\rho}\left(  x_{0}\right)  }d\mu_{t}\left(
x\right)  \right]  dt+\frac{C}{\rho^{2}}\int_{T_{1}}^{T_{2}}\int
_{\Omega\setminus B_{\rho}\left(  x_{0}\right)  }\eta_{R}\left(  x\right)
d\mu_{t}\left(  x\right)  dt+\\
&  +\frac{C}{\rho}\int_{T_{1}}^{T_{2}}\int\int\left[  \frac{\eta_{R}\left(
x\right)  }{R}+1\right]  d\mu_{t}\left(  x\right)  d\mu_{t}\left(  y\right)
dt+\\
&  +\frac{1}{4\pi\rho^{2}}\int_{T_{1}}^{T_{2}}\int_{B_{\rho}\left(
x_{0}\right)  \times B_{\rho}\left(  x_{0}\right)  \cap\left\{  x\neq
y\right\}  }\eta_{R}\left(  x\right)  \eta_{R}\left(  y\right)  d\mu
_{t}\left(  x\right)  d\mu_{t}\left(  y\right)  dt
\end{align*}

\end{proof}

\bigskip

We define some auxiliary sets that will be used in the following.

\begin{definition}
For any $\sigma,$ $\delta_{1},\;\delta_{2}>0$ we define:%
\begin{align*}
I_{\delta_{1},\delta_{2}}  &  =\left\{  t\in\left[  0,\infty\right)
:\forall\text{ }x\in S_{t},\;B_{2\delta_{1}}\left(  x\right)  \cap
S_{t}=\left\{  x\right\}  ,\;\forall\text{ }Y\in\Omega,\text{\ }%
\int_{B_{2\delta_{1}}\left(  Y\right)  \setminus S_{t}}d\mu_{t}\left(
x\right)  <\delta_{2}^{2}\right\} \\
I_{\delta_{1},\delta_{2}}^{+}\left(  X\right)   &  =\left\{  t\in\left[
0,\infty\right)  :\frac{U_{\delta_{1}}\left(  t;X\right)  }{4\pi}%
>2\int_{B_{\delta_{1}}\left(  X\right)  }d\mu_{t}\left(  x\right)  +\delta
_{2}\right\} \\
I_{\delta_{1},\delta_{2}}^{-}\left(  X\right)   &  =\left\{  t\in\left[
0,\infty\right)  :\frac{U_{\delta_{1}}\left(  t;X\right)  }{4\pi}%
<2\int_{B_{\delta_{1}}\left(  X\right)  }d\mu_{t}\left(  x\right)  -\delta
_{2}\right\}
\end{align*}
where $U_{\delta_{1}}\left(  t;X\right)  $ is defined as in (\ref{W4E1}) and
$X\in\Omega.$
\end{definition}

In the proof of the following result, it will be convenient to make more
explicit the dependence on the singular point of the Dirac masses $\alpha
_{j},\ \gamma_{j}^{-}$ in (\ref{T1E1c}), (\ref{G2E6a}). We write:%
\begin{align}
\mu_{t}  &  =\sum_{x_{j}\left(  t\right)  \in S_{t}}\alpha\left(
t;x_{j}\right)  \delta_{x_{j}\left(  t\right)  }+u\left(  \cdot,t\right)
dx\;\;,\;\;u\left(  \cdot,t\right)  \in L^{1}\left(  \Omega\right)
\label{Z2E3}\\
\int_{S^{1}}d\hat{\mu}_{t}^{-}\left(  \cdot,\nu\right)   &  =\sum
_{x_{j}\left(  t\right)  \in S_{t}}\gamma^{-}\left(  t;x_{j}\right)
\delta_{x_{j}\left(  t\right)  } \label{Z2E4}%
\end{align}

\begin{lemma}
There exists $\delta_{0}>0$ small depending only on $\int_{\Omega}d\mu
_{0}\left(  x\right)  $\ such that for any $\delta_{1},\;\delta_{2}\in\left(
0,\delta_{0}\right)  $ we have
\begin{equation}
\left\vert I_{\delta_{1},\delta_{2}}\cap I_{\delta_{1},\delta_{2}}^{+}\left(
X\right)  \right\vert =0 \label{W4E2}%
\end{equation}%
\begin{equation}
\left\vert I_{\delta_{1},\delta_{2}}\cap I_{\delta_{1},\delta_{2}}^{-}\left(
X\right)  \right\vert =0 \label{W4E2bis}%
\end{equation}
for any $X\in\Omega${\Huge .}
\end{lemma}

\begin{proof}
We will prove (\ref{W4E2}) since the proof of (\ref{W4E2bis}) is similar. We
argue by contradiction. Suppose that for some $X\in\Omega,$ $\left\vert
I_{\delta_{1},\delta_{2}}\cap I_{\delta_{1},\delta_{2}}^{+}\left(  X\right)
\right\vert >0.$

A well known result \cite{Stein} states that for any measurable set
$A\subset\mathbb{R}:$%
\begin{equation}
\lim_{\varepsilon\rightarrow0}\frac{1}{\varepsilon}\left\vert \left\{
t-t_{0}<\varepsilon\right\}  \cap A\right\vert =1\text{ \ \ ,\ \ \ \ }%
a.e.\;\;\;t_{0}\in A \label{W4E7a}%
\end{equation}

Let $A=I_{\delta_{1},\delta_{2}}\cap I_{\delta_{1},\delta_{2}}^{+}\left(
X\right)  $ and fix $n$ integer such that $n\delta_{2}\in\left(  1,2\right)
.$ For every $t_{0}\in A$ such that (\ref{W4E7a}) holds, there exists a
sequence $\left\{  \varepsilon_{\ell}>0\right\}  ,\;\varepsilon_{\ell
}\rightarrow0$ such that
\begin{equation}
\left\vert \left\{  t-t_{0}<\varepsilon_{\ell}\right\}  \setminus A\right\vert
\leq\frac{\delta_{2}}{K}\varepsilon_{\ell} \label{W4E7b}%
\end{equation}
where $K>0$ is a fixed numerical constant independent on $\delta_{1}%
,\delta_{2}$ that will be precised later.

Suppose that $K>8.$ Then, since $\frac{\delta_{2}\varepsilon_{\ell}}{K}%
\leq\frac{2\varepsilon_{\ell}}{Kn}\leq\frac{\varepsilon_{\ell}}{2\left(
n+1\right)  },$ we can obtain, for each $\ell$, $n$ times $t_{i}^{\ell}\in
A,\;i=1,...,n$ such that:
\begin{align}
t_{0}  &  =t_{1}^{\ell}<...<t_{n}^{\ell}=t_{0}+\varepsilon_{\ell}\nonumber\\
\frac{\varepsilon_{\ell}}{2n}  &  \leq\left(  t_{i+1}^{\ell}-t_{i}^{\ell
}\right)  \leq\frac{\varepsilon_{\ell}}{n}\;\;,\;\;i=1,...,\left(  n-1\right)
\label{W4E7c}%
\end{align}

We now prove that for $t\in I_{\delta_{1},\delta_{2}}\cap I_{\delta_{1}%
,\delta_{2}}^{+}$ there exists a singular point $Y\in S_{t}\cap B_{\delta_{1}%
}\left(  X\right)  .$ Indeed, notice that for $t\in A=I_{\delta_{1},\delta
_{2}}\cap I_{\delta_{1},\delta_{2}}^{+}$ we have:
\[
U_{\delta_{1}}\left(  t;X\right)  >\delta_{2}%
\]
On the other hand by (\ref{G2E6a}):
\[
U_{\delta_{1}}\left(  t;X\right)  =\int_{B_{\delta_{1}}\left(  X\right)
\times S^{1}}d\hat{\mu}_{t}^{-}\left(  x,\nu\right)  \leq\left(
\int_{B_{\delta_{1}}\left(  X\right)  }d\mu_{t}\left(  x\right)  \right)  ^{2}%
\]

Suppose that $S_{t}\cap B_{\delta_{1}}\left(  X\right)  =\emptyset.$ Then,
since $t\in A\subset I_{\delta_{1},\delta_{2}}$ it follows from the definition
of $I_{\delta_{1},\delta_{2}}$ that:
\[
\int_{B_{\delta_{1}}\left(  X\right)  }d\mu_{t}\left(  x\right)  \leq
\delta_{2}^{2}%
\]

Therefore:
\[
\delta_{2}<U_{\delta_{1}}\left(  t;X\right)  \leq\delta_{2}^{4}%
\]
and this gives a contradiction for $\delta_{2}$ sufficiently small. Then for
$t\in A$ we have $S_{t}\cap B_{\delta_{1}}\left(  X\right)  \neq\emptyset.$
Moreover $S_{t}\cap B_{\delta_{1}}\left(  X\right)  =\left\{  Y\right\}  $ for
some $Y\in\Omega$ and $\left[  B_{2\delta_{1}}\left(  Y\right)  \setminus
\left\{  Y\right\}  \right]  \cap S_{t}=\emptyset.$

Given the sequence of times $\left\{  t_{i}^{\ell}:i=1,...,n\right\}  ,$ let
us denote as $Y_{i}^{\ell}$ the corresponding singular points $Y_{i}^{\ell}\in
S_{t}\cap B_{\delta_{1}}\left(  X\right)  .$

It now follows using Lemma \ref{contS} that
\[
\left\vert Y_{i}^{\ell}-Y_{i+1}^{\ell}\right\vert \leq L\sqrt{\left\vert
t_{i}^{\ell}-t_{i+1}^{\ell}\right\vert }\leq L\sqrt{\frac{\varepsilon_{\ell}%
}{n}}\;\;,\;\;1=1,...,\left(  n-1\right)
\]
where $L>0$ is a constant depending only on $\int_{\Omega}d\mu_{0}\left(
x\right)  $ and $\Omega.$

We then apply (\ref{Z1E1}) with $T_{1}=t_{i}^{\ell},\;T_{2}=t_{i+1}^{\ell
},\;x_{0}=\frac{Y_{i}^{\ell}+Y_{i+1}^{\ell}}{2}.$ We will assume also that:
\begin{equation}
\rho=4L\sqrt{\delta_{2}\varepsilon_{\ell}}\;\;,\;\;R=D\rho\label{Z1E6a}%
\end{equation}
where the constant $D$ will be chosen depending only on $\int_{\Omega}d\mu
_{0}.$ Notice that $t\in A$ implies $t\in I_{\delta_{1},\delta_{2}}\cap
I_{\delta_{1},\delta_{2}}^{+}\left(  X\right)  $ for the above mentioned value
of $\delta_{1}.$ We only assume for the moment that:
\begin{equation}
D\geq8 \label{Z1E6}%
\end{equation}

Then:%
\begin{align*}
&  \int\varphi_{\rho}\left(  x\right)  d\mu_{t_{i+1}^{\ell}}\left(  x\right)
-\int\varphi_{\rho}\left(  x\right)  d\mu_{t_{i}^{\ell}}\left(  x\right) \\
&  \geq\frac{1}{\rho^{2}}\int_{\left[  t_{i}^{\ell},t_{i+1}^{\ell}\right]
\cap A}\left[  \frac{U_{\rho}\left(  t;Y_{i}^{\ell}\right)  }{4\pi}%
-2\int_{B_{\rho}\left(  x_{0}\right)  }d\mu_{t}\left(  x\right)  \right]
dt+\frac{1}{\rho^{2}}\int_{\left[  t_{i}^{\ell},t_{i+1}^{\ell}\right]
\setminus A}\left[  \frac{U_{\rho}\left(  t;Y_{i}^{\ell}\right)  }{4\pi}%
-2\int_{B_{\rho}\left(  x_{0}\right)  }d\mu_{t}\left(  x\right)  \right]
dt-\\
&  -\frac{C}{\rho^{2}}\int_{t_{i}^{\ell}}^{t_{i+1}^{\ell}}\int_{\Omega
\setminus B_{\rho}\left(  x_{0}\right)  }\eta_{R}\left(  x\right)  d\mu
_{t}\left(  x\right)  dt-\\
&  -\frac{C}{\rho}\int_{t_{i}^{\ell}}^{t_{i+1}^{\ell}}\int\left[  \frac
{\eta_{R}\left(  x\right)  }{R}+1\right]  d\mu_{t}\left(  x\right)  \int
d\mu_{t}\left(  y\right)  dt-\frac{1}{4\pi\rho^{2}}\int_{t_{i}^{\ell}%
}^{t_{i+1}^{\ell}}\int_{B_{\rho}\left(  x_{0}\right)  \times B_{\rho}\left(
x_{0}\right)  \cap\left\{  x\neq y\right\}  }d\mu_{t}\left(  x\right)
d\mu_{t}\left(  y\right)  dt
\end{align*}

Notice that, due to our choice of $x_{0}$ and the definition of $\varphi
_{\rho}$ we have that
\begin{align*}
\left\vert Y_{i+1}^{\ell}-x_{0}\right\vert  &  \leq\rho\;\;,\;\;\left\vert
Y_{i}^{\ell}-x_{0}\right\vert \leq\rho\\
\varphi_{\rho}\left(  Y_{i}^{\ell}\right)   &  =\varphi_{\rho}\left(
Y_{i+1}^{\ell}\right)  \geq\frac{1}{2}%
\end{align*}

(It is important to take into account that $x_{0}$ depends on $i$). Then:
\begin{align*}
&  \alpha\left(  Y_{i+1}^{\ell};t_{i+1}^{\ell}\right)  -\alpha\left(
Y_{i}^{\ell};t_{i}^{\ell}\right)  +\int_{\Omega\setminus\left\{  Y_{i+1}%
^{\ell}\right\}  }\varphi_{\rho}\left(  x;Y_{i}^{\ell}\right)  d\mu
_{t_{i+1}^{\ell}}\left(  x\right)  -\int_{\Omega\setminus\left\{  Y_{i}^{\ell
}\right\}  }\varphi_{\rho}\left(  x;Y_{i}^{\ell}\right)  d\mu_{t_{i}^{\ell}%
}\left(  x\right) \\
&  \geq\frac{1}{\rho^{2}}\int_{\left[  t_{i}^{\ell},t_{i+1}^{\ell}\right]
\cap A}\left[  \frac{U_{\rho}\left(  t;Y_{i}^{\ell}\right)  }{4\pi}%
-2\int_{B_{\rho}\left(  x_{0}\right)  }d\mu_{t}\left(  x\right)  \right]
dt+\frac{1}{\rho^{2}}\int_{\left[  t_{i}^{\ell},t_{i+1}^{\ell}\right]
\setminus A}\left[  \frac{U_{\rho}\left(  t;Y_{i}^{\ell}\right)  }{4\pi}%
-2\int_{B_{\rho}\left(  x_{0}\right)  }d\mu_{t}\left(  x\right)  \right]
dt-\\
&  -\frac{C}{\rho^{2}}\int_{t_{i}^{\ell}}^{t_{i+1}^{\ell}}\int_{\Omega
\setminus B_{\rho}\left(  x_{0}\right)  }\eta_{R}\left(  x\right)  d\mu
_{t}\left(  x\right)  dt-\\
&  -\frac{C}{\rho}\int_{t_{i}^{\ell}}^{t_{i+1}^{\ell}}\int\left[  \frac
{\eta_{R}\left(  x\right)  }{R}+1\right]  d\mu_{t}\left(  x\right)  \int
d\mu_{t}\left(  y\right)  dt-\frac{1}{4\pi\rho^{2}}\int_{t_{i}^{\ell}%
}^{t_{i+1}^{\ell}}\int_{B_{\rho}\left(  x_{0}\right)  \times B_{\rho}\left(
x_{0}\right)  \cap\left\{  x\neq y\right\}  }d\mu_{t}\left(  x\right)
d\mu_{t}\left(  y\right)  dt
\end{align*}
where we use the fact that $t_{i}^{\ell},t_{i+1}^{\ell}\in I_{\delta
_{1},\delta_{2}}$ and we write explicitly the dependence on the center $x_{0}$
for $\varphi_{\rho}=\varphi_{\rho}\left(  x;x_{0}\right)  .$ Notice that we
use also the fact that $U_{\rho}\left(  t;Y_{i}^{\ell}\right)  =U_{\rho
}\left(  t;x_{0}\right)  $ for $t\in\left[  t_{i}^{\ell},t_{i+1}^{\ell
}\right]  \cap A.$

Using the global boundedness of $\int d\mu_{t}\left(  x\right)  $ as well as
(\ref{W4E7c}) as the fact that $R\leq1$ and the definition of $\eta_{R}$ we
obtain:
\begin{align*}
&  \alpha\left(  Y_{i+1}^{\ell};t_{i+1}^{\ell}\right)  -\alpha\left(
Y_{i}^{\ell};t_{i}^{\ell}\right)  +\int_{\Omega\setminus\left\{  Y_{i+1}%
^{\ell}\right\}  }\varphi_{\rho}\left(  x;Y_{i}^{\ell}\right)  d\mu
_{t_{i+1}^{\ell}}\left(  x\right)  -\int_{\Omega\setminus\left\{  Y_{i}^{\ell
}\right\}  }\varphi_{\rho}\left(  x;Y_{i}^{\ell}\right)  d\mu_{t_{i}^{\ell}%
}\left(  x\right) \\
&  \geq\frac{1}{\rho^{2}}\int_{\left[  t_{i}^{\ell},t_{i+1}^{\ell}\right]
\cap A}\left[  \frac{U_{\rho}\left(  t;x_{0}\right)  }{4\pi}-2\int_{B_{\rho
}\left(  x_{0}\right)  }d\mu_{t}\left(  x\right)  \right]  dt-\frac{C}%
{\rho^{2}}\int_{\left[  t_{i}^{\ell},t_{i+1}^{\ell}\right]  \setminus A}dt\\
&  -\frac{C}{\rho^{2}}\int_{\left[  t_{i}^{\ell},t_{i+1}^{\ell}\right]  \cap
A}\int_{B_{2R}\left(  x_{0}\right)  \setminus B_{\rho}\left(  x_{0}\right)
}d\mu_{t}\left(  x\right)  dt-\frac{C}{\rho^{2}}\int_{\left[  t_{i}^{\ell
},t_{i+1}^{\ell}\right]  \setminus A}\int_{B_{2R}\left(  x_{0}\right)
\setminus B_{\rho}\left(  x_{0}\right)  }d\mu_{t}\left(  x\right)  dt-\\
&  -\frac{C}{\rho R}\int_{t_{i}^{\ell}}^{t_{i+1}^{\ell}}dt-\frac{1}{4\pi
\rho^{2}}\int_{\left[  t_{i}^{\ell},t_{i+1}^{\ell}\right]  \cap A}%
\int_{B_{\rho}\left(  x_{0}\right)  \times B_{\rho}\left(  x_{0}\right)
\cap\left\{  x\neq y\right\}  }d\mu_{t}\left(  x\right)  d\mu_{t}\left(
y\right)  dt-\\
&  -\frac{1}{4\pi\rho^{2}}\int_{\left[  t_{i}^{\ell},t_{i+1}^{\ell}\right]
\setminus A}\int_{B_{\rho}\left(  x_{0}\right)  \times\left[  B_{\rho}\left(
x_{0}\right)  \right]  \cap\left\{  x\neq y\right\}  }d\mu_{t}\left(
x\right)  d\mu_{t}\left(  y\right)  dt
\end{align*}
and using again (\ref{W4E7b}) as well as the boundedness of $\int d\mu
_{t}\left(  x\right)  =\int d\mu_{0}\left(  x\right)  $ we obtain:
\begin{align}
&  \alpha\left(  Y_{i+1}^{\ell};t_{i+1}^{\ell}\right)  -\alpha\left(
Y_{i}^{\ell};t_{i}^{\ell}\right)  +\int_{\Omega\setminus\left\{  Y_{i+1}%
^{\ell}\right\}  }\varphi_{\rho}\left(  x;Y_{i}^{\ell}\right)  d\mu
_{t_{i+1}^{\ell}}\left(  x\right)  -\int_{\Omega\setminus\left\{  Y_{i}^{\ell
}\right\}  }\varphi_{\rho}\left(  x;Y_{i}^{\ell}\right)  d\mu_{t_{i}^{\ell}%
}\left(  x\right) \nonumber\\
&  \geq\frac{1}{\rho^{2}}\int_{\left[  t_{i}^{\ell},t_{i+1}^{\ell}\right]
\cap A}\left[  \frac{U_{\rho}\left(  t;x_{0}\right)  }{4\pi}-2\int_{B_{\rho
}\left(  x_{0}\right)  }d\mu_{t}\left(  x\right)  \right]  dt-\frac{C}%
{\rho^{2}}\int_{\left[  t_{i}^{\ell},t_{i+1}^{\ell}\right]  \setminus
A}dt-\nonumber\\
&  -\frac{C}{\rho^{2}}\int_{\left[  t_{i}^{\ell},t_{i+1}^{\ell}\right]  \cap
A}\int_{B_{2R}\left(  x_{0}\right)  \setminus B_{\rho}\left(  x_{0}\right)
}d\mu_{t}\left(  x\right)  dt-\nonumber\\
&  -\frac{C}{\rho R}\int_{t_{i}^{\ell}}^{t_{i+1}^{\ell}}dt-\frac{1}{4\pi
\rho^{2}}\int_{\left[  t_{i}^{\ell},t_{i+1}^{\ell}\right]  \cap A}%
\int_{B_{\rho}\left(  x_{0}\right)  \times B_{\rho}\left(  x_{0}\right)
\cap\left\{  x\neq y\right\}  }d\mu_{t}\left(  x\right)  d\mu_{t}\left(
y\right)  dt \label{Z1E1a}%
\end{align}

Then since $B_{2R}\left(  x_{0}\right)  \setminus B_{\rho}\left(
x_{0}\right)  \subset B_{2\delta_{1}}\left(  Y\right)  \setminus\left\{
Y\right\}  $ due to (\ref{Z1E6}) we have:
\[
\int_{B_{2R}\left(  x_{0}\right)  \setminus B_{\rho}\left(  x_{0}\right)
}d\mu_{t}\left(  x\right)  \leq\int_{B_{\sigma}\left(  Y\right)  \setminus
S_{t}}d\mu_{t}\left(  x\right)  <\delta_{2}^{2}%
\]
due to the definition of $I_{\delta_{1},\delta_{2}}.$ Then:
\begin{equation}
\frac{C}{\rho^{2}}\int_{\left[  t_{i}^{\ell},t_{i+1}^{\ell}\right]  \cap
A}\int_{B_{2R}\left(  x_{0}\right)  \setminus B_{\rho}\left(  x_{0}\right)
}d\mu_{t}\left(  x\right)  dt\leq\frac{C\delta_{2}^{2}}{\rho^{2}}%
\frac{\varepsilon_{\ell}}{n} \label{Z1E2}%
\end{equation}

Moreover, since $t\in A\subset I_{\delta_{1},\delta_{2}}^{+}$ and $X=x_{0}$ we
have
\begin{equation}
\frac{U_{\rho}\left(  t;x_{0}\right)  }{4\pi}-2\int_{B_{\rho}\left(
x_{0}\right)  }d\mu_{t}\left(  x\right)  >\delta_{2} \label{Z1E3}%
\end{equation}

We finally estimate the last term in (\ref{Z1E1a}). Since $t\in A\subset
I_{\delta_{1},\delta_{2}}$ there is only one singular point $Y\in B_{\rho
}\left(  x_{0}\right)  .$ Then:
\begin{align*}
&  \int_{B_{\rho}\left(  x_{0}\right)  \times B_{\rho}\left(  x_{0}\right)
\cap\left\{  x\neq y\right\}  }d\mu_{t}\left(  x\right)  d\mu_{t}\left(
y\right) \\
&  \leq2\left[  \int_{B_{\rho}\left(  x_{0}\right)  }d\mu_{t}\left(  x\right)
\right]  \cdot\left[  \int_{B_{\rho}\left(  x_{0}\right)  \setminus\left\{
Y\right\}  }d\mu_{t}\left(  x\right)  \right]  \leq C\int_{B_{\sigma}\left(
Y\right)  \setminus\left\{  Y\right\}  }d\mu_{t}\left(  x\right)
\end{align*}
and using the definition of $I_{\delta_{1},\delta_{2}}$ we obtain:
\begin{equation}
\int_{B_{\rho}\left(  x_{0}\right)  \times B_{\rho}\left(  x_{0}\right)
\cap\left\{  x\neq y\right\}  }d\mu_{t}\left(  x\right)  d\mu_{t}\left(
y\right)  \leq C\delta_{2}^{2} \label{Z1E4}%
\end{equation}

Plugging (\ref{Z1E2}), (\ref{Z1E3}), (\ref{Z1E4}) into (\ref{Z1E1a}) we
obtain:
\begin{align*}
&  \alpha\left(  Y_{i+1}^{\ell};t_{i+1}^{\ell}\right)  -\alpha\left(
Y_{i}^{\ell};t_{i}^{\ell}\right)  +\left[  \int_{\Omega\setminus\left\{
Y_{i+1}^{\ell}\right\}  }\varphi_{\rho}\left(  x;Y_{i}^{\ell}\right)
d\mu_{t_{i+1}^{\ell}}\left(  x\right)  -\int_{\Omega\setminus\left\{
Y_{i}^{\ell}\right\}  }\varphi_{\rho}\left(  x;Y_{i}^{\ell}\right)
d\mu_{t_{i}^{\ell}}\left(  x\right)  \right] \\
&  \geq\frac{\delta_{2}}{\rho^{2}}\int_{\left[  t_{i}^{\ell},t_{i+1}^{\ell
}\right]  \cap A}dt-\frac{C}{\rho^{2}}\int_{\left[  t_{i}^{\ell},t_{i+1}%
^{\ell}\right]  \setminus A}dt-\frac{C}{\rho R}\int_{t_{i}^{\ell}}%
^{t_{i+1}^{\ell}}dt-\frac{C\delta_{2}^{2}}{\rho^{2}}\frac{\varepsilon_{\ell}%
}{n}%
\end{align*}

Adding for all $i=1,...,\left(  n-1\right)  :$%
\begin{align*}
&  \alpha\left(  Y_{n}^{\ell};t_{0}+\varepsilon_{\ell}\right)  -\alpha\left(
Y_{1}^{\ell};t_{0}\right)  +\sum_{i=1}^{n-1}\left[  \int_{\Omega
\setminus\left\{  Y_{i+1}^{\ell}\right\}  }\varphi_{\rho}\left(  x;Y_{i}%
^{\ell}\right)  d\mu_{t_{i+1}^{\ell}}\left(  x\right)  -\int_{\Omega
\setminus\left\{  Y_{i}^{\ell}\right\}  }\varphi_{\rho}\left(  x;Y_{i}^{\ell
}\right)  d\mu_{t_{i}^{\ell}}\left(  x\right)  \right] \\
&  \geq\frac{\delta_{2}}{\rho^{2}}\int_{\left[  t_{0},t_{0}+\varepsilon_{\ell
}\right]  \cap A}dt-\frac{C}{\rho^{2}}\int_{\left[  t_{0},t_{0}+\varepsilon
_{\ell}\right]  \setminus A}dt-\frac{C}{\rho R}\int_{t_{0}}^{t_{0}%
+\varepsilon_{\ell}}dt-\frac{C\delta_{2}^{2}}{\rho^{2}}\varepsilon_{\ell}\\
&  =\frac{\delta_{2}}{\rho^{2}}\int_{\left[  t_{0},t_{0}+\varepsilon_{\ell
}\right]  }dt-\frac{\left(  C+\delta_{2}\right)  }{\rho^{2}}\int_{\left[
t_{0},t_{0}+\varepsilon_{\ell}\right]  \setminus A}dt-\frac{C}{\rho R}%
\int_{t_{0}}^{t_{0}+\varepsilon_{\ell}}dt-\frac{C\delta_{2}^{2}}{\rho^{2}%
}\varepsilon_{\ell}%
\end{align*}

\bigskip We now (\ref{W4E7b}) to obtain:
\begin{align*}
&  \alpha\left(  Y_{n}^{\ell};t_{0}+\varepsilon_{\ell}\right)  -\alpha\left(
Y_{1}^{\ell};t_{0}\right)  +\sum_{i=1}^{n-1}\left[  \int_{\Omega
\setminus\left\{  Y_{i+1}^{\ell}\right\}  }\varphi_{\rho}\left(  x;Y_{i}%
^{\ell}\right)  d\mu_{t_{i+1}^{\ell}}\left(  x\right)  -\int_{\Omega
\setminus\left\{  Y_{i}^{\ell}\right\}  }\varphi_{\rho}\left(  x;Y_{i}^{\ell
}\right)  d\mu_{t_{i}^{\ell}}\left(  x\right)  \right]  \\
&  \geq\frac{\delta_{2}}{\rho^{2}}\varepsilon_{\ell}-\frac{C}{\rho^{2}}%
\frac{\delta_{2}}{K}\varepsilon_{\ell}-\frac{C}{\rho R}\varepsilon_{\ell
}-\frac{C\delta_{2}^{2}}{\rho^{2}}\varepsilon_{\ell}\\
&  =\frac{\delta_{2}\varepsilon_{\ell}}{\rho^{2}}\left[  1-\frac{2C}{K}%
-\frac{C\rho}{R}-C\delta_{2}\right]  \\
&  =\frac{\delta_{2}\varepsilon_{\ell}}{\rho^{2}}\left[  1-\frac{2C}{K}%
-\frac{C}{D\delta_{2}}-C\delta_{2}\right]
\end{align*}
where we have used (\ref{Z1E6a}). The constant $C$ depend only on
$\int_{\Omega}d\mu_{0}\left(  x\right)  .$ We now choose $\delta_{2}$ such
that $C\delta_{2}\leq\frac{1}{10},\;D$ such that $\frac{C}{D\delta_{2}}%
=\frac{1}{10},\;$and $K$ such that $\frac{2C}{K}\leq\frac{1}{10}.$ Notice that
$R=D\rho=4LD\sqrt{\delta_{2}\varepsilon_{\ell}}=\frac{40LC}{\sqrt{\delta_{2}}%
}\sqrt{\varepsilon_{\ell}}\rightarrow0$ as $\ell\rightarrow\infty.$ Finally we
choose $D$ satisfying (\ref{Z1E6}) and using the choice of $\rho$ in
(\ref{Z1E6a}) we obtain:
\[
\alpha\left(  Y_{n}^{\ell};t_{0}+\varepsilon_{\ell}\right)  -\alpha\left(
Y_{1}^{\ell};t_{0}\right)  \geq\frac{1}{2\left(  4L\right)  ^{2}}-\sum
_{i=1}^{n-1}\left[  \int_{\Omega\setminus\left\{  Y_{i+1}^{\ell}\right\}
}\varphi_{\rho}\left(  x;Y_{i}^{\ell}\right)  d\mu_{t_{i+1}^{\ell}}\left(
x\right)  -\int_{\Omega\setminus\left\{  Y_{i}^{\ell}\right\}  }\varphi_{\rho
}\left(  x;Y_{i}^{\ell}\right)  d\mu_{t_{i}^{\ell}}\left(  x\right)  \right]
\]

We estimate the sum as:
\begin{align*}
&  \sum_{i=1}^{n-1}\left[  \int_{\Omega\setminus\left\{  Y_{i+1}^{\ell
}\right\}  }\varphi_{\rho}\left(  x;Y_{i}^{\ell}\right)  d\mu_{t_{i+1}^{\ell}%
}\left(  x\right)  -\int_{\Omega\setminus\left\{  Y_{i}^{\ell}\right\}
}\varphi_{\rho}\left(  x;Y_{i}^{\ell}\right)  d\mu_{t_{i}^{\ell}}\left(
x\right)  \right] \\
&  \leq\sum_{i=1}^{n-1}\delta_{2}^{2}\leq\delta_{2}^{2}n\leq2\delta_{2}%
\end{align*}

Then:
\begin{equation}
\alpha\left(  Y_{n}^{\ell};t_{0}+\varepsilon_{\ell}\right)  -\alpha\left(
Y_{1}^{\ell};t_{0}\right)  \geq\frac{1}{32L^{2}}-2\delta_{2}\geq\frac
{1}{64L^{2}}\equiv\theta>0 \label{Z6E1}%
\end{equation}
where $\theta$ depends only on $\int_{\Omega}d\mu_{0}\left(  x\right)  $ and
$\Omega.$\texttt{\ }

We can now derive a contradiction as follows. Let us denote as $\mathcal{A}$
the set of density points of $A.$ More precisely:
\[
\mathcal{A}=\left\{  t\in A:\lim_{\varepsilon\rightarrow0}\frac{\left\vert
A\cap\left[  t,t+\varepsilon\right]  \right\vert }{\varepsilon}=1\right\}
\]
We now use that (cf. \cite{Stein}) $\left\vert A\setminus\mathcal{A}%
\right\vert =0.$ Therefore $\lim_{\varepsilon\rightarrow0}\frac{\left\vert
A\cap\left[  t,t+\varepsilon\right]  \right\vert }{\varepsilon}=\lim
_{\varepsilon\rightarrow0}\frac{\left\vert \mathcal{A}\cap\left[
t,t+\varepsilon\right]  \right\vert }{\varepsilon}$ and all the points of
$\mathcal{A}$\ are density points. By assumption $\left\vert A\right\vert
=\left\vert \mathcal{A}\right\vert >0.$ We have proved in (\ref{Z6E1}) the following.

For any $t_{0}\in\mathcal{A}$ there exists $\varepsilon\left(  t_{0}\right)  $
such that, for any $\varepsilon\leq\varepsilon\left(  t_{0}\right)  $ we
have:
\begin{equation}
\alpha\left(  Y\left(  t_{0}+\varepsilon\right)  ;t_{0}+\varepsilon\right)
-\alpha\left(  Y\left(  t_{0}\right)  ;t_{0}\right)  \geq\theta\label{Z6E2}%
\end{equation}
where $Y\left(  t\right)  $ is the unique point in $S_{t}\cap B_{\delta_{1}%
}\left(  X\right)  $ that exists for any $t\in A=I_{\delta_{1},\delta_{2}}\cap
I_{\delta_{1},\delta_{2}}^{+}\left(  X\right)  .$

Moreover:
\begin{equation}
\left\vert \mathcal{A}\cap\left[  t_{0},t_{0}+\varepsilon\right]  \right\vert
\geq\left(  1-\frac{\delta_{2}}{K}\right)  \varepsilon\label{Z6E3}%
\end{equation}
for any $\varepsilon\leq\varepsilon\left(  t_{0}\right)  .$

We now argue iteratively. Due to (\ref{Z6E3}) we can find $t_{1}\in
\mathcal{A}\cap\left(  t_{0},t_{0}+\varepsilon\left(  t_{0}\right)  \right)
,$ and there exists $\varepsilon\left(  t_{1}\right)  \leq\left(
t_{0}+\varepsilon\left(  t_{0}\right)  -t_{1}\right)  $ such that for any
$\varepsilon\leq\varepsilon\left(  t_{1}\right)  $ we have:
\begin{equation}
\alpha\left(  Y\left(  t_{1}+\varepsilon\right)  ;t_{1}+\varepsilon\right)
-\alpha\left(  Y\left(  t_{1}\right)  ;t_{1}\right)  \geq\theta\label{Z6E4}%
\end{equation}%
\begin{equation}
\left\vert \mathcal{A}\cap\left[  t_{1},t_{1}+\varepsilon\right]  \right\vert
\geq\left(  1-\frac{\delta_{2}}{K}\right)  \varepsilon\label{Z6E5}%
\end{equation}

Taking $\varepsilon=t_{1}-t_{0}$ in (\ref{Z6E2}) and using also (\ref{Z6E4})
we obtain:
\[
\alpha\left(  Y\left(  t_{1}+\varepsilon\right)  ;t_{1}+\varepsilon\right)
-\alpha\left(  Y\left(  t_{0}\right)  ;t_{0}\right)  \geq2\theta
\]
for any $\varepsilon\leq\varepsilon\left(  t_{1}\right)  .$

Iterating the argument, something that it is possible due to (\ref{Z6E5}) we
obtain the existence of sequences $\left\{  t_{n}\right\}  \subset
\mathcal{A,\;}\left\{  \varepsilon\left(  t_{n}\right)  \right\}
\subset\left[  0,\infty\right)  $ such that:
\[
\alpha\left(  Y\left(  t_{n}+\varepsilon\right)  ;t_{n}+\varepsilon\right)
-\alpha\left(  Y\left(  t_{0}\right)  ;t_{0}\right)  \geq\left(  n+1\right)
\theta
\]
for any $\varepsilon\leq\varepsilon\left(  t_{n}\right)  .$ Since
$\alpha\left(  Y;t\right)  $ is bounded for the total mass $\int_{\Omega}%
d\mu_{0}\left(  x\right)  $ this gives a contradiction.
\end{proof}

Using (\ref{W4E2}), (\ref{W4E2bis}) we can now conclude the Proof of Theorem
\ref{RepMeas}.

\begin{proof}
[Proof of Theorem \ref{RepMeas}]Notice that for any $\delta_{2}>0$ fixed the
sets $I_{\delta_{1},\delta_{2}}$ are an decreasing sequence of sets in the
sense that:
\[
0<\bar{\delta}_{1}<\delta_{1}\;\;\text{implies\ \ }I_{\delta_{1},\delta_{2}%
}\subset I_{\bar{\delta}_{1},\delta_{2}}%
\]

Moreover, for any $\delta_{2}>0$ fixed we have:%
\begin{equation}
\left\vert \left[  0,\infty\right)  \setminus\bigcup_{\delta_{1}>0}^{\infty
}I_{\delta_{1},\delta_{2}}\right\vert =0 \label{W4E2a}%
\end{equation}

Let us write:
\[
\mathcal{Z}=\left[  0,\infty\right)  \setminus\bigcup_{\delta_{1}>0}^{\infty
}I_{\delta_{1},\delta_{2}}%
\]

Then:
\begin{align}
\left[  0,\infty\right)   &  =\mathcal{Z}\cup\bigcup_{\delta_{1}>0}^{\infty
}I_{\delta_{1},\delta_{2}}=\mathcal{Z}\cup\bigcup_{n=1}^{\infty}I_{\frac{1}%
{n},\delta_{2}}\label{Z7E1}\\
\left\vert \mathcal{Z}\right\vert  &  =0\nonumber
\end{align}

Let us consider a countable set $\mathcal{F}$ dense in $\Omega.$ We have:
\begin{equation}
\Omega=\bigcup_{X\in\mathcal{F}}B_{\delta_{1}}\left(  X\right)  \label{Z7E3}%
\end{equation}

We define
\begin{equation}
\mathcal{U}_{\delta_{2}}=\left\{  t\in\left[  0,\infty\right)  :\exists\;Y\in
S_{t},\;\gamma^{-}\left(  Y;t\right)  -8\pi\alpha\left(  Y;t\right)  \geq
8\pi\delta_{2}\right\}  \label{Z7E2}%
\end{equation}

Using (\ref{Z7E1}) we obtain:
\begin{equation}
\mathcal{U}_{\delta_{2}}=\left[  \mathcal{Z}\cap\mathcal{U}_{\delta_{2}%
}\right]  \cup\bigcup_{n=1}^{\infty}\left[  I_{\frac{1}{n},\delta_{2}}%
\cap\mathcal{U}_{\delta_{2}}\right]  \label{Z7E4}%
\end{equation}

Suppose that $t\in\left[  I_{\frac{1}{n},\delta_{2}}\cap\mathcal{U}%
_{\delta_{2}}\right]  .$ Then, there exists $Y\in S_{t}$ such that $\gamma
^{-}\left(  Y,t\right)  -8\pi\alpha\left(  Y,t\right)  \geq8\pi\delta_{2}$
and, due to (\ref{Z7E3}) there exists $\tilde{X}\in\mathcal{F}$ such that
$Y\in B_{\frac{1}{n}}\left(  \tilde{X}\right)  .$ Then:
\[
\int_{B_{\frac{1}{n}}\left(  \tilde{X}\right)  \times S^{1}}d\hat{\mu}%
_{t}\left(  x,\nu\right)  -8\pi\int_{B_{\frac{1}{n}}\left(  \tilde{X}\right)
}d\mu_{t}\geq8\pi\delta_{2}-8\pi\delta_{2}^{2}\geq4\pi\delta_{2}%
\]

Therefore $t\in I_{\frac{1}{n},\delta_{2}}^{+}\left(  \tilde{X}\right)
\subset\bigcup_{X\in\mathcal{F}}I_{\frac{1}{n},\delta_{2}}^{+}\left(
X\right)  .$ Then:
\[
\left[  I_{\frac{1}{n},\delta_{2}}\cap\mathcal{U}_{\delta_{2}}\right]
\subset\bigcup_{X\in\mathcal{F}}\left[  I_{\frac{1}{n},\delta_{2}}\cap
I_{\frac{1}{n},\delta_{2}}^{+}\left(  X\right)  \right]
\]

It then follows from (\ref{Z7E4}) that:
\[
\mathcal{U}_{\delta_{2}}\subset\left[  \mathcal{Z}\cap\mathcal{U}_{\delta_{2}%
}\right]  \cup\bigcup_{n=1}^{\infty}\bigcup_{X\in\mathcal{F}}\left[
I_{\frac{1}{n},\delta_{2}}\cap I_{\frac{1}{n},\delta_{2}}^{+}\left(  X\right)
\right]
\]

Then:
\[
\left\vert \mathcal{U}_{\delta_{2}}\right\vert \leq\left\vert \mathcal{Z}%
\cap\mathcal{U}_{\delta_{2}}\right\vert +\sum_{n=1}^{\infty}\sum
_{X\in\mathcal{F}}\left\vert I_{\frac{1}{n},\delta_{2}}\cap I_{\frac{1}%
{n},\delta_{2}}^{+}\left(  X\right)  \right\vert =0+\sum_{n=1}^{\infty}%
\sum_{X\in\mathcal{F}}0=0
\]
whence:
\[
\left\vert \mathcal{U}_{\delta_{2}}\right\vert =0
\]
for any $\delta_{2}>0$ sufficiently small$.$ Due to the definition of
$\mathcal{U}_{\delta_{2}}$ in (\ref{Z7E2}) it follows that:
\[
\forall\;Y\in S_{t},\;\frac{\gamma^{-}\left(  Y,t\right)  }{8\pi}%
-\alpha\left(  Y,t\right)  \leq\delta_{2}\;\;\;a.e.\;t\in\left[
0,\infty\right)
\]

Then, since $\delta_{2}$ can be made arbitrarily small it follows that:
\[
\forall\;Y\in S_{t},\;\;\frac{\gamma^{-}\left(  Y,t\right)  }{8\pi}%
-\alpha\left(  Y,t\right)  \leq0\;\;\;a.e.\;t\in\left[  0,\infty\right)
\]

A similar argument taking as starting point (\ref{W4E2bis}) yields:
\[
\forall\;Y\in S_{t},\;\;\frac{\gamma^{-}\left(  Y,t\right)  }{8\pi}%
-\alpha\left(  Y,t\right)  \geq0\;\;\;a.e.\;t\in\left[  0,\infty\right)
\]
whence:
\[
\forall\;Y\in S_{t},\;\;\frac{\gamma^{-}\left(  Y,t\right)  }{8\pi}%
=\alpha\left(  Y,t\right)  \;\;\;a.e.\;t\in\left[  0,\infty\right)
\]
or, equivalently:
\[
\frac{1}{8\pi}\int_{S^{1}}d\hat{\mu}_{t}^{-}\left(  \cdot,\nu\right)
=d\mu_{t}^{\operatorname{sing}}\left(  \cdot\right)
\]
\bigskip

The previous argument yields the contribution of the interior singular points.
A similar argument could be use to compute the contribution of the boundary
terms. The main idea needed is now sketched. Taking a test function that is
quadratic near a singular boundary point (with Neumann boundary conditions),
and using the test function $\varphi$ in (\ref{M4E1}) we obtain formally that
the main contribution is due to the terms containing $\nabla^{2}\psi$ that
give, assuming that curvature effects in the test function are higher order
terms (as well as the curvature term in $\varphi$ that seems to give also a
negligible contribution):%
\begin{align*}
\psi &  \approx-\frac{\left\vert x-x_{0}\right\vert ^{2}}{2\rho^{2}%
}\ \ ,\ \ \Delta\psi\approx-\frac{2}{\rho^{2}}\ \ \\
&  \frac{Y\cdot\nabla^{2}\psi\left(  y,t\right)  \cdot Y}{4\pi}+\left[
\frac{\nu\left(  y\right)  \cdot\nabla^{2}\psi\left(  y,t\right)  \cdot
\nu\left(  y\right)  }{4\pi}\right]  \left(  \lambda_{1}+\lambda_{2}\right)
^{2}\\
&  \approx-\frac{1}{4\pi\rho^{2}}\left[  Y^{2}+\left(  \lambda_{1}+\lambda
_{2}\right)  ^{2}\right]  =-\frac{1}{4\pi\rho^{2}}%
\end{align*}

This gives the boundary contributions.
\end{proof}

\bigskip

\section{Characterizing oscillations in the microscopic scale: Measure valued
Young measures.\label{Young}}

\bigskip

\subsection{Generalities.}

A possible feature of the solutions obtained in this paper that we do not rule
out in this paper is the possibility of having oscillations at a microscopic
time scale of order $\varepsilon^{2}$ for functions like $f_{\varepsilon
}\left(  u^{\varepsilon}\right)  $ in the case of the first regularization or
$u^{\varepsilon}+\varepsilon\left(  u^{\varepsilon}\right)  ^{\frac{7}{6}}.$
This is the reason because we obtained in the weak limits of $f_{\varepsilon
}\left(  u^{\varepsilon}\left(  x,t\right)  \right)  f_{\varepsilon}\left(
u^{\varepsilon}\left(  y,t\right)  \right)  dxdy$ measures defined in
$\bar{\Omega}\times\bar{\Omega}\times\left[  0,\infty\right)  $ having the
form $d\omega_{t}^{-}\left(  x,y\right)  .$ It is not clear if this limit
measure can be decomposed as $d\mu_{t}^{-}\left(  x\right)  d\mu_{t}%
^{-}\left(  y\right)  .$ In this section we introduce some general formalism
that allows to characterize the limits of this nonlinear expressions even if
such oscillations take place. We will also prove that the limit objects, that
will be denoted as Measure valued Young measures, can be characterized by
means of a standard set of Young measures that depend only on the oscillations
of the masses near the singular points.

\bigskip

Let us assume that $\left\{  \mu_{1,}\;\mu_{2},...,\;\mu_{L}\right\}  $ is a
set of measures in $M_{+}\left(  \bar{\Omega}\times\mathbb{R}^{+}\right)  $
satisfying:
\begin{align}
d\mu_{k}\left(  x,t\right)   &  =d\mu_{k,t}\left(  x\right)
dt\;\;,\;\;k=1,...,L\label{P1E1}\\
\int_{\bar{\Omega}}d\mu_{k,t}\left(  x\right)   &  \leq
A\;\;,\;\;k=1,...,L\;\;,\;\;a.e\;t\in\left[  0,\infty\right)  \label{P1E2}%
\end{align}
for some $A>0.$

We will assume also that these measures can be approximated in the weak
topology by means of sequences $\left\{  d\mu_{k,t}^{\varepsilon}\left(
x\right)  dt\right\}  $as $\varepsilon\rightarrow0^{+},$ where $d\mu
_{k,t}^{\varepsilon}\left(  x\right)  =U_{\varepsilon}\left(  x,t\right)  dx,$
and $U_{\varepsilon}\in C^{\infty}\left(  \bar{\Omega}\times\mathbb{R}%
^{+}\right)  .$ It will be always understood that convergence takes place for
suitable subsequences.

Specific examples would be the sequences $\left\{  f_{\varepsilon}\left(
u^{\varepsilon}\right)  dxdt\right\}  ,\;\left\{  \left[  u^{\varepsilon
}+\varepsilon\left(  u^{\varepsilon}\right)  ^{\frac{7}{6}}\right]
dxdt\right\}  ,$ that converge respectively to $d\mu_{t}^{-}dt,\;d\mu_{t}%
^{+}dt.$ We could also consider sequences like $\left\{  u^{\varepsilon
}dxdt\right\}  $ that converge to $d\mu_{t}dt,$ but since in this case
oscillations do not take place, the formalism presented below would be trivial
and the resulting measure valued Young measures would be suitable Dirac masses.

\bigskip

Given measures $\left\{  \mu_{k}\right\}  _{k=1}^{L}$ satisfying (\ref{P1E1}),
(\ref{P1E2}), test functions $\varphi_{k,j}\in C\left(  \bar{\Omega}\right)
,\;k=1,...,L,\;j=1,...,M_{k},\;M_{k}\geq1$ as well as $T\in\left(
0,\infty\right)  $ we define the following functional:
\begin{align*}
L_{\left\{  \varphi_{k,j}\right\}  }^{\varepsilon}  &  :C_{0}\left(
\mathbb{R}^{M}\times\left[  0,T\right]  \right)  \rightarrow\mathbb{R\;\;}%
,\mathbb{\;\;}M=\sum_{j=1}^{L}M_{j}\\
\Phi &  \rightarrow L_{\left\{  \varphi_{k,j}\right\}  }^{\varepsilon}\left[
\Phi\right]  \ \ ,\ \ \Phi\in C\left(  \left[  -B,B\right]  ^{M}\times\left[
0,T\right]  \right)
\end{align*}
where:
\begin{equation}
L_{\left\{  \varphi_{k,j}\right\}  ,T}^{\varepsilon}\left[  \Phi\right]
=\int_{0}^{T}\Phi\left(  \int_{\Omega}\varphi_{1,1}d\mu_{1}^{\varepsilon
},...,\int_{\Omega}\varphi_{1,M_{1}}d\mu_{1}^{\varepsilon},\int_{\Omega
}\varphi_{2,1}d\mu_{2}^{\varepsilon},...,\int_{\Omega}\varphi_{2,M_{2}}%
d\mu_{2}^{\varepsilon},...,\int_{\Omega}\varphi_{L,M_{L}}d\mu_{L}%
^{\varepsilon},t\right)  dt \label{P1E3}%
\end{equation}

Due to (\ref{P1E1}), (\ref{P1E2}) it follows that $\left|  \int_{\Omega
}\varphi_{k,j}d\mu_{k}^{\varepsilon}\right|  \leq
B,\;k=1,...,L,\;j=1,...,M_{k}$ for some $B$ depending only on $A,\;\left\|
\varphi_{k,j}\right\|  _{L^{\infty}\left(  \bar{\Omega}\right)  }.$ Therefore
the functional defined in (\ref{P1E3}) can be considered as a functional in
$C\left(  \left[  -B,B\right]  ^{M}\times\left[  0,T\right]  \right)  $ and it
satisfies:
\[
\left|  L_{\left\{  \varphi_{k,j}\right\}  ,T}^{\varepsilon}\left[
\Phi\right]  \right|  \leq T\left\|  \Phi\right\|  _{C\left(  \left[
-B,B\right]  ^{M}\times\left[  0,T\right]  \right)  }%
\]

Therefore, there exists a Radon measure $d\lambda_{\left\{  \varphi
_{k,j}\right\}  ,T}^{\varepsilon}$ such that:
\[
L_{\left\{  \varphi_{k,j}\right\}  ,T}^{\varepsilon}\left[  \Phi\right]
=\int_{\left[  -B,B\right]  ^{M}\times\left[  0,T\right]  }\Phi\left(
\xi,t\right)  d\lambda_{\left\{  \varphi_{k,j}\right\}  ,T}^{\varepsilon
}\left(  \xi,t\right)
\]
where:
\[
\xi=\left(  \xi_{1,1},...,\xi_{1,M_{1}},\xi_{2,1},...,\xi_{2,M_{2}}%
,...\xi_{L,M_{L}}\right)
\]

The compactness of $\left[  -B,B\right]  ^{M}\times\left[  0,T\right]  $
implies that for suitable subsequences:
\[
d\lambda_{\left\{  \varphi_{k,j}\right\}  ,T}^{\varepsilon}\rightharpoonup
d\lambda_{\left\{  \varphi_{k,j}\right\}  ,T}%
\]

Therefore (for subsequences):
\[
L_{\left\{  \varphi_{k,j}\right\}  ,T}^{\varepsilon}\left[  \Phi\right]
\rightarrow\int_{\left[  -B,B\right]  ^{M}\times\left[  0,T\right]  }%
\Phi\left(  \xi,t\right)  d\lambda_{\left\{  \varphi_{k,j}\right\}  ,T}\left(
\xi,t\right)
\]

Extending the measure $d\lambda_{\left\{  \varphi_{k,j}\right\}  ,T}$ by zero
for $\left\vert \xi_{k,j}\right\vert >B$ we can ensure the existence of a
family of measures $d\lambda_{\left\{  \varphi_{k,j}\right\}  ,T}$ such that:
\begin{equation}
L_{\left\{  \varphi_{k,j}\right\}  ,T}^{\varepsilon}\left[  \Phi\right]
\rightarrow\int_{\mathbb{R}^{M}\times\left[  0,T\right]  }\Phi\left(
\xi,t\right)  d\lambda_{\left\{  \varphi_{k,j}\right\}  ,T}\left(
\xi,t\right)  \label{Z8E3}%
\end{equation}

The family of measures $\left\{  \lambda_{\left\{  \varphi_{k,j}\right\}
,T}\right\}  $ will be denoted as measure valued Young measures. Notice that
they allow to compute weak limits for the weak limit of any finite sequence of
"macroscopic" magnitudes obtained using sequences of measures $d\mu
_{k}^{\varepsilon}.$

Our goal is to reduce the computation of the measures $d\lambda_{\left\{
\varphi_{k,j}\right\}  ,T}$ in the case of the limits of sequences $\left\{
f_{\varepsilon}\left(  u^{\varepsilon}\right)  dxdt\right\}  ,\;\left\{
\left[  u^{\varepsilon}+\varepsilon\left(  u^{\varepsilon}\right)  ^{\frac
{7}{6}}\right]  dxdt\right\}  $ to the Young measures associated to the
sequences that yield the masses near the singular set $S_{t}.$

\bigskip

\subsection{On the family of Young measures describing the aggregation of
$f_{\varepsilon}\left(  u^{\varepsilon}\right)  ,\ u^{\varepsilon}%
+\varepsilon\left(  u^{\varepsilon}\right)  ^{\frac{7}{6}}.$}

\bigskip

We now consider the first regularization, and we define a family of Young
measures labelled by the spatial and time positions and describing the weak
limits of the weak limits associated to $\left\{  f_{\varepsilon}\left(
u^{\varepsilon}\right)  \right\}  ,\ \left\{  u^{\varepsilon}+\varepsilon
\left(  u^{\varepsilon}\right)  ^{\frac{7}{6}}\right\}  .$ We will include
also the dependence on the sequence $\left\{  u^{\varepsilon}\right\}  $ that
will yield a trivial Dirac mass contribution, but we will include it by completedness.

\bigskip

In order to obtain a general result describing the possible limits of
nonlinear functionals associated to the sequences $\left\{  f_{\varepsilon
}\left(  u^{\varepsilon}\right)  dxdt\right\}  $ we will need detailed
information on the behaviour of these sequences near the singular set.

\bigskip

The singular set varies in a continuous manner due to Lemma \ref{contS}. We
consider the set of continuous functions on $S,$ $C\left(  S\right)  .$ Since
$S$ is a closed set in $\overline{\Omega}\times\left[  0,\infty\right)  ,$
given $\psi\in C\left(  S\right)  $ it is possible to find a continuous
extension of $\psi$ to $\overline{\Omega}\times\left[  0,\infty\right)  .$ We
will denote as $\tilde{\psi}\in C\left(  \overline{\Omega}\times\mathbb{R}%
^{+}\right)  $ a generic continuous extension of $\psi.$ Our goal is to define
a family of measures that will describe the structure of oscillations of the
sequence $\left\{  f_{\varepsilon}\left(  u^{\varepsilon}\right)
dxdt\right\}  $ near $S.$

For each fixed $\delta>0,\ \varepsilon>0$, $L>0$ and any family of functions
$\psi_{1},...,\psi_{N},\varphi_{1},...,\varphi_{N}\in C\left(  S\right)  ,$ we
consider any family of continuous extensions of these functions $\tilde{\psi
}_{1},...,\tilde{\psi}_{N},\tilde{\varphi}_{1},...,\tilde{\varphi}_{N}\in
C\left(  \overline{\Omega}\times\mathbb{R}^{+}\right)  $ supported in the set
$S+B_{\sigma}\left(  0\right)  ,$ with $\sigma>0$ small. We then define a
family of functionals in $C\left(  \left(  \left[  0,B\right]  \right)
^{2N}\times\left[  0,T\right]  \right)  $ by means of:
\begin{align}
C\left(  \left(  \overline{\mathbb{R}}^{+}\right)  ^{2N}\right)   &
\rightarrow\mathbb{R}\nonumber\\
\phi &  \rightarrow M_{\left\{  \tilde{\psi}_{m}\right\}  _{m=1}^{N}%
}^{\varepsilon,\sigma}\left[  \phi\right]  =\label{Z8E2}\\
&  =\int_{0}^{T}\phi\left(  \int_{\Omega}\tilde{\psi}_{1}f_{\varepsilon
}\left(  u^{\varepsilon}\right)  dx,...,\int_{\Omega}\tilde{\psi}%
_{N}f_{\varepsilon}\left(  u^{\varepsilon}\right)  dx,\int_{\Omega}%
\tilde{\varphi}_{1}u^{\varepsilon}dx,...,\int_{\Omega}\tilde{\varphi}%
_{N}u^{\varepsilon}dx,t\right)  dt\nonumber
\end{align}

\begin{equation}
\left\vert M_{\left\{  \tilde{\psi}_{m},\tilde{\varphi}_{m}\right\}
_{m=1}^{N}}^{\varepsilon,\sigma}\left[  \phi\right]  \right\vert \leq
T\left\Vert \phi\right\Vert _{C\left(  \left[  0,B\right]  ^{M}\times\left[
0,T\right]  \right)  } \label{Z8E1}%
\end{equation}

Therefore, there exist Radon measures $\nu_{\left\{  \tilde{\psi}_{m}%
,\tilde{\varphi}_{m}\right\}  _{m=1}^{N}}^{\varepsilon,\sigma}\in M^{+}\left(
\left[  0,B\right]  ^{2N}\times\left[  0,T\right]  \right)  $ such that:
\[
M_{\left\{  \tilde{\psi}_{m},\tilde{\varphi}_{m}\right\}  _{m=1}^{N}%
}^{\varepsilon,\sigma}\left[  \phi\right]  =\int_{\left[  0,B\right]
^{2N}\times\left[  0,T\right]  }\phi d\nu_{\left\{  \tilde{\psi}_{m}%
,\tilde{\varphi}_{m}\right\}  _{m=1}^{N}}^{\varepsilon,\sigma}%
\]

The weak compactness of the sequence $\left\{  \nu_{\left\{  \tilde{\psi}%
_{m},\tilde{\varphi}_{m}\right\}  _{m=1}^{N}}^{\varepsilon,\sigma}\right\}  $
implies the existence of measures $\left\{  \nu_{\left\{  \tilde{\psi}%
_{m},\tilde{\varphi}_{m}\right\}  _{m=1}^{N}}^{\sigma}\right\}  ,$
$\nu_{\left\{  \tilde{\psi}_{m},\tilde{\varphi}_{m}\right\}  _{m=1}^{N}}$ such
that:
\begin{align*}
\nu_{\left\{  \tilde{\psi}_{m},\tilde{\varphi}_{m}\right\}  _{m=1}^{N}%
}^{\varepsilon,\sigma}  &  \rightharpoonup\nu_{\left\{  \tilde{\psi}%
_{m},\tilde{\varphi}_{m}\right\}  _{m=1}^{N}}^{\sigma}\ \ \text{as\ \ }%
\varepsilon\rightarrow0\\
\nu_{\left\{  \tilde{\psi}_{m},\tilde{\varphi}_{m}\right\}  _{m=1}^{N}%
}^{\sigma}  &  \rightharpoonup\nu_{\left\{  \tilde{\psi}_{m},\tilde{\varphi
}_{m}\right\}  _{m=1}^{N}}\ \ \text{as\ \ }\sigma\rightarrow0
\end{align*}
for suitable subsequences. Then:
\begin{equation}
M_{\left\{  \tilde{\psi}_{m},\tilde{\varphi}_{m}\right\}  _{m=1}^{N}%
}^{\varepsilon,\sigma}\left[  \phi\right]  \rightarrow\int_{\left[
0,B\right]  ^{2N}\times\left[  0,T\right]  }\phi d\nu_{\left\{  \tilde{\psi
}_{m},\tilde{\varphi}_{m}\right\}  _{m=1}^{N}} \label{Z8E4}%
\end{equation}

We now remark that the measures $\nu_{\left\{  \tilde{\psi}_{m},\tilde
{\varphi}_{m}\right\}  _{m=1}^{N}}$ depend only on the values of the functions
$\left\{  \tilde{\psi}_{m},\tilde{\varphi}_{m}\right\}  _{m=1}^{N}$ at the
singular set. Indeed, for any extension of the functions $\left\{  \psi
_{m},\varphi_{m}\right\}  _{m=1}^{N}$ to $S+B_{\sigma}\left(  0\right)  $ we
can choose $\sigma_{0}$ small such that the functions $\left\{  \tilde{\psi
}_{m},\tilde{\varphi}_{m}\right\}  _{m=1}^{N}$ differ from the values of
$\left\{  \psi_{m},\varphi_{m}\right\}  _{m=1}^{N}$ in the closest point in
$S$ by an arbitrarily small amount. On the other hand we have the estimates:
\begin{align*}
\int_{\Omega\setminus\left[  S+B_{\sigma_{0}}\left(  0\right)  \right]
}\tilde{\psi}_{m}f_{\varepsilon}\left(  u^{\varepsilon}\right)  dx  &  \leq
C\int_{\left[  S+B_{\sigma}\left(  0\right)  \right]  \setminus\left[
S+B_{\sigma_{0}}\left(  0\right)  \right]  }u^{\varepsilon}dx\\
\int_{\Omega\setminus\left[  S+B_{\sigma_{0}}\left(  0\right)  \right]
}\tilde{\varphi}_{m}u^{\varepsilon}dx  &  \leq C\int_{\left[  S+B_{\sigma
}\left(  0\right)  \right]  \setminus\left[  S+B_{\sigma_{0}}\left(  0\right)
\right]  }u^{\varepsilon}dx
\end{align*}
and the right hand side of these formulas can be made arbitrarily small,
uniformly on $\varepsilon,$ if $\sigma$ is small for $a.e.$ $t$. It then
follows that taking the limit in the order indicated above we obtain:
\begin{equation}
\nu_{\left\{  \psi_{m},\varphi_{m}\right\}  _{m=1}^{N}} \label{Z8E5}%
\end{equation}

The measures $\left\{  \nu_{\left\{  \psi_{m},\varphi_{m}\right\}  _{m=1}^{N}%
}\right\}  $ that can be extended to describe the possible oscillations of the
sequences $\left\{  f_{\varepsilon}\left(  u^{\varepsilon}\right)  \right\}
,\ \left\{  u^{\varepsilon}\right\}  $ near the singular set. Since nontrivial
oscillations can take place only at the singular set, they are sufficient to
describe the measures $\left\{  \lambda_{\left\{  \varphi_{k,j}\right\}
,T}\right\}  $ described above. Moreover, since the functions $\int_{\Omega
}\tilde{\varphi}_{m}u^{\varepsilon}dx$ change continuously in the macroscopic
scale, they do not contribute to the oscillations. Therefore, there exist
measures $\mathcal{V}_{_{\left\{  \psi_{m}\right\}  _{m=1}^{N}}}\in\left(
C\left(  \mathbb{R}_{+}^{N}\times\mathbb{R}_{+}\right)  \right)  ^{\ast}$ such
that:
\begin{equation}
d\nu_{\left\{  \psi_{m},\varphi_{m}\right\}  _{m=1}^{N}}\left(  \xi
_{1},...,\xi_{N},\eta_{1},...,\eta_{N},t\right)  =d\mathcal{V}_{_{\left\{
\psi_{m}\right\}  _{m=1}^{N}}}\left(  \xi_{1},...,\xi_{N},t\right)
{\displaystyle\prod\limits_{m=1}^{N}}
\delta_{\int_{S_{t}}\tilde{\varphi}_{m}d\mu_{t}}\left(  \eta_{m}\right)
\label{S8E6}%
\end{equation}

Therefore, the only non trivial measures that need to be computed are
$\left\{  \mathcal{V}_{_{\left\{  \psi_{m}\right\}  _{m=1}^{N}}}\right\}  .$
Intuitively these measures describe the statistical distribution of the
oscillations as well as their correlations at different points of the singular set.

\bigskip

Notice that, as mentioned above, it is possible to compute the measures
$\left\{  \lambda_{\left\{  \varphi_{k,j}\right\}  ,T}\right\}  $ in
(\ref{Z8E3}) in terms of the simplest family of measures $\left\{
\mathcal{V}_{_{\left\{  \psi_{m}\right\}  _{m=1}^{N}}}\right\}  .$ More
precisely, suppose that we define a family of measures $\left\{
\lambda_{\left\{  \varphi_{k,j}\right\}  ,T}\right\}  $ using as measures
$\left\{  \mu_{m}^{\varepsilon}\right\}  $ the sequences $\left\{
f_{\varepsilon}\left(  u^{\varepsilon}\right)  \right\}  ,\ \left\{
u^{\varepsilon}\right\}  .$ Suppose that we define measures $\left\{
\lambda_{\left\{  \varphi_{k,j}\right\}  ,T}\right\}  $ by means of:
\[
L_{\left\{  \varphi_{k,j}\right\}  ,T}^{\varepsilon}\left[  \Phi\right]
=\int_{0}^{T}\Phi\left(  \int_{\Omega}\varphi_{1,1}f_{\varepsilon}\left(
u^{\varepsilon}\right)  dx,...,\int_{\Omega}\varphi_{1,M_{1}}f_{\varepsilon
}\left(  u^{\varepsilon}\right)  dx,\int_{\Omega}\varphi_{2,1}u^{\varepsilon
}dx,...,\int_{\Omega}\varphi_{2,M_{2}}u^{\varepsilon}dx,t\right)  dt
\]
and limits:
\begin{align*}
L_{\left\{  \varphi_{k,j}\right\}  ,T}^{\varepsilon}\left[  \Phi\right]   &
\rightarrow L_{\left\{  \varphi_{k,j}\right\}  ,T}^{\varepsilon}\left[
\Phi\right] \\
&  \rightarrow\int_{0}^{T}\int\Phi\left(  \sigma_{1,1},...\sigma_{1,M_{1}%
},\sigma_{2,1},...\sigma_{2,M_{2}},t\right)  d\lambda_{\left\{  \varphi
_{k,j}\right\}  ,T}\left(  \sigma_{1,1},...\sigma_{1,M_{1}},\sigma
_{2,1},...,\sigma_{2,M_{2}},t\right)
\end{align*}
as $\varepsilon\rightarrow0.$

Given a function $\Phi\in C^{M_{1}+M_{2}+1}\left(  \mathbb{R}^{M_{1}}%
\times\mathbb{R}^{M_{2}}\times\mathbb{R}_{+}\right)  $ and real numbers
$\left\{  \alpha_{1,j}\right\}  _{j=1}^{M_{1}}$ we define:
\[
\mathcal{T}_{\left\{  \alpha_{1,j}\right\}  _{j=1}^{M_{1}}}\left(
\Phi\right)  \left(  \sigma_{1,1},...\sigma_{1,M_{1}},\sigma_{2,1}%
,...\sigma_{2,M_{2}},t\right)  =\Phi\left(  \sigma_{1,1}+\alpha_{1,1}%
,...\sigma_{1,M_{1}}+\alpha_{1,M_{1}},\sigma_{2,1},...\sigma_{2,M_{2}%
},t\right)
\]

Using the convergence properties of the sequences $\left\{  u^{\varepsilon
}\right\}  ,$ $\left\{  f_{\varepsilon}\left(  u^{\varepsilon}\right)
\right\}  $ outside the singular set, as well as the definition of the
measures $\left\{  \mathcal{V}_{_{\left\{  \psi_{m}\right\}  _{m=1}^{N}}%
}\right\}  $ we obtain the following representation formula:
\begin{align*}
&  \int_{0}^{T}\int\Phi\left(  \sigma_{1,1},...\sigma_{1,M_{1}},\sigma
_{2,1},...\sigma_{2,M_{2}},t\right)  d\lambda_{\left\{  \varphi_{k,j}\right\}
,T}\left(  \sigma_{1,1},...\sigma_{1,M_{1}},\sigma_{2,1},...,\sigma_{2,M_{2}%
},t\right) \\
&  =\int_{0}^{T}\int\mathcal{T}_{\left\{  \int_{\Omega\setminus S_{t}}%
\varphi_{1,j}udx\right\}  _{j=1}^{M_{1}}}\left(  \Phi\right)  \cdot\\
&  \cdot\left(  \sigma_{1,1},...,\sigma_{1,M_{1}},\int_{\Omega}\varphi
_{2,1}d\mu_{t},...,\int_{\Omega}\varphi_{2,M_{2}}d\mu_{t},t\right)
d\mathcal{V}_{_{\left\{  \psi_{m}\right\}  _{m=1}^{N}}}\left(  \sigma
_{1,1},...,\sigma_{1,M_{1}},t\right)
\end{align*}
This formula provides the desired representation formula for nonlinear
functions of limits of sequences $\left\{  f_{\varepsilon}\left(
u^{\varepsilon}\right)  \right\}  ,\;\left\{  u^{\varepsilon}+\varepsilon
\left(  u^{\varepsilon}\right)  ^{\frac{7}{6}}\right\}  ,\ \left\{
u^{\varepsilon}\right\}  $ in terms of the measures $\mathcal{V}_{_{\left\{
\psi_{m}\right\}  _{m=1}^{N}}}.$

\end{document}